\documentclass[a4paper,12pt]{amsart}

\usepackage{graphics,graphicx}
\usepackage{tikz-cd}
\usetikzlibrary{calc}

\usepackage{xcolor}
\usepackage{enumitem}
\usepackage{booktabs}
\usepackage{multicol}

\usepackage{array,longtable}
\newcolumntype{C}{>{$}c<{$}}  
\usepackage[a4paper]{geometry}
\usepackage{amssymb,amsmath,amsthm}
\usepackage{fdsymbol}

\newtheorem{thm}{Theorem}[section] 
\newtheorem{cor}[thm]{Corollary}

\newtheorem{prop}[thm]{Proposition}

\newtheorem*{ack}{Acknowledgements}

\theoremstyle{definition}
\newtheorem{defn}[thm]{Definition}
\newtheorem{rem}[thm]{Remark}
\newtheorem{exa}[thm]{Example}

\newtheorem{assumption}[thm]{Assumption}

\DeclareMathOperator{\SL}{SL}
\DeclareMathOperator{\PSL}{PSL}
\DeclareMathOperator{\GL}{GL}
\DeclareMathOperator{\Sp}{Sp}

\DeclareMathOperator{\SO}{SO}

\DeclareMathOperator{\PGL}{PGL}
\DeclareMathOperator{\Mat}{Mat}
\DeclareMathOperator{\Aut}{Aut}
\DeclareMathOperator{\Int}{Int}
\DeclareMathOperator{\diag}{diag}
\DeclareMathOperator{\rank}{rank}
\DeclareMathOperator{\ord}{ord}

\DeclareMathOperator{\Hom}{Hom}
\DeclareMathOperator{\cone}{Cone}
\DeclareMathOperator{\Bl}{Bl}
\DeclareMathOperator{\pic}{pic}
\DeclareMathOperator{\degree}{degree}

\newcommand{\bbC}{\mathbb{C}}
\newcommand{\bbQ}{\mathbb{Q}}
\newcommand{\bbZ}{\mathbb{Z}}
\newcommand{\bbP}{\mathbb{P}}
\newcommand{\bbR}{\mathbb{R}}
\newcommand{\bbN}{\mathbb{N}}
\newcommand{\bbG}{\mathbb{G}}
\newcommand{\bbF}{\mathbb{F}}

\newcommand{\id}{\text{id}}

\begin{document}

\title[spherical Fano varieties of dimension \(\leq 4\) and rank \(\leq 2\)]{Spherical actions on locally factorial Fano varieties of dimension \(\leq 4\) and rank \(\leq 2\)}

\author{Thibaut Delcroix}
\address{Thibaut Delcroix, IMAG Univ Montpellier, CNRS, Montpellier, France}
\email{thibaut.delcroix@umontpellier.fr}
\urladdr{http://delcroix.perso.math.cnrs.fr/}

\author{Pierre-Louis Montagard}
\address{Pierre-Louis Montagard, IMAG Univ Montpellier, CNRS, Montpellier, France}
\email{pierre-louis.montagard@umontpellier.fr} 

\date{\today}

\begin{abstract}
We obtain the exhaustive list of 337 faithful spherical actions of rank two or less on locally factorial Fano manifolds of dimension four or less. 
As a preliminary step, we determine the explicit list of spherical homogeneous spaces of dimension four or less, together with their combinatorial data.  
Then we classify the possible locally factorial \(G/H\)-reflexive polytopes for each such spherical homogeneous space \(G/H\). 
From the combinatorial data gathered in this article, one can easily read off the Picard rank (even the Picard group), Fano index, anticanonical volume of the underlying locally factorial Fano variety, etc.
\end{abstract}

\maketitle


\section{Introduction}

Kollár, Miyaoka and Mori proved boundedness of smooth Fano manifolds of a given dimension over \(\bbC\) \cite{Kollar_Miyaoka_Mori_1992, Kollar_1996}, in particular, that there is only a finite number of deformation classes of Fano manifolds of a given dimension. 
This number of deformation classes is expected to grow very fast, but the classification of Fano threefolds (up to deformations) was obtained through the combined work of Iskovskikh \cite{Iskovskih_1977,Iskovskih_1978}, Mori and Mukai \cite{Mori_Mukai_1981,Mori_Mukai_2003}: there are 105 deformation classes. 
Furthermore, their geometry is quite well understood, with various descriptions available. 

The full classification in dimension four is still unknown and challenging. 
It is thus desirable to have at least partial classification results, and in this direction, to consider subclasses of Fano manifolds. 
Almost-homogeneous Fano manifolds form a natural such subclass. 
These are the manifolds which admit a large group of automorphism, acting with an open and dense orbit. 
The most well-known among these are toric manifolds, when there is an algebraic torus of automorphisms acting with an open and dense orbit. 

More importantly, this additional data of a group action is very important in the study of various geometric questions. 
We thus switch to the goal of classifying almost-homogeneous, faithful actions of connected reductive groups on Fano fourfolds. 
This remains widely unknown to our knowledge, and we must restrict to a more reasonable class of actions. 

Toric manifolds are fully encoded by combinatorial data, and in particular, toric Fano manifolds are encoded by a class of polytopes with integral vertices up to unimodular transformations, namely smooth reflexive polytopes. 
Batyrev, Watanabe and Watanabe, and Sato used this combinatorial correspondence to classify toric Fano manifolds in dimension three \cite{Batyrev_1981,Watanabe_Watanabe_1982} and four \cite{Batyrev_1999,Sato_2000}. 
More generally, a huge step forward in the study of equivariant embeddings of homogeneous spaces (in other words, almost-homogeneous varieties) under the action of a connected reductive algebraic group was achieved by Luna and Vust \cite{Luna_Vust_1983}. 
In the case when the homogeneous space is spherical (that is, when not only the group but also any of its Borel subgroups act with an open and dense orbit), there is a combinatorial classification of embeddings, vastly generalizing the case of toric varieties \cite{Knop_1991}. 
Furthermore, spherical homogeneous spaces themselves are uniquely determined by combinatorial data \cite{Losev_2009}. 

While great effort was put on understanding the classification of spherical varieties, it appears that very little was done in an approach via small dimensions. 
Motivated by earlier work of Hofscheier in his PhD thesis, we initiate this approach by classifying all spherical homogeneous spaces of dimension less than four, then all their locally factorial Fano embeddings in rank two or less. 
The rank of a spherical variety being less than the dimension, rank three and four remain, but rank four corresponds to the toric case obtained by Batyrev and Sato. 

\begin{thm}
There are 337 faithful spherical actions on locally factorial Fano varieties of dimension four or less, and rank two or less. 
The combinatorial data of these actions are described in the body of the paper.  
\begin{equation*}
\begin{array}{c|cccc}
 \dim & 1 & 2 & 3 & 4 \\
\toprule
\rank =0 & 1 & 2 & 6 & 9 \\
\rank = 1 & 1 & 5 & 13 & 57 \\
\rank = 2 & 0 & 5 & 44 & 194 \\
\bottomrule
\rank \leq 2 & 2 & 12 & 63 & 260 \\
\end{array}
\end{equation*}
\end{thm}

We leave the rank three case for future work, as it most likely requires computer assistance (see however Section~\ref{section_SL2} for a characterization of the polytopes to classify, and some examples). We learned, as this preprint was almost fully written, that Girtrude Hamm is working out the classification of all canonical spherical Fano fourfolds of rank three, from which it should be easy to read off the locally factorial ones. 
Note that canonical toric Fano varieties of dimension three were classified by Kasprzyk \cite{Kasprzyk_2010} using computer assistance, yielding an astounding list of 674 688 integral polytopes up to \(\GL_3(\bbZ)\)-action. 

It must be noted, and obvious in view of the cases of dimension one, two and three, that a given locally factorial Fano variety may be equipped with several different faithful actions of connected reductive groups. 
These different structures can be useful in different settings, so it is important to classify these. 
As a basic example, \(\bbP^1\times \bbP^1\) is homogeneous under the action of \(\SL_2^2\), but if one considers the diagonal action of \(\SL_2\), one is able to consider the pairs \((\bbP^1\times \bbP^1, \varepsilon(\diag\bbP^1))\) and study their geometry with tools from the theory of spherical varieties.  

In a few exceptional cases, there are two faithful actions of different groups, but one action factors through the other \emph{and} the orbits are the same. 
This is obviously the least interesting case of our classification. 
The only examples that we are aware of in our list (apart from the action on the left of the connected normalizer of the isotropy group, which will always be taken into account), are the homogeneous space \(\bbP^3\) under the action of \(\Sp_4\), as well as the natural Fano \(\bbP^1\)-bundles built from it \(\bbP_{\bbP^3}(\mathcal{O}\oplus \mathcal{O}(k))\) for \(0\leq k\leq 3\), and the Fano cone over it, which is \(\bbP^4\). 

Toric varieties form, as already mentioned, the best-known subclass of spherical varieties. 
Toric varieties are often used to provide tractable examples of higher complexity actions by so-called downgrading the action: one considers a smaller dimensional torus acting, and obtains the combinatorial data for this smaller group action from the combinatorial data of the toric variety. 
In the list of spherical varieties, quite a few consist, to the opposite, of \emph{upgradings} of toric varieties: the connected reductive group is of rank four, so that the maximal torus action provides the structure of toric variety, but as a (horo)spherical variety the rank is lower. 
In particular, moment polytopes of upgradings of toric varieties are of a strictly smaller dimension than the dimension of the variety. 

Despite these comments, we should highlight that many examples in our classification are neither homogeneous nor toric under a larger group action. 
In fact, this was already known in the case of smooth Fano threefolds: there are 18 smooth Fano threefolds that admit a toric structure, 5 that are homogeneous, 3 of which being at the intersection of both classes, some of these admit spherical structures that are neither toric nor homogeneous, and there are nine Fano threefolds that admit only spherical structures which are neither toric nor homogeneous.  
For Fano fourfolds, there are many more examples. 

Let us finally comment on the \emph{locally factorial} assumption. 
We started this introduction discussing smooth Fano manifolds, but switched for our main statement to locally factorial Fano varieties. 
This is because, for spherical varieties, the criterion for being locally factorial is much easier to deal with than the criterion for being smooth. 
For toric varieties the two conditions are in fact equivalent. 
For spherical varieties in general it is not true, but it is still the case for toroidal spherical varieties. 
This observation, added to the various explicit descriptions given throughout the paper and some previously known classifications, allows one to show,  with little effort, that all but 16 of our varieties are smooth. 
Among the 16 remaining, we know for sure that three are singular, and leave it to the interested reader to work out which of the 13 remaining Fano varieties are smooth. 
Combinatorial smoothness criterions for spherical varieties are available \cite{Brion_1991,Gagliardi_2015}, but their application is somewhat involved in general. 
Let us also mention that a method to classify smooth Fano polytopes for spherical varieties under the action of a given group is described in \cite{Cupit-Foutou_Pezzini_Van_Steirteghem_2020}.  
We do not rely on this approach here. 

The paper is organized as follows. 
In Section~\ref{section_reminder}, we review the definitions needed for spherical homogeneous spaces and locally factorial Fano embeddings of such homogeneous spaces, as well as the fundamental properties that will be needed. We conclude that section with a very quick review of how one can recover plenty of geometric information on the embeddings directly from the combinatorial data provided in the paper. 
In Section~\ref{section_homogeneous_spaces}, we carry out the classification of spherical homogeneous spaces of dimension four or less. 
The key ingredient, that severely restricts the possible groups acting faithfully on fourfolds, is the local structure theorem for spherical homogeneous spaces. 
Without this ingredient, one should a priori consider all reductive groups of rank four or less (as there cannot be a faithful action of a torus of dimension strictly more than four on a fourfold). 
With the local structure theorem, together with the known classification of rank one spherical homogeneous spaces, one is essentially reduced to considering spherical subgroups of \(\SL_2\times \bbG_m^3\), \(\SL_2^2\times \bbG_m^2\) and \(\SL_3\times \bbG_m^2\). 
For these, we rely on the known classification of Lie subalgebras to conclude. 

The remaining of the paper provides, for each spherical homogeneous space of dimension four or less, the combinatorial data, as well as the list of locally factorial Fano embeddings of rank two or less, and dimension four or less. 
After introducing our conventions in Section~\ref{section_dimension_2}, we start with the case when the group is of the form \(\SL_2\times \bbG_m^n\) in Section~\ref{section_SL2}. 
Since fourfolds would be of rank three, we do not classify fourfolds in this case, but we do recover the classification results in the PhD theses of Pasquier \cite{Pasquier_2006} and Hofscheier for threefolds. 
In Section~\ref{section_SL22}, we deal with spherical varieties under the action of \(\SL_2^2\times \bbG_m^n\). This is where there are more diverse cases in rank two. 
Some special cases were known thanks to the work of Ruzzi \cite{Ruzzi_2012} for embeddings of symmetric spaces. 
Finally, we deal with the case of \(\SL_3\times \bbG_m^n\) in Section~\ref{section_SL3}, then with the remaining rank one cases in Section~\ref{section_remaining}. 

In the appendices we provide big tables providing information on each spherical action and the underlying Fano variety. 
For threefolds we identify the underlying Fano threefold in each case, while for the case of Fano fourfolds we leave it for later refinement. 
We however do provide geometrical data such as the Picard rank, anticanonical degree, and we determine if they admit Kähler-Einstein metrics. 
These invariants alone allow us to affirm (by comparing with the data from \cite{Batyrev_1999,Sato_2000})  that at least 117 different underlying locally factorial Fano fourfolds appear, among which at least 42 non toric fourfolds.  
A natural next step, that we leave for the future, would be to compute the Cox rings using \cite{Gagliardi_2014}. 
This would give in particular a way to identify the underlying Fano varieties.

Another natural next step to this paper would be to relax the assumptions on the singularities, to classify Gorenstein/terminal/canonical spherical Fano fourfolds. 
Part of our paper (the determination of spherical homogeneous spaces and their combinatorial data) is still applicable, but the determination of all possible polytopes may be more involved. 

Finally, one of the first author's motivation is naturally the quest for canonical Kähler metrics. 
Again, up to the issue of identifying isomorphic underlying Fano fourfolds, we solved here the question of existence of Kähler-Einstein metrics, obtaining at least 24 Kähler-Einstein fourfolds, and at least 93 non Kähler-Einstein fourfolds. 
The existence of many other canonical Kähler metrics could be settled using our data, by applying for example \cite{Delcroix_2023_RK1,Delcroix_2023_KSSV2} for cscK metrics, and \cite{Li_Li_Wang} for weighted solitons. 

The article is written with lots of details when the first notions appear, with lots of examples. 
In the later sections, we give less details on how the classifications are obtained to keep the article not excessively long. 

\begin{ack}
The first author is partially funded by ANR-21-CE40-0011 JCJC project MARGE and ANR-18-CE40-0003 JCJC project FIBALGA.
We thank Girtrude Hamm for informing us of her upcoming work. 
We thank Stéphanie Cupit-Foutou, Liana Heuberger, Alexander Kasprzyk, Laurent Manivel, Naoto Yotsutani and Xiaohua Zhu for their interest and comments on our work. 
\end{ack}

\section{Preliminaries on spherical varieties}
\label{section_reminder}

\subsection{Notation}

Throughout the article, we work over the field of complex numbers \(\bbC\). 

Let \(G\) be a connected complex linear reductive group.  
Fix a Borel subgroup \(B\) of \(G\) (recall that any two Borel subgroups are conjugate) and a maximal torus \(T\subset B\). 
We denote by \(X^*(B)\) the group of characters of \(B\), usually denoted additively (we will use the same notation for other groups, for example \(X^*(G)\) is the group of characters of \(G\), reduced to \(\{0\}\) if \(G\) is semisimple). 

Let \(R\subset X^*(T)\) denote the root system of \(G\) with respect to \(T\), and let \(R^+\) and \(S=\{\alpha_1,\ldots,\alpha_{\rank(G)}\}\subset R^+\subset R\) denote the set of positive roots and the set of simple roots associated to the Borel subgroup \(B\). 
Let \(\{\varpi_1,\ldots, \varpi_{\rank(G)}\}\) denote the fundamental weights associated to the simple roots. 
The integer \(\rank(G)\) is called the rank of \(G\). 

When \(G=G_1\times\cdots\times G_n\), we fix \(T=T_1\times \cdots\times T_n\) and \(B=B_1\times \cdots\times B_n\) where \(T_i\) is a maximal torus of \(G_i\) and \(B_i\) is a Borel subgroup of \(G_i\) containing \(T_i\). 
We always assume that the simple roots and fundamental weights are ordered consistently with the product: the roots \(\alpha_1,\ldots,\alpha_{\rank(G_1)}\) are the simple roots of \(G_1\), the roots \(\alpha_{\rank(G_1)+1},\ldots, \alpha_{\rank(G_1)+\rank(G_2)}\) are the simple roots of \(G_2\), etc. 
When we explicitly write \(G=G_1\times \bbG_m^n\) where \(G_1\) is a semisimple group, we denote by \(\chi_i\) the projection to the \(i\)-th \(\bbG_m\) factor. 
These define characters of \(G\). 

A character \(\lambda\) of \(T\) may, depending on the context, be defined on various groups, either subgroups of \(T\) or groups containing \(T\). 
We distinguish these when necessary by indicating the domain we are considering as \(\lambda|_\bullet\). 
For example, if \(G=G_1\times \bbG_m\), \(\chi_1\) may be considered as the identity character \(\chi_1|_{\{1\}\times\bbG_m}\) on \(\bbG_m\), as a character \(\chi_1|_G\) of the group \(G\), or as an element \(\chi_1|_T\in X^*(T)\). 

Standard parabolic subgroups of \(G\) are the subgroups of \(G\) that contain \(B\). 
They are parametrized by the set of subsets of \(S\). 
More precisely, for \(I\subset S\), we denote by \(P_I\) the (unique) parabolic subgroup containing \(B\) such that \(\varpi_i\) extends to \(P_I\) for \(i\in I\). We may thus write, with the convention above, \(\varpi_i|_{P_I} \in X^*(P_I)\).
Beware that some authors rather associate the finite set \(S\setminus I\) with \(P_I\). 
Note that any parabolic subgroup of \(G\) is conjugate to a standard parabolic subgroup, which is of the form \(P_I\). 

Let \(B^-\) denote the Borel subgroup opposite to \(B\). 
We denote also by \(Q_I\) the (unique) parabolic subgroup containing \(B^-\) such that \(\varpi_i\in X^*(Q_I)\) for \(i\in I\).

\subsection{Spherical homogeneous spaces}

\begin{defn}
An algebraic subgroup \(H\subset G\) is \emph{spherical} if \(B\) acts with an open orbit in the homogeneous space \(G/H\). 
The homogeneous space \(G/H\) itself is also called a \emph{spherical homogeneous space}. 
\end{defn}

Equivalently, an algebraic subgroup is spherical if and only if the action of \(H\) on \(G/B\) has an open orbit. 
Note that this condition reads on the Lie algebras: 
an algebraic subgroup \(H\) is spherical if there exists a \(g\in G\) such that \(\mathfrak{g}=\mathfrak{b} + \operatorname{Ad}(g)(\mathfrak{h})\). 

\begin{exa}
\label{example_torus_subgroups}
We start with a trivial example for reference later in the paper. 
Consider the case where \(G\simeq \bbG_m^n\) is a torus. 
Then algebraic subgroups are all of the form \(H=\bigcap_{i=1}^k \ker(\chi_i : G\to \bbG_m)\) where \(\{\chi_i\}\subset X^*(G)\) is a family of \(\bbZ\)-linearly independent characters of \(G\). 
Any such subgroup is spherical, since \(B=G\) in this case, hence the action of \(B\) on \(G/H\) is surjective. 
\end{exa}

\begin{exa}
\label{example_sl2}
An algebraic subgroup \(H\) of \(\SL_2\) is spherical if it acts with an open orbit on \(\SL_2/B=\bbP^1\). 
As a consequence, any positive dimensional algebraic subgroup of \(\SL_2\) is spherical. 
The classification of subgroups of \(\SL_2\) up to conjugation is well known: 
a positive dimensional algebraic subgroup \(H\) of \(\SL_2\) is either:
\(H=\SL_2\), \(H=T\) a maximal torus, \(H=N_{\SL_2}(T)\) its normalizer, or \(H=\ker\chi\subset B\) where \(B\) is a Borel subgroup of \(\SL_2\) and \(\chi\) is a character of \(B\). 
Note that the group \(X^*(B)\) of characters of \(B\) is isomorphic to \(\bbZ\), and that \(\chi\) and \(-\chi\) define the same algebraic subgroup \(\ker(\chi)\), hence the latter infinite family of subgroups is parametrized by \(\bbZ_{\geq 0}\).  
\end{exa}

From now on, we fix \(H\), a spherical subgroup of \(G\). 
The group \(G\) acts on the vector space \(\bbC(G/H)\) of rational functions on \(G/H\) by \((g\cdot f)(x)=f(g^{-1}x)\). 

\begin{defn}
The \emph{weight lattice} \(M=M(G/H)\) of \(G/H\) is the subset of \(X^*(B)\) formed by all eigenvalues for the action of \(B\) on \(\bbC(G/H)\), that is, \(\lambda\in M\) if there exists \(0\neq f\in \bbC(G/H)\) such that for all \(b\in B\), \(b\cdot f=\lambda(b)f\).   
The \emph{rank} \(\rank(G/H)\) of \(G/H\) is the rank of the lattice \(M(G/H)\).
\end{defn}

\begin{exa}
If \(G\) is a torus and \(H\) is the trivial subgroup, then \(M=X^*(G)\) since any character of \(G\) defines a \(B=G\)-semi-invariant rational function on \(G\).  
\end{exa}

\begin{exa}
\label{exa_rk_0}
By the Bruhat decomposition, a parabolic subgroup \(Q\) of \(G\) is spherical. 
Furthermore, since the unipotent subgroup \(B^u\) also acts with an open dense orbit on \(G/Q\), any \(B\)-eigenvector is constant. 
Hence \(M(G/Q)=\{0\}\) and \(\rank(G/Q)=0\). 

Conversely, if \(\rank(G/H)=0\) then there is an open dense orbit of \(B^u\) in \(G/H\), and by Lie algebra considerations, \(H\) must contain a Borel subgroup of \(G\), i.e. \(H\) is a parabolic subgroup of \(G\). 
\end{exa}

\begin{exa}
\label{P^1xP^1}
Let \(G=\SL_2\) and \(H=T\) the maximal torus of diagonal matrices. Since we are interested in rational functions, we may as well work on a nice projective model of \(G/H\). 
Consider the diagonal action of \(\SL_2\) on \(\bbP^1\times \bbP^1\), by 
\[ \begin{bmatrix} a & b \\ c & d \end{bmatrix} \cdot ([x_1:y_1],[x_2:y_2]) = ([ax_1+by_1:cx_1+dy_1],[ax_2+by_2:cx_2+dy_2]). \]
Then there are two orbits: the diagonal embedding of \(\bbP^1\) and its complement, which is \(G\)-equivariantly isomorphic to \(G/H\) (\(H\) is the stabilizer of the point \(([1:0],[0:1])\). 

Let \(B\) be the Borel subgroup of upper triangular matrices. 
One easily checks that the rational function \(f:([x_1:y_1],[x_2:y_2])\mapsto \frac{x_1}{y_1} - \frac{x_2}{y_2}\) is \(B\)-semi-invariant with weight the character \(\begin{bmatrix} a & b \\ 0 & a^{-1} \end{bmatrix} \mapsto a^{-2}\), which is the negative root \(-\alpha_1=-2\varpi_1\). 
As a consequence, we have \(\alpha_1\bbZ\subset M\). 

Conversely, let \(\lambda\in M\) and let \(f\) be a  \(B\)-semi-invariant rational function with weight \(\lambda\). 
Then since \(-I_2\) acts trivially on \(\bbP^1\times \bbP^1\), we have 
\( -I_2\cdot f = f = \lambda(-I_2) f\), hence \(\lambda\in \alpha_1\bbZ\). 
\end{exa}

\begin{exa}
\label{P^2}
For \(G=\SL_2\) and \(H=N(T)\), we obtain a projective model by considering the projectivization \(\bbP^2\) of size two symmetric matrices, equipped with the action of \(\SL_2\) by congruences. 
Writing such a matrix as \(\begin{bmatrix} x & y \\ y & z \end{bmatrix}\), and homogeneous coordinates as \([x:y:z]\), the action is given by 
\[ \begin{bmatrix} a & b \\ c & d \end{bmatrix} \cdot [x:y:z] = [a^2x+2aby+b^2z: acx+(ad+bc)y+bdz : c^2x+2cdy+d^2z] \]
The rational function \(f:[x:y:z]\mapsto \frac{y^2-xz}{z^2}\) is semi-invariant under the action of the Borel subgroup \(B\) of upper-triangular matrices, with weight \(-4\varpi_1 : \begin{bmatrix} a & b \\ 0 & a^{-1} \end{bmatrix} \mapsto a^{-4}\). 
Hence \(4\varpi_1\bbZ\subset M\). 

This is actually an equality. 
Indeed, note that the equation \(y^2-xz=0\) defines the prime \(\SL_2\)-stable divisor in \(\bbP^2\) which is the complement of the open orbit \(\SL_2/N(T)\). 
If \(M\) was larger than \(4\varpi_1\bbZ\), then there would be a \(B\)-semi-invariant rational function \(r\) on \(\bbP^2\) with weight \(-2\varpi_1\). 
Its square \(r^2\) would be a \(B\)-semi-invariant rational function on \(\bbP^2\) with weight \(-4\varpi_1\) hence by the spherical property, \(r^2=Cf\) for some constant \(C\in \bbC\). 
This is inconsistent since the order of vanishing of \(f\) along the divisor \(\{y^2-xz=0\}\) is one, while the order of vanishing of \(r^2\) along this divisor is an even integer. 
\end{exa}

A fundamental result in the theory of spherical homogeneous spaces and their equivariant embeddings is the local structure theorem of Brion, Luna and Vust \cite{BLV_1986}. 
We only state the result for homogeneous spaces in order to derive some simple consequences. 

\begin{defn}
The \emph{adapted parabolic subgroup} associated to \(G/H\) is the stabilizer \(P\) of the open \(B\)-orbit in \(G/H\).
\end{defn}

\begin{prop}{\cite{BLV_1986}} 
\label{local_structure_theorem}
Assume that \(BH\) is open in \(G\). 
Then there exists a Levi subgroup \(L\) of the adapted parabolic subgroup \(P\), with connected center \(Z\), such that \(P\cap H=L\cap H\) contains \([L,L]\), and that the map 
\(P^u\times Z/(Z\cap H)\to B/(B\cap H), (p,x)\mapsto p\cdot x\) is an isomorphism.  
In particular, \(M=X^*(Z/(Z\cap H))\). 
\end{prop}

Note that we can always reduce to the case when \(BH\) is open in \(G\), by replacing \(H\) or \(B\) with a conjugate. Anyway we will mostly use the following consequence, which does not require conjugating \(H\) or \(B\). 

\begin{cor}
\label{cor_local_structure_theorem}
Let \(P\) be the adapted parabolic subgroup of \(G\). 
Then \(\dim(G/H)=\rank(G/H)+\dim(G/P)\).
Furthermore, if \(P\) contains a simple factor of \(G\), then this factor acts trivially on \(G/H\). 
\end{cor}

\subsection{Valuation cone}

A \emph{valuation} on \(G/H\) is a group morphism \(\nu:\bbC(G/H)^*\to \bbQ\) such that \(\nu\) takes value zero on constant functions, and \(\nu(f_1+f_2)\geq \min(\nu(f_1),\nu(f_2))\) whenever \(f_1\), \(f_2\) and \(f_1+f_2\) are in \(\bbC(G/H)\). 
Given a valuation \(\nu\) on \(G/H\), let \(\rho(\nu)\) denote the morphism 
\(M \to \bbQ\) which sends \(\lambda \in M\) to \(\nu(f_{\lambda})\) where \(f_{\lambda}\) is an eigenvector of \(B\) in \(\bbC(G/H)\) corresponding to the eigenvalue \(\lambda\).  
Set \(N:=\Hom(M,\bbZ)\), and let \(\langle\cdot,\cdot\rangle\) denote the duality bracket between \(M\otimes \bbQ\) and \(N\otimes\bbQ\).  

\begin{prop}{\cite{Brion_Pauer_1987}}
\label{prop_valuation_cone}
\begin{itemize}
    \item The map \(\rho\) induces an isomorphism from the set of all \(G\)-invariant valuations on \(G/H\) to a closed cosimplicial convex cone \(\mathcal{V}\) in \(N\otimes\bbQ\). 
    \item The cone \(\mathcal{V}\) is the cone of elements \(x\in N\otimes\bbQ\) such that \(\langle \sigma,x\rangle\leq 0\) for all \(\sigma\) such that there exist two simple \(G\)-sub-modules of \(\bbC[G]^{(H)}\), \(M_1\) and \(M_2\), and  \(B\)-eigenvectors \(f_1\in M_1\), \(f_2\in M_2\) and \(f\in M_1M_2\), such that \(\sigma\) is the \(B\)-eigenvalue of \(f_1f_2f^{-1}\).  
\end{itemize}
\end{prop}

\begin{defn}
The cone \(\mathcal{V}\) is called the \emph{valuation cone} of \(G/H\).
The (finite) set \(\Sigma\) of primitive elements of \(M\) such that \(\mathcal{V}=\{x\in N\otimes\bbQ\mid \langle\sigma,x\rangle\leq 0\}\) is called the set of \emph{spherical roots} of \(G/H\). 
\end{defn}

\begin{exa}
\label{exa_valuation_cone_SL2}
Consider the case \(G=\SL_2\) and \(H=T\). 
We write an element of \(G\) as 
\[ \begin{bmatrix} a & b\\ c & d\end{bmatrix} \]
and choose as usual \(B\) the subgroup of upper triangular matrices, and \(T\) the subgroup of diagonal matrices. 
Consider the regular functions on \(G\) defined by  
\[ f_1: \begin{bmatrix} a & b\\ c & d\end{bmatrix} \mapsto d \]
and 
\[ f_2: \begin{bmatrix} a & b\\ c & d\end{bmatrix} \mapsto c \]
By direct computation, these two functions are \(B\)-semi-invariant (on the right) with weight \(\varpi_1\), and \(T\)-semi-invariant (on the left). 
Let \(M_1\) and \(M_2\) the simple \(G\)-sub-modules of \(\bbC[G]^{(T)}\) with \(B\)-stable line generated by \(f_1\) and \(f_2\). 
Consider now the function 
\begin{align*} f & = f_1\left(\begin{bmatrix} 0 & 1\\ -1 & 0\end{bmatrix}\cdot f_2\right) - \left(\begin{bmatrix} 0 & 1\\ -1 & 0\end{bmatrix}\cdot f_1\right)f_2 \in M_1M_2 \\
f & : \begin{bmatrix} a & b\\ c & d\end{bmatrix} \longmapsto da-bc \equiv 1
\end{align*}
Since \(f\) is constant, it is a \(B\)-eigenvector, and the \(B\)-eigenvalue of \(f_1f_2f^{-1}\) is equal to \(2\varpi_1=\alpha_1\). 

Recall from Example~\ref{P^1xP^1} that \(M=\alpha_1\bbZ\), and let \(\xi\in N\) be the primitive generator of \(N\) such that \(\langle\xi,\alpha_1\rangle=1\).  
From the above we deduce that \(\alpha_1\in \Sigma\), and thus \(\mathcal{V}\subset \xi \bbQ_{\leq 0}\). 
But from Example~\ref{P^1xP^1} again, we know that \(\mathcal{V}\neq \{0\}\) since the order of vanishing along the diagonal in \(\bbP^1\times \bbP^1\) is a  non-trivial \(G\)-invariant valuation. 
Thus \(\mathcal{V}=\xi \bbQ_{\leq 0}\) and \(\{\alpha_1\} = \Sigma\).
\end{exa}

From that description of the valuation cone, Brion and Pauer derive strong conditions on the possible spherical roots. 

\begin{prop}{\cite{Brion_Pauer_1987}}
\label{prop_spherical_root_as_sum_of_positive_roots}
Choose a maximal torus \(T\subset B\). 
An element of \(\Sigma\) is, up to multiple, either a positive root or the sum of two strongly orthogonal positive roots of \(G\). 
\end{prop}

\begin{exa}
When \(G=\SL_2\), or more generally \(\SL_2\times \bbG_m^n\), then either \(\Sigma=\emptyset\) or \(\Sigma\) is the singleton whose element is the primitive multiple of \(\alpha_1\) in \(M\).  
\end{exa}

\subsection{Colors}

\begin{defn}
A codimension one orbit of \(B\) in \(G/H\) is called a \emph{color} of \(G/H\). 
We will often identify the orbit with its orbit closure, which is a prime divisor in \(G/H\), or its closure in a given embedding \(X\) of \(G/H\), which is a prime divisor in \(X\).  
The set of colors of \(G/H\) is denoted by \(\mathcal{D}(G/H)\). 
There are two maps associated to colors. 
The first is the restriction of the map \(\rho\) to \(\mathcal{D}(G/H)\), identified with the set of divisorial valuations induced by the colors. 
The second is the map \(\zeta\) from \(\mathcal{D}(G/H)\) to the set  \(\mathcal{P}(S)\) of parabolic subgroups of \(G\) containing \(B\), which sends a color to its stabilizer in \(G\). 
For simplicity we identify \(P_I\) with \(I\) in the notations when it does not create confusions. 
\end{defn}

Beware that the valuations associated to colors are not \(G\)-invariant, hence the color map \(\rho: \mathcal{D}(G/H)\to N\) is not injective in general.  
Note also that the adapted parabolic is \(P_I\) where \(I=\bigcup_{D\in \mathcal{D}} \zeta(D)\). 

\begin{exa}
We go on with Example~\ref{P^1xP^1}, with the same notation. 
It follows from the description of the action that there are two colors \(D_+\) and \(D_-\) of \(\SL_2/T\) whose equations are given by \(y_1=0\) and \(y_2=0\). 
We have \(\zeta(D_{\pm})=\{\alpha_1\}\), that is, the stabilizer of each color is the Borel subgroup itself. 
The order of the function \(f:([x_1:y_1],[x_2:y_2])\mapsto \frac{x_1}{y_1}-\frac{x_2}{y_2} = \frac{x_1y_2-x_2y_1}{y_1y_2}\) 
on \(D_{\pm}\) is \(-1\) for \(i=1\), \(2\). 
Hence \(\rho(D_{\pm})\) is the morphism defined by \(\rho(D_{\pm})(\alpha_1)=1\). 
\end{exa}

\begin{defn}
The \emph{combinatorial data} associated to \(G/H\) is the triple \((M,\Sigma,(\mathcal{D},\rho,\zeta))\) where \(M\subset X^*(B)\) is the weight lattice, \(\Sigma\subset M\) is the set of spherical roots, and \(\mathcal{D}\) is the set of colors, thought of as an abstract set equipped with two maps \(\rho:\mathcal{D}\to N=\Hom(M,\bbZ)\) and \(\zeta:\mathcal{D}\to \mathcal{P}(S)\). 
\end{defn}

In the above definition, stating that \(\mathcal{D}\) is an abstract set means that the description of elements of \(\mathcal{D}\) as \(B\)-orbits in \(G/H\) can be totally forgotten. 
In other words, we may write \((\mathcal{D}_1,\rho_1,\zeta_1)=(\mathcal{D}_2,\rho_2,\zeta_2)\) if there exists a bijection \(f:\mathcal{D}_1\to \mathcal{D}_2\) such that \(\rho_1=\rho_2\circ f\) and \(\zeta_1=\zeta_2\circ f\). 

\begin{prop}{\cite{Losev_2009}}
The combinatorial data \((M,\Sigma,(\mathcal{D},\rho,\zeta))\) fully encodes the spherical subgroup \(H\) up to conjugacy. 
\end{prop}

\begin{exa}
Let \(\alpha_i\) be a simple root, and \(Q_{\{\alpha_i\}}\) the maximal parabolic subgroup of \(G\) containing \(B^-\) and with \(\varpi_i\in X^*(Q_{\{\alpha_i\}})\) as introduced in the notations section. 
By the Bruhat decomposition, there are two \(B\)-orbits in \(G/Q_{\{\alpha_i\}}\): an open orbit, and a codimension one orbit \(D_{\alpha_i}\). 
In particular, \(Q_{\{\alpha_i\}}\) is a spherical subgroup,  \(\mathcal{D}(G/Q_{\{\alpha_i\}})=\{D_{\alpha_i}\}\) and \(\zeta(D_{\alpha_i})=P_{\{\alpha_i\}}\). 

More generally, for \(I\subset \mathcal{P}(S)\), let \(Q_I\) denote the associated parabolic subgroup containing \(B^-\) and for \(i\in I\), let \(\pi_i:G/Q_I\to G/Q_{\{\alpha_i\}}\) denote the projection induced by the inclusion \(Q_I\subset Q_{\{\alpha_i\}}\). 
By the Bruhat decomposition again, the parabolic subgroup \(Q_I\) is a spherical subgroup, \(\mathcal{D}(G/Q_I)=\{\pi_i^{-1}(D_{\alpha_i})\mid i\in I\}\simeq I\) and \(\zeta(\pi_i^{-1}(D_{\alpha_i}))=\{\alpha_i\}\). 

Recall from Example~\ref{exa_rk_0} that \(M(G/Q_I)=\{0\}\). 
As a consequence, \(\Sigma=\emptyset\) and \(\rho:\mathcal{D}(G/Q_I)\to \{0\}\) is the zero map. 
We have fully described the combinatorial data for spherical homogeneous spaces of rank \(0\). 
\end{exa}

\subsection{Parabolic induction of spherical homogeneous spaces}

\begin{defn}
If there exists a \emph{proper} parabolic subgroup \(Q\) of \(G\), a connected reductive group \(G_0\), a spherical subgroup \(H_0\subset G_0\) and an epimorphism \(\pi:Q\to G_0\) such that \(G/H\) is the quotient \(\left(G\times G_0/H_0\right) /Q\) under the action \(q\cdot (g,x)=(gq^{-1},\pi(q)\cdot x)\), then we say that \(G/H\) is obtained by \emph{parabolic induction}l from \(G_0/H_0\). 
\end{defn}

\begin{rem}
\label{induction_from_Lie}
Note that \(H\subset Q\), and \(Q^u\subset H\). In fact, this characterizes spherical homogeneous spaces that are obtained by parabolic induction, and this can be read off from the Lie algebras of \(G\) and \(H\) alone. 
\end{rem}

We place ourselves in the setting of the previous definition and assume furthermore that \(Q\) contains the Borel subgroup \(B^-\) opposite to \(B\).
Hence \(Q=Q_I\) for some subset \(I\) of simple roots. 
Let \(B_0\) be the image of \(Q\cap B\) by \(\pi\). 
It is a Borel subgroup of \(G_0\). 
The projection from \(Q\cap B\) to \(B_0\) induces an inclusion \(X^*(B_0)\subset X^*(Q\cap B)=X^*(B)\). 

The following result allows to recover the combinatorial data of the spherical homogeneous space \(G/H\), obtained by parabolic induction, from the combinatorial data of \(G_0/H_0\) and the data of \(I\) alone.

\begin{prop}
\label{prop_parabolic_induction}
The combinatorial data \((M,\Sigma,(\mathcal{D},\rho,\zeta))\) and \((M_0,\Sigma_0,(\mathcal{D}_0,\rho_0,\zeta_0))\) of \(G/H\) and \(G_0/H_0\) satisfy the following relations:

\[\arraycolsep=12pt
\begin{array}{lll}
M=M_0 & \forall \alpha_i\in I, \rho(\alpha_i)=\alpha_i^{\vee}|_{M} & \rho|_{\mathcal{D}_0}=\rho_0
\\  
\Sigma=\Sigma_0 & \forall \alpha_i\in I, \zeta(\alpha_i)=\{\alpha_i\} & \zeta|_{\mathcal{D}_0}=\zeta_0  
\\ 
\mathcal{D}\simeq \mathcal{D}_0\sqcup I &  & 
\\
\end{array}
\]
\end{prop}

\begin{proof}
All statements but that about \(\zeta\) are directly contained in \cite[Proposition~20.4]{Timashev_2011}. 
Let us just recall a bit how the colors can be described to justify the expression of \(\zeta\). 

For each \(\alpha_i\in I\), the associated \(B\)-stable prime divisor in \(G/H\) is the pullback of the closed \(B\)-orbit in \(G/Q_{\alpha_i}\) under the projection \(G/H\to G/Q_{\alpha_i}\) induced by the inclusion \(H\subset Q_I\subset Q_{\alpha_i}\). 
The stabilizer of this divisor is thus exactly \(P_{\alpha_i}\). 

Given a \(B_0\)-stable prime divisor \(D_0\) of \(G_0/H_0\), the corresponding \(B\)-stable prime divisor is \(\frac{G\times D_0}{Q}\) of \(G/H\). 
It follows that \(\zeta(D)=\zeta_0(D)\). 
Note that this equality makes sense since the set of simple roots of \(G_0\) with respect to \(B_0\) is a subset of the set of simple roots of \(G\) with respect to \(B\) under the inclusion \(\pi^*\). 
\end{proof}

\begin{defn}
A spherical homogeneous space \(G/H\) obtained by parabolic induction from a torus (\(G_0=\bbG_m^k\), \(H_0=\{\id\}\)) is called a \emph{horospherical homogeneous space}. 
\end{defn}

\begin{rem}
\label{remark_horospherical_subgroups}
It follows from the definition and Example~\ref{example_torus_subgroups} that a horospherical homogeneous space is defined by the data of a parabolic subgroup \(Q\) and of a finite set \(\{\chi_1,\ldots,\chi_k\}\subset X^*(Q)\) of \(\bbZ\)-linearly independent characters of \(Q\). 
Note also that \(Q=N_G(H)\) and that one recovers the horospherical subgroup \(H\) as the intersection of kernels \(H=\bigcap_{i=1}^k\ker(\chi_i)\). 
\end{rem}

\begin{exa}
Let \(I\subset S\) and let \(\{\chi_1,\ldots,\chi_k\}\) be a \(\bbZ\)-linearly independent family in \(X^*(Q_I)\). 
Let \(H=\bigcap_{i=1}^k \ker (\chi_i:Q_I\to \bbG_m)\) be a horospherical subgroup. 
Then by Proposition~\ref{prop_parabolic_induction},
\(M\) is the lattice generated by the \(\chi_i\), 
\(\Sigma=\emptyset\), 
\(\mathcal{D}\simeq I\), 
\(\zeta(\alpha_i)=\{\alpha_i\}\) and \(\rho(\alpha_i)=\alpha_i^{\vee}|_M\). 
\end{exa}

By Example~\ref{example_sl2}, apart from the torus and its normalizer, all spherical subgroups of \(\SL_2\) are horospherical. 

\subsection{Equivariant automorphism group}

\begin{prop}{\cite[Section~5]{Brion_Pauer_1987}}
\label{prop_equivariant_automorphisms}
The group \(\Aut^G(G/H)=N_G(H)/H\) of \(G\)-equivariant automorphisms of \(G/H\) is diagonalizable. 
Its neutral component is a torus of dimension \(d=\dim(\mathcal{V}\cap(-\mathcal{V}))\), the dimension of the linear part of the valuation cone. 
\end{prop}

\begin{rem}
Note that a spherical homogeneous space \(G/H\) is thus also spherical under the action of the connected reductive group \(G\times \Aut^{G,0}(G/H)\), and under the action of its image in \(\Aut(G/H)\). 
\end{rem}

\begin{cor}{\cite[Corollaire~5.4]{Brion_Pauer_1987}}
\label{cor_horospherical_by_valuation_cone}
A spherical homogeneous space \(G/H\) is horospherical if and only if the valuation cone is \(\mathcal{V}=N\otimes \bbQ\), equivalently, \(\Sigma=\emptyset\). 
\end{cor}

Combining this characterization with Proposition~\ref{prop_spherical_root_as_sum_of_positive_roots}, we can get a description of the spherical roots when there are few options available. 

\begin{cor}
\label{cor_Sigma_SL_2}
Assume that \(G=\SL_2\times \bbG_m^n\) for some \(n\in \bbZ_{\geq 0}\), and that \(G/H\) is \emph{not} horospherical. 
Then \(\Sigma\) is the singleton formed by the primitive positive multiple of \(\alpha_1\) in \(M\). 
\end{cor}

\subsection{Fano embeddings and \(G/H\)-reflexive polytopes}

We recall the combinatorial caracterization of Fano Gorenstein spherical embeddings of \(G/H\) obtained from \cite{Brion_1989,Brion_1997} in \cite{Gagliardi_Hofscheier_2015_fano}, adding in the locally factorial condition. 

Let \(\kappa\) be the sum of all roots of the adapted parabolic \(P\).
For \(D\in \mathcal{D}\), set \(m_D:=1\) if \(\zeta(D)\cap \left(\Sigma\cup \frac{1}{2}\Sigma\right)\neq \emptyset\), and \(m_D:=\langle \alpha^{\vee}, \kappa \rangle\) for \(\alpha\in \zeta(D)\) otherwise. 

\begin{defn}
\label{defn_reflexive}
A polytope \(\Omega\subset N\otimes \bbQ\) with set of vertices \(V(\Omega)\) is called \emph{locally factorial \(G/H\)-reflexive} if: 
\begin{enumerate}
    \item \(0\in \Int(\Omega)\),
    \item \(\forall D\in \mathcal{D}, \frac{\rho(D)}{m_D}\in \Omega\),
    \item \(V(\Omega)\subset \left((N\cap\mathcal{V}) \cup \{\frac{\rho(D)}{m_D}, D\in \mathcal{D}\}\right)\),
    \item for every facet \(F\) of \(\Omega\) such that \(\cone(F)\cap \Int(\mathcal{V})\neq \emptyset\), let \(\mathcal{D}_F=\{D\in \mathcal{D}\mid \frac{\rho(D)}{m_D} \in F\}\), then 
    \begin{enumerate}
        \item \(\rho:\mathcal{D}_F\to \cone(F)\) is injective,
        \item \(V(F)=\{\frac{\rho(D)}{m_D}\mid D\in \mathcal{D}_F\}\cup\mathcal{C}_F\), and \(\mathcal{C}_F\cup \rho(\mathcal{D}_F)\) forms a basis of \(N\) .
    \end{enumerate}
\end{enumerate}
\end{defn}

\begin{rem}
In \cite[Definition~2.10]{Gagliardi_Hofscheier_2015_fano}, with their notation,  the colored cone associated to a supported cone spanned by a facet \(F\) of the polytope should read \((\cone(F),\{D_i \mid \frac{\rho(D_i)}{m_i}\})\) instead of \((\cone(F),\rho^{-1}(F))\) (compare \cite[Remarque~3.3]{Pasquier_2008}). 
\end{rem}

\begin{prop}{\cite[Theorem~1.9]{Gagliardi_Hofscheier_2015_fano}}
The set of isomorphism classes of locally factorial Fano spherical embeddings of \(G/H\) is in bijection with the set of locally factorial \(G/H\)-reflexive polytopes. 
\end{prop}

\begin{exa}
\label{LFF_SL_2_surfaces}
For \(\SL_2\)-varieties, the possible normal equivariant embeddings are well known, and it is easy to check which among them are locally factorial Fano. 
In the non-horospherical case, the only examples are \(\bbP^1\times \bbP^1\) and \(\bbP^2\), which are smooth and Fano. 
In the horospherical case, apart from the rank zero \(\bbP^2\), one gets two embeddings for each choice of horospherical subgroup \(H=\ker(a\varpi_1:B^-\to \bbG_m)\) for \(a\in \bbZ_{>0}\): the weighted projective plane \(\bbP(1,1,a)\) and the \(\bbP^1\)-bundle  \(\bbP_{\bbP^1}(\mathcal{O}\oplus \mathcal{O}(a))\). 
The only examples that are Fano and locally factorial are for \(a=1\): \(\bbP^2=\bbP(1,1,1)\) and  \(\Bl_{\text{point}}\bbP^2\simeq \bbP_{\bbP^1}(\mathcal{O}\oplus \mathcal{O}(1))\).  

This last statement can be recovered quickly from Gagliardi and Hofscheier's result: consider the case \(H=\ker(a\varpi_1)\). 
One has \(M=a\varpi_1\bbZ\). 
We let \(\xi\) denote the generator of \(N\) such that \(\xi(a\varpi_1)=1\). 
Let \(\Omega\) be a locally factorial \(G/H\)-reflexive polytope, then \(\Omega=[d\xi,e\xi] \subset N\otimes \bbQ\) and from the definition, we have the following conditions: 
\begin{enumerate}
    \item \(d<0<e\)
    \item \(\frac{a}{2}\leq e\)
    \item \(d\in \bbZ\) and \(e\in \bbZ\cup\{\frac{a}{2}\}\)
    \item \begin{enumerate}
        \item no condition
        \item \(d=-1\) and \(e=1\neq \frac{a}{2}\) or \(a=1=2e\)
    \end{enumerate}
\end{enumerate}
To exclude \(a\geq 2\) we first focus on the last condition: it excludes \(a=2\) and imposes \(e=1\) if \(a>2\). But if \(e=1\) we have by the second condition \(a=1\) or \(2\), thus \(a=2\), a contradiction. 
One easily checks that the conditions are satisfied if \(a=1\). 
\end{exa}

\begin{exa}
\label{exa_smooth_Fano_polytope}
In the toric case \(G=\bbG_m^n\), \(H=\{1\}\), we recover the definition of a \emph{smooth Fano polytope}: a polytope in \(M\otimes \bbQ\) with vertices in \(M\), which contains the origin as an interior point, and such that the collection of vertices of any facet is a basis of \(M\). 
\end{exa}

\subsection{Geometry of the associated Fano manifold}

For each of the locally factorial Fano  embeddings obtained, one can read off from the locally factorial \(G/H\)-reflexive polytope some aspects of its geometry, such as a description of its Picard group, Fano index, moment polytope, anticanonical degree and K-stability. This is possible thanks to the following combinatorial translations of these data for spherical varieties. 

Let \(X\) be a locally factorial Fano spherical embedding of \(G/H\), and let \(\Delta\) be the corresponding locally factorial \(G/H\)-reflexive polytope. 
Then it follows from \cite{Gagliardi_Hofscheier_2015_fano} that \(G\)-stable prime divisors in \(X\) correspond to vertices of \(\Delta\) which are in the valuation cone, and are not of the form \(\frac{\rho(D)}{m_D}\). 
Let \(\mathcal{I}\) denote this set of divisors, and let us denote by \(D_v\) the \(G\)-stable prime divisor associated to the vertex \(v\). 

From \cite{Brion_1989}, since \(X\) is locally factorial, the Picard group of \(X\) is generated by \(\mathcal{D}\cup \mathcal{I}\), and the relations are given by 
\[ 
\sum_{D\in \mathcal{D}} \langle\rho(D),\lambda\rangle D + \sum_{D_v\in \mathcal{I}} \langle v,\lambda\rangle D_v = 0 
\]
for \(\lambda\in M\). 

From \cite{Brion_1997}, we know a \(B\)-stable anticanonical divisor: it is given by 
\[ -K_X = \sum_{D\in \mathcal{D}} m_D D + \sum_{D_v\in \mathcal{I}} D_v \]
with the same notation \(m_D\) as in the previous section. 
Using these two results, we can easily recover the Fano index of \(X\) from its associated combinatorial data. 
Note that these results obviously predate Gagliardi and Hofscheier's characterization of Fano spherical embeddings and are used as essential ingredients in \cite{Gagliardi_Hofscheier_2015_fano}.

By \cite{Brion_1989} combined with \cite{Brion_1997}, we further know the anticanonical moment polytope \(\Delta^+(X,K_X^{-1})\) of \(X\) and the anticanonical degree \((K_X^{-1})^{\dim(X)}\) of \(X\). 
Namely, the moment polytope is given by 
\[ \Delta^+(X,K_X^{-1}) = \kappa + \{ m\in M\otimes \bbR \mid \forall D \quad \frac{\rho(D)}{m_D}(m)+1 \geq 0 \} \]
where \(\kappa\) still denotes the sum of all roots of the adapted parabolic, \(D\) runs over all \(B\)-stable prime divisors in \(X\), and a divisor is identified with its associated divisorial valuation.  
This is the translate by \(\kappa\) of the polytope dual to the locally factorial \(G/H\)-reflexive polytope \(\Omega\). 

Let \(R_{P^u}\) denote the set of positive roots of the unipotent radical of \(P\). 
The anticanonical degree is given by 
\[ \degree(X):=(K_X^{-1})^{(\dim(X))} = (\dim(X))! \int_{-\kappa+\Delta^+} \prod_{\alpha\in R_{P^u}}\frac{\left\{ \kappa+p, \alpha \right\} }{\left\{ \varrho, \alpha \right\}} \mathop{d\lambda}(p) \]
where \(\mathop{d\lambda}\) denotes the Lebesgue measure on \(M\otimes \bbR\) normalized by \(M\), \(\left\{,\right\}\) denote the Killing form and \(\varrho\) denotes the half sum of all positive roots of \(G\). 

Finally, it follows from \cite{Delcroix_2020} that \((X, K_X^{-1})\) is K-semistable if and only if the point 
\[ \int_{-\kappa + \Delta^+} p \prod_{\alpha\in R_{P^u}}\frac{\left\{ \kappa+p, \alpha \right\} }{\left\{ \varrho, \alpha \right\}} \mathop{d\lambda}(p) \in M\otimes \bbR \]
is in the closed convex cone generated by \(\Sigma\). 
Furthermore, it is K-stable, hence \(X\) admits a Kähler-Einstein metric, if and only if the above point is in the relative interior of the cone generated by \(\Sigma\). 

To make the above quantities easily computable, we will give an explicit description of the function  
\[ f: M\otimes \bbR \to \bbR, p \mapsto \prod_{\alpha\in R_{P^u}}\frac{\left\{ \kappa+p, \alpha \right\} }{\left\{ \varrho, \alpha \right\}} \]
when describing the combinatorial data of spherical homogeneous spaces. 

\section{Spherical homogeneous spaces of dimension four or less}
\label{section_homogeneous_spaces}

In this section, we classify all possible spherical homogeneous spaces of dimension four or less, regardless of whether or not they admit Fano embeddings. 

\subsection{Classification assumptions and first consequences}

For the purpose of the classification, we will make the following assumptions on \(G\) and \(G/H\). 

\begin{assumption}
\label{assumptions}
We assume that \(G=G^{sc}\times C\) where \(G^{sc}\) is a semisimple, simply connected group and \(C\simeq \bbG_m^n\) is a torus. 
We assume that the action of \(G^{sc}\) on \(G/H\) has finite kernel and that the action of \(C\) on \(G/H\) is faithful. 
\end{assumption}

We derive some first consequences of these assumptions for later use. 

\begin{rem}
\label{rem_rank_lower_bound}
From Proposition~\ref{local_structure_theorem}, we obtain that the rank of \(G/H\) is at least \(\dim(C)\) if Assumption~\ref{assumptions} is satisfied. 
\end{rem}

\begin{prop}
\label{horospherical_by_torus_dimension}
Under Assumption~\ref{assumptions}, assume furthermore that \(\dim(C)=\rank(G/H)\). 
Then \(G/H\) is horospherical. 
Conversely, if \(G/H\) is horospherical, then there exists a torus \(\tilde{C}\) such that \(\dim(C\times \tilde{C})=\dim(G/H)-\dim(G/P)\), and \(G\times C\times \tilde{C}\) acts transitively on \(G/H\) satisfying the above assumptions.  
\end{prop}

\begin{proof}
Assume that \(\dim(C)=\rank(G/H)\). 
Since \(C\) acts faithfully on \(G/H\), we have \(\dim(\Aut^G(G/H))\geq \dim(C)=\rank(G/H)\).  
By Proposition~\ref{prop_equivariant_automorphisms}, we deduce that \(\mathcal{V}\cap -\mathcal{V} = N\otimes \bbQ\), and thus by Corollary~\ref{cor_horospherical_by_valuation_cone}, that \(G/H\) is horospherical.

The converse follows from applying Corollary~\ref{cor_horospherical_by_valuation_cone} and Proposition~\ref{prop_equivariant_automorphisms} in the reverse order. 
\end{proof}

\begin{rem}
\label{rem_rank_group}
For a \(G\)-variety \(X\), any orbit of \(G\) in \(X\) is locally closed.  
If \(G=T\) is a torus and the action is faithful, then there is an orbit isomorphic to \(T\) as a variety. 
In particular, the dimension of \(T\) must be less than that of \(X\), and if \(X\) is normal and \(\dim(T)=\dim(X)\), then \(X\) is a toric variety. 
In particular, under Assumption~\ref{assumptions}, we have 
\(\rank(G)=\dim(C)+\rank(G^{sc})\leq \dim(G/H)\), and equality implies that \(X\) is toric under the action of a maximal torus of \(G\). 
\end{rem}

\subsection{Idea of the classification}

Let \(G=G^{sc}\times C\) be a connected complex reductive group and let \(G/H\) be a spherical homogeneous space satisfying Assumption~\ref{assumptions}. 
There are two ingredients in the classification of spherical homogeneous spaces of low dimension: 
\begin{enumerate}
    \item the consequences of the local structure theorem in Corollary~\ref{cor_local_structure_theorem}, which together with Assumption~\ref{assumptions} allows to put strong restrictions on the possible groups \(G\)
    \item the classification of Lie subalgebras of low rank complex semisimple groups \cite{Douglas_Repka_2016,Douglas_Repka_2016_so4} . 
\end{enumerate}
In order to use Corollary~\ref{cor_local_structure_theorem} to its full extent, we first recall the classification of projective homogeneous spaces of dimension less than four in Table~\ref{table:PHS}, which exhaust the possibilities for \(G/P\) where \(P\) is the adapted parabolic subgroup. 
The notation \(W\) denotes the variety of full flags in \(\bbC^3\), and \(Q^n\) denotes the quadric of dimension \(n\).

\begin{table}
\begin{equation*}
\begin{array}{ccccc}
\toprule
G & X=G/P & \dim & \pic & \degree \\
\midrule
\SL_2 & \bbP^1 & 1 & 1 & 2 \\
\midrule
\SL_3 & \bbP^2 & 2 & 1 & 9  \\
\SL_2^2  & \bbP^1\times \bbP^1  & 2 &2 & 8 \\
\midrule
\SL_3  & W & 3 & 2 & 48 \\
\Sp_4   & Q^3 & 3 & 1 & 54 \\
\Sp_4   & \bbP^3 & 3 & 1 & 64 \\
\SL_2^3  & \bbP^1\times \bbP^1\times \bbP^1 & 3 & 3 & 48\\
\SL_3\times \SL_2  & \bbP^2\times \bbP^1 & 3 & 2 & 54\\
\SL_4  & \bbP^3 & 3 & 1 & 64\\
\midrule
\SL_3\times \SL_2   & W\times \bbP^1 & 4 & 3 & 384 \\
\Sp_4\times \SL_2  & Q^3\times \bbP^1 &4 & 2 & 432 \\
\Sp_4\times \SL_2  & \bbP^3\times \bbP^1 & 4 & 2 & 512 \\
\SL_4  & Q^4 & 4 & 1 & 512 \\
\SL_2^4  & \bbP^1\times \bbP^1\times\bbP^1\times \bbP^1 & 4 & 4 & 384 \\
\SL_3\times \SL_2^2  & \bbP^2\times \bbP^1 \times \bbP^1 & 4 & 3 & 432\\
\SL_3^2  & \bbP^2\times \bbP^2 & 4 & 2 & 486 \\
\SL_4\times \SL_2  & \bbP^3\times \bbP^1 & 4 & 2 & 512 \\
\SL_5  & \bbP^4 & 4 & 1 & 625 \\
\bottomrule 
\end{array}
\end{equation*}
\caption{Projective homogeneous spaces of dimension \(\leq 4\)}
\label{table:PHS}
\end{table}

As a first consequence of Corollary~\ref{cor_local_structure_theorem}, we know that if \(\dim(G/H)\leq 4\), then \(G^{sc}\) is one of the groups appearing in the first column of Table~\ref{table:PHS}. 
More precisely, and more generally, we have the following consequence on spherical varieties of arbitrary dimension, but with rank close to the dimension. 

\begin{prop}
Assume that \(\rank(G/H)=\dim(G/H)\). 
Then \(P=G=C\) and \(H\) is trivial.
\end{prop}

\begin{prop}
Assume that \(\rank(G/H)+1=\dim(G/H)\). 
Then \(G^{sc}=\SL_2\) and the adapted parabolic subgroup is equal to the Borel subgroup \(P=B\). 
\end{prop}


\begin{prop}
Assume that 
\(\rank(G/H)+2=\dim(G/H)\). 
Then \(G^{sc}\in \{\SL_3,\SL_2^2\}\). 
Moreover, if \(G^{sc}=\SL_3\) then the adapted parabolic \(P\) is the maximal parabolic subgroup (there is only one choice up to outer automorphism), and if \(G^{sc}=\SL_2^2\), then \(P=B\). 
\end{prop}

\begin{prop}
\label{rank+3=dim}
Assume that 
\(\rank(G/H)+3=\dim(G/H)\). 
Then 
\[ G^{sc}\in \{\SL_4,\Sp_4,\SL_3\times \SL_2,\SL_2^3,\SL_3\}. \]
\end{prop}

One could go on further using the classification of projective homogeneous spaces of higher dimensions, but as the difference between the rank and the dimension grows larger, the number of possibilities grows very fast.  

Under the assumption that \(\dim(G/H)\leq 4\), a spherical homogeneous space which satisfies \(\rank(G/H)+3=\dim(G/H)\) is of rank zero or one. 
Since the classification of rank one spherical homogeneous spaces of arbitrary dimensions is well known \cite{Akhiezer_1983}, we can use it to treat this case. 
In the following subsections, we deal with the homogeneous spaces under groups with \(G^{sc}\in \{\SL_2,\SL_2^2,\SL_3\}\), then with the remaining rank one cases. 

\subsection{Spherical homogeneous spaces with \(G^{sc}=\SL_2\)}

We wish to classify spherical subgroups of \(\SL_2\times \bbG_m^n\) up to conjugation, with \(n\leq 3\) (see Remark~\ref{rem_rank_group}). 
Since it is not significantly different to treat the general case, we allow \(n>3\) as well. 

\begin{rem}
This is essentially treated in the case \(n=1\) in \cite{Nguyen_van_der_Put_Top_2008}, where all algebraic subgroups of \(\GL_2\) up to conjugation are classified.
We note however that there is a minor mistake in the case \(H=\gamma(D_{\infty})\) in the notations of that paper, as the group \(D_{\infty}\) itself does not appear in the classification. 
\end{rem}

\begin{prop}
\label{prop_Gsc=SL2}
Let \(H\) be an algebraic subgroup of \(\SL_2\times \bbG_m^n\) whose projection to the \(\SL_2\) factor is of positive dimension. 
Then \(H\) is of one of the following forms:
\begin{enumerate}
    \item \(H=\SL_2\times H_2\) where \(H_2\) is an algebraic subgroup of \(\bbG_m^n\),
    \item \(H\subset B\) is a horospherical subgroup, 
    \item \(H\) is an algebraic subgroup of \(T\), 
    \item \(H=\left\langle T_1\times K_2, \left(\begin{bmatrix} 0& -1\\ 1&0\end{bmatrix},c\right)\right\rangle\) where \(K_2\) is an algebraic subgroup of \(\bbG_m^n\), \(c\) is a primitive \(2^l\)-th root of unity in \(\bbG_m^n\) for some \(l\geq 0\), \(c^2\in K_2\) and \(c\notin K_2\) unless \(l=0\). 
\end{enumerate}
\end{prop}

Since we will work a bit at the Lie algebra level, we recall the following fact on Lie algebras of algebraic subgroups \cite{Chevalley_Tuan_1945,Chevalley_1943} (actually, it is a property of the slightly larger class of quasi-algebraic Lie subalgebras).
\begin{rem}
\label{rem_quasi_algebraic}
Let \(H\) be an algebraic subgroup of \(G\). 
For any element in its Lie algebra \(h\in \mathfrak{h}\) with Jordan-Chevalley decomposition \(h=t+u\) (where \(t\) is a semisimple element of \(\mathfrak{g}\) and \(u\) is a nilpotent element of \(\mathfrak{g}\)), we have \(t\in \mathfrak{h}\) and \(u\in \mathfrak{h}\).  
\end{rem}

\begin{proof}
We recalled in Example~\ref{example_sl2} the classification of spherical subgroups of \(\SL_2\) up to conjugation, that is, all positive dimensional algebraic subgroups of \(\SL_2\). 
As a consequence an algebraic subgroup \(H\) of \(\SL_2\times \bbG_m^n\) is spherical if and only if its image \(H_1\) under the projection to \(\SL_2\) is positive dimensional.  

We first work at the Lie algebra level. Recall that \(\mathfrak{sl}_2\) is generated as a complex vector space by one semisimple element \(e_0\) and two nilpotent elements \(e_+\) and \(e_-\) satisfying the relation \(e_0=[e_+,e_-]\). 
We identify these elements with elements of \(\mathfrak{g}\), and we may assume that \(e_+\in \mathfrak{b}\). 

Assume that \(H_1=\SL_2\). 
Then by Remark~\ref{rem_quasi_algebraic}, \(e_+\) and \(e_-\) are in \(\mathfrak{h}\), thus \(e_0=[e_+,e_-]\) as well. 
Hence \(\SL_2\times\{1\}\subset H\subset \SL_2\times \bbG_m^n\), that is, \(H=\SL_2\times H_2\) for some algebraic subgroup of \(\bbG_m^n\).  

Assume now that \(H_1\) is a Borel subgroup of \(\SL_2\). 
Then by Remark~\ref{rem_quasi_algebraic}, \(e_+\in \mathfrak{h}\), thus \(B^u\subset H\subset B\), and \(H\) is a horospherical subgroup. 

If \(H_1=T_1\), then \(H\) is a subgroup of the torus \(T=T_1\times \bbG_m^n\). 

Finally we deal with the case \(H_1=N_{\SL_2}(T_1)\). 
Fix for the moment \(a\in \bbC^*\). 
Set 
\[n=\begin{bmatrix}0&-1\\1&0\end{bmatrix} \qquad \text{and} \qquad  t=\begin{bmatrix}a&0\\0&a^{-1}\end{bmatrix}.\] 
As \(H_1=N_{\SL_2}(T_1)\), there exists \(b\) and \(c\in \bbG_m^n\) such that \((t,b)\in H\) and \((n,c)\in H\). 
We then have the following commutator in \(H\):
\[ (t,b)(n,c)(t,b)^{-1}(n,c)^{-1} = (tnt^{-1}n^{-1},1) = \left(\begin{bmatrix}a^2&0\\0&a^{-2}\end{bmatrix},1\right). \] 
Since \(a\mapsto a^2\) is surjective on \(\bbC^*\), we conclude that \(T_1\times \{1\}\subset H\subset N_{\SL_2}(T_1)\times \bbG_m^n\). 

Consider the kernel \(K\) of the surjective morphism \(H\to N_{\SL_2}(T_1)/T_1 \simeq \bbZ/2\bbZ\). 
It is of the form \(T_1\times K_2\) for some algebraic subgroup \(K_2\) of \(\bbG_m^n\). 
The group \(H\) is fully determined by one preimage \(\left(n,c\right)\) of the non-trivial element in \(\bbZ/2\bbZ\): \(H=\langle T_1\times K_2 , \left(n,c\right) \rangle\) is the subgroup generated by \(T_1\times K_2\) and this preimage.  
Note that \(c\) is the square root of an element of \(K_2\). If it is an element of \(K_2\), then actually \(H=N_{\SL_2}(T_1)\times K_2\) and we may take \(c=1\). If \(c\) is not an element of \(K_2\), then it is of finite order since \(K_2\) is algebraic. In that case, to avoid redundancy, we may as well assume that \(c\) is a primitive \(2^l\)-th root of unity in \(\bbG_m^n\) for some \(l> 0\). 
\end{proof}

\begin{cor}
\label{cor_all_H_SL2}
Under Assumption~\ref{assumptions}, if \(G=\SL_2\times \bbG_m^n\) with \(n\geq 1\), then up to equivariant isomorphism, there exists a character \(\chi_1\) and thus a decomposition \(\bbG_m^n=\bbG_m\oplus \ker\chi_1|_{\bbG_m^n} = \bbG_m\times \bbG_m^{n-1}\) such that either:
\begin{enumerate}
    \item \label{cor_all_H_SL2:T} \(H= \ker((a_1\varpi_1 + \chi_1)|_{T_1\times \bbG_m})\times \{1_{\bbG_m^{n-1}}\}\) for some \(a \bbZ_{\geq 0}\)
    \item \label{cor_all_H_SL2:N} \(H=\left\langle T_1\times \{1_{\bbG_m}\}, \left(\begin{bmatrix} 0& -1\\ 1&0\end{bmatrix},c\right)\right\rangle \times \{1_{\bbG_m^{n-1}}\}\) where \(c\in \{1,-1\}\) 
    \item \label{cor_all_H_SL2:B} \(H= \ker((a_1\varpi_1 + \chi_1)|_{B_1\times \bbG_m})\times \{1_{\bbG_m^{n-1}}\}\) for some \(a\in \bbZ_{\geq 0}\)
\end{enumerate}
\end{cor}

\begin{rem}
\label{rem_SL2_non-uniqueness_of_chi}
It should be noted that the character \(\chi_1\) and thus the decomposition \(\bbG_m^n= \bbG_m\times \bbG_m^{n-1}\) are not uniquely determined up to equivariant isomorphism. 
As follows from the proof, if \(a_1>0\), then one may replace \(\chi_1\) with \(\chi_1+\sum_{i=2}^na_i\chi_i\) for any \((n-1)\)-tuple of integers \((a_i)\), where \((\chi_2,\ldots,\chi_n)\) is a basis of characters for the factor \(\bbG_m^{n-1}\). 
If \(a=0\) or \(c=+1\), then we may actually replace \(\chi_1\) by any primitive character of \(X^*(\bbG_m^n)\): the spherical homogeneous space is the product of an \(\SL_2\)-spherical homogeneous space with a factor, such that the factors of the group act only on the respective factors of the homogeneous space. 
\end{rem}

\begin{proof}
We examine successively all cases of Proposition~\ref{prop_Gsc=SL2}. 

If \(H\) satisfies the first case in Proposition~\ref{prop_Gsc=SL2}, then Assumption~\ref{assumptions} is not satisfied. 

In the fourth case, Assumption~\ref{assumptions} is satisfied if and only if \(l=0\) or \(l=1\). 
If \(l=1\), choose a primitive character \(\lambda_1\) of \(\bbG_m^n\) such that \(\lambda_1(c)=-1\), then complete it into a \(\bbZ\)-basis \((\lambda_1,\lambda_2,\ldots,\lambda_n)\) of \(X^*(\bbG_m^n)\) such that \(\lambda_i(c)=1\) for \(i\geq 2\). Then under the automorphism \(\SL_2\times \bbG_m^n \to \SL_2\times \bbG_m^n\) given by \((\text{id}_{\SL_2},\lambda_1,\ldots,\lambda_n)\), we arrive to item (\ref{cor_all_H_SL2:N}) of the statement. 

In the third case of Proposition~\ref{prop_Gsc=SL2}, \(H\) writes \(H=\bigcap_{i=1}^k \ker(\lambda_i)\) for a family of \(\bbZ\)-linearly independent characters \(\lambda_i\in X^*(T)\). 
For the action of \(\bbG_m^n\) on \(G/H\) to be faithful, the family \(\{\chi_i\}_i\subset X^*(\bbG_m^n)\) of projections of the \(\lambda_i\in X^*(T)=X^*(T_1)\oplus X^*(\bbG_m^n)\) must be a \(\bbZ\)-basis of \(X^*(\bbG_m^n)\). 
Hence the \(\lambda_i\) write as \(\lambda_i=a_i\varpi_1\oplus \chi_i\) for some \(a_i\in \bbZ\). 
However, we may write \(H=\bigcap_{i=1}^k \ker(\lambda_i')\) for any other \(\bbZ\)-basis \(\{\lambda_i'\}\) of the \(\bbZ\)-module generated by the \(\lambda_i\). 
Since the action of \(\GL_{n}(\bbZ)\) is transitive on primitive vectors in \(\bbZ^{n}\simeq X^*(\bbG_m^n)\), we may find a such a \(\bbZ\)-basis \(\{\lambda_i'\}\) with \(\lambda_1=a\varpi_1 + \chi_1'\) and \(\lambda_i=\chi'_i\) for \(i\geq 2\), for some \(a\in \bbN\) and some \(\bbZ\)-basis \(\{\chi_i'\}\) of \(X^*(C)\). 
The automorphism \(\SL_2\times \bbG_m^n \to \SL_2\times \bbG_m^n\) given by \((\text{id}_{\SL_2},\chi_1',\ldots,\chi_n')\) takes us to item~(\ref{cor_all_H_SL2:T}) of the statement. 

The horospherical case is dealt with in the same way and leads to item~(\ref{cor_all_H_SL2:B}).  
\end{proof}

\subsection{Spherical homogeneous spaces with \(G^{sc}=\SL_2^2\)}

\begin{prop}
\label{prop_H_SL22_notPI}
Let \(G=\SL_2^2\times \bbG_m^n\) and \(G/H\) satisfy Assumption~\ref{assumptions}, with \(\dim(G/H)\leq 4\).  
Then we are in one of the following situations (up to \(G\)-equivariant isomorphism):
\begin{enumerate}
    \item \label{prop_H_SL22_notPI:n=0} \(n=0\) and \(H\) is one of the following: 
    \begin{enumerate}
        \item \label{prop_H_SL22_notPI:n=0:SL2} \(\diag(\SL_2)\)
        \item \label{prop_H_SL22_notPI:n=0:PGL2} \(N(\diag(\SL_2))\)
        \item \label{prop_H_SL22_notPI:n=0:diagB} \(\diag(B_1)\)
        \item  \label{prop_H_SL22_notPI:n=0:NdiagB} \(N(\diag(B_1))\) 
        \item \label{prop_H_SL22_notPI:n=0:TxT} \(T_1\times T_2\)
        \item \label{prop_H_SL22_notPI:n=0:NxT} \(N(T_1)\times T_2\) 
        \item \label{prop_H_SL22_notPI:n=0:diagN} \(\diag N(T_1)\)
        \item \label{prop_H_SL22_notPI:n=0:NxN} \(N(T_1)\times N(T_2)\) 
    \end{enumerate} 
    \item \label{prop_H_SL22_notPI:n=1} \(n=1\) and \(H\) is one of the following: 
    \begin{enumerate}
        \item \label{prop_H_SL22_notPI:n=1:SL2xC} \(\diag(\SL_2)\times\{1\}\)
        \item \label{prop_H_SL22_notPI:n=1:GL2} \(\langle \diag(\SL_2),\left(I_2, -I_2, -1\right)\rangle\)
        \item \label{prop_H_SL22_notPI:n=1:PGL2xC} \(N(\diag(\SL_2))\times\{1\}\)
    \end{enumerate}
    \item \(G/H\) is obtained by parabolic induction.
\end{enumerate}
\end{prop}

\begin{proof}
By Assumption~\ref{assumptions} and Corollary~\ref{cor_local_structure_theorem}, the adapted parabolic subgroup is \(P=B_1\times B_2\times \bbG_m^n\). 
Then \(\dim(G/P)=2\) and by Corollary~\ref{cor_local_structure_theorem} again, \(\rank(G/H)=\dim(G/H)-2\). 
Furthermore, by \(\dim(G/H)\leq 4\), Remark~\ref{rem_rank_lower_bound} and Remark~\ref{rem_rank_group}, \(n\leq \rank(G/H)= \dim(G/H)-2\leq 2\).

We first deal with the case \(n=0\). 
The Lie subalgebras, up to conjugation, of \(\mathfrak{sl}_2\oplus\mathfrak{sl}_2\), are classified in  \cite{Douglas_Repka_2016_so4}. 
There are several restrictions to take into account here. 
First, the Lie algebra of a spherical subgroup of \(\SL_2^2\) must be of dimension at least two, and more precisely, the projection to each summand \(\mathfrak{sl}_2\) must have positive dimension. 
Second, by Remark~\ref{rem_quasi_algebraic}, the subalgebras \(K_2^3\) and \(K_2^5\) in Douglas and Repka's notations cannot be the Lie algebras of spherical subgroups.
Third, by Assumption~\ref{assumptions}, the Lie algebra of \(H\) cannot contain one of the \(\mathfrak{sl}_2\) summand. 

Finally, using Remark~\ref{induction_from_Lie}, we deduce from Douglas and Repka's classification that the possible Lie algebras of spherical subgroups of \(\SL_2^2\) that are \emph{not} obtained by parabolic induction are the following: 
\(\mathfrak{t}_1\oplus\mathfrak{t}_2\), \(\diag(\mathfrak{sl}_2)\) and \(\diag(\mathfrak{b}_1)\). 
These yield item~(\ref{prop_H_SL22_notPI:n=0}) in the statement. 

Assume now that \(n=1\). 
We then have \(4\geq \dim(G/H)\geq 3\). 

If \(\dim(G/H)=3\), then \(\rank(G/H)=1\). 
On the other hand, since the \(\bbG_m\) factor acts faithfully, we have \(\dim(\Aut^G(G/H))\geq 1\), and by Proposition~\ref{horospherical_by_torus_dimension}, \(G/H\) is horospherical. 
In particular, it is obtained by parabolic induction. 

Assume now that \(\dim(G/H)=4\), thus \(H\) is a spherical subgroup of \(\SL_2^2\times \bbG_m\) of dimension \(3\). 
Let \(H_1\) be the projection of \(H\) to \(\SL_2^2\). 
It is a spherical subgroup of \(\SL_2^2\).  
In view of the description in the previous case, there are few situations to consider: 
\begin{itemize}
    \item If \(H_1\) is a parabolic induction, then by Remark~\ref{rem_quasi_algebraic}, \(H\) is a parabolic induction as well. 
    \item If \(\mathfrak{h}_1=\mathfrak{t}_1\oplus \mathfrak{t}_2\), then \(H\subset N(T_1)\times N(T_2)\times \bbG_m\) and, by dimension, \(T_1\times T_2\times \bbG_m\subset H\). But then the action of \(\bbG_m\) on \(G/H\) is not faithful.  
    \item Similarly, if \(\mathfrak{h}_1=\diag(\mathfrak{b}_1)\) then \(H\subset N(\diag(B_1))\times \bbG_m\) but by dimension, we then have \(\diag(B_1)\times \bbG_m\subset H\). Again, in this case, the action of \(\bbG_m\) on \(G/H\) is not faithful. 
    \item Finally, we deal with the case \(\mathfrak{h}_1=\diag(\mathfrak{sl}_2)\). 
    Denote, as in the proof of Proposition~\ref{prop_Gsc=SL2}, by \((e_+,e_-,e_0)\) a basis of \(\mathfrak{sl}_2\) such that \(e_{\pm}\) are nilpotent, \(e_0\) is semisimple and \([e_+,e_-]=e_0\). 
    Since \((e_{+},e_{+})\) \((e_-,e_-)\in \mathfrak{h}_1\), by Remark~\ref{rem_quasi_algebraic}, we actually have \((e_+,e_+,0)\) and \((e_-,e_-,0)\in \mathfrak{h}\), and their brackets as well. 
    Thus \(\diag(\SL_2)\times \{1\} \subset H\). 
    Under the assumption of faithful action of the \(\bbG_m\) factor, there are only three possibilities for \(H\) which are the three in item~(\ref{prop_H_SL22_notPI:n=1}) of the statement.
\end{itemize}

Finally, if \(n=2\) then \(\dim(G/H)=4\) and \(\rank(G/H)=2\). 
By Proposition~\ref{horospherical_by_torus_dimension} again \(G/H\) is horospherical, hence obtained by parabolic induction. 
\end{proof}

\begin{prop}
\label{prop_H_SL_22_PI}
Let \(G=\SL_2^2\times \bbG_m^n\) and let \(G/H\) satisfy Assumption~\ref{assumptions}, with \(\dim(G/H)\leq 4\).  
Assume that \(G/H\) is obtained by parabolic induction, and that the image of \(G\) in \(\Aut(G/H)\) contains \(\Aut^{G,0}(G/H)\).
Then up to \(G\)-equivariant isomorphism, we have the following possibilities: 
\begin{enumerate}
    \item \label{prop_H_SL_22_PI:B} \(n=0\) and \(H=B^-\)
    \item \label{prop_H_SL_22_PI:BxT} \(n=0\) and \(H=B_1^-\times T_2\) 
    \item \label{prop_H_SL_22_PI:BxN} \(n=0\) and \(H=B_1^-\times N(T_2)\)
    \item \label{prop_H_SL_22_PI:rk1horo} \(n=1\) and \(H=\ker((a_1\varpi_1+a_2\varpi_2+\chi_1)|_{B^-})\) for some \(a_1\), \(a_2\in \bbZ\) with \(a_1\geq\vert a_2\rvert\). 
    \item \label{prop_H_SL_22_PI:PIT} \(n=1\) and \(H=\ker((a_1\varpi_1+a_2\varpi_2+\chi_1|)_{T_1\times B_2^-\times \bbG_m})\) with \(a_1\in \bbZ_{\geq 0}\) and \(a_2\in \bbZ\)
    \item \label{prop_H_SL_22_PI:PIN} \(n=1\) and \(H=\left\langle T_1\times\ker((a_2\varpi_2+\chi_1)|_{B_2^-\times \bbG_m}), \left(\begin{bmatrix}0&-1\\1&0\end{bmatrix},\begin{bmatrix}\xi&0\\0&\xi^{-1}\end{bmatrix},\varepsilon\right)\right\rangle\) where \(a_2\in \bbZ_{\geq 0}\), \(\varepsilon=\pm 1\), \(\xi\) is a primitive \(2^l\)-th root of \(1\), and either \(2^l\) divides \(a_2\) if \(\varepsilon=-1\), or \(\varepsilon=1\), \(2^{l}\) divides \(2a_2\) and \(2^l\) does not divide \(a_2\) unless \(l=0\).
    \item \label{prop_H_SL_22_PI:rk2horo} \(n=2\) and \(H=\ker((a_1\varpi_1+a_2\varpi_2+\chi_1)|_{B^-})\cap \ker((b_2\varpi_2+\chi_2)|_{B^-})\) where either \(a_2=b_2=0\) and \(0\leq a_1\), or \(0<a_1\leq a_2\wedge b_2\), and \(0\leq a_2 <b_2\). 
\end{enumerate}
\end{prop}

\begin{rem}
\label{rem_SL_22_non-uniqueness_of_chi}
As in Remark~\ref{rem_SL2_non-uniqueness_of_chi}, we note that in case \ref{prop_H_SL_22_PI:rk2horo} of Proposition~\ref{prop_H_SL_22_PI}, if \(a_2=b_2=0\), we may replace \(a_1\varpi_1+\chi_1\) with \(a_1\varpi_1+\chi_1+\lambda\chi_2\) for any \(\lambda\in \bbZ\) without changing the subgroup \(H\). 
This is in fact directly explained by Remark~\ref{rem_SL2_non-uniqueness_of_chi}, since in this case, the homogeneous space is a product of \(\bbP^1\) with a  horospherical homogeneous space under \(\SL_2\times \bbG_m^2\). 
If furthermore \(a_1=0\), then the homogeneous space is a product of \(\bbP^1\times \bbP^1\) with a torus and obviously any basis \((\chi_1,\chi_2)\) may be chosen. 

In case \ref{prop_H_SL_22_PI:rk1horo} of Proposition~\ref{prop_H_SL_22_PI}, the conditions on \(a_i\) still allow some \(G\)-equivariant isomorphisms as well: in case of equality \(a_1=a_2\), one may exchange the two \(\SL_2\) factors. 
\end{rem}

\begin{proof}
There are two possibilities for the proper parabolic subgroup involved in the parabolic induction: a Borel subgroup or a maximal parabolic subgroup (up to exterior automorphism there is only one). 

We start with the first case. 
Since any quotient of a Borel subgroup is a torus, we only obtain horospherical subgroups.  
In view of our assumption on \(G\), \(G/H\) and \(\Aut^{G,0}(G/H)\), we have the following possibilities. 

If \(n=0\) then \(H=B^-\). 

If \(n=1\) then \(H\) is the kernel of a character \(a_1\varpi_1+a_2\varpi_2 + a_3 \chi_1\) of \(B^-\) for some integers \(a_1\), \(a_2\) and \(a_3\). 
By Assumption~\ref{assumptions}, we have \(a_3=\pm 1\), and up to \(G\)-equivariant isomorphism, we can further assume that \(a_3=1\) and \(a_1\geq \lvert a_2\rvert\). 

If \(n=2\) then \(H\) is the intersection of the kernels of two characters \(a_1\varpi_1+a_2\varpi_2+\chi_1\) and \(b_1\varpi_1+b_2\varpi_2+\chi_2\) of \(B^-\) for some integers \(a_1\), \(a_2\), \(b_1\), \(b_2\), and by Assumption~\ref{assumptions}, \((\chi_1,\chi_2)\) is a basis of \(X^*(\bbG_m^2)\). 
We may replace \((\chi_1,\chi_2)\) with another basis of \(X^*(\bbG_m^2)\), which corresponds to replacing \(\begin{bmatrix}a_1 & b_1\\a_2 & b_2\end{bmatrix}\) with an element of its \(\GL_2(\bbZ)\)-orbit on the right. We may thus assume that this matrix is in Hermite normal form. 
We may also exchange the two \(\SL_2\) factors or reorder the two characters. 
We may thus assume either that \(a_2=b_2=b_1=0\) and \(0\leq a_1\), or that \(b_1=0\), \(0<a_1\leq a_2\wedge b_2\), and \(0\leq a_2 <b_2\). 

We now consider the case of a parabolic induction with respect to a maximal parabolic subgroup. 
Then the basis of the parabolic induction is \(\bbP^1\) and the fiber \(G_0/H_0\) is a spherical homogeneous space of dimension less than three, with \(G_0\) a quotient of \(\SL_2\times\bbG_m^n\). 
We reason according to the rank. 
With Assumption~\ref{assumptions}, the only possibilities for the rank are one and two. Indeed, the adapted parabolic must be the Borel subgroup since it cannot contain a simple factor (which would act trivially) hence the rank is at most two. 
For the same reasons, \(G/H\) cannot be horospherical: if it were, then one of the \(\SL_2\) factor would be in \(H\), thus act trivially since this factor is a normal subgroup of \(G\). 
In particular it cannot have rank zero. 

Assume that the rank is one. 
Since \(G/H\) is not horospherical, we must have \(n=0\) and \(\Aut^{G,0}(G/H)=\{1\}\). 
The fiber \(G_0/H_0\) of the parabolic induction must thus be a spherical homogeneous space \(G_0/H_0\) of rank one and dimension at least two and we can assume \(G_0=\SL_2\). 
This means that \(H=T_1\times B_2\), or \(H=N_{\SL_2}(T_1)\times B_2\). 

Assume finally that the rank is two (hence \(n=0\) or \(1\)). 
Since \(G/H\) is not horospherical, the fiber \(G_0/H_0\) must be a spherical homogeneous space of dimension three under \(G_0=\SL_2\times \bbG_m^{1+n}\) (one could of course take a quotient but it is more convenient to consider the largest possible group here). 
In particular, \(\dim (H_0)=n+1\), and by Proposition~\ref{prop_Gsc=SL2}, we know the possible \(H_0\). 
Furthermore, we note that \(H_0\) cannot contain the \(\SL_2\)-factor or a maximal unipotent subgroup of this \(\SL_2\)-factor, otherwise either the faithfulness assumptions would be violated, or we would actually recover a horospherical example of the previous case obtained by parabolic induction with respect to a Borel subgroup. 
In view of the remaining cases in Proposition~\ref{prop_Gsc=SL2}, we have \(\dim \Aut^{G,0}(G/H) \geq 1\) hence \(n=1\) since we assume that the image of \(G\) in \(\Aut(G/H)\) contains \(\Aut^{G,0}(G/H)\).  

In view of Assumption~\ref{assumptions}, we finally obtain that either \(H_0=\ker((a_1\varpi_1+a_2\varpi_2+\chi_1)|_{T1\times T_2\times \bbG_m})\) with \(a_1\in \bbZ_{\geq 0}\) and \(a_2\in \bbZ\), or \(H_0=\left\langle T_1\times K_2, \left(\begin{bmatrix} 0& -1\\ 1&0\end{bmatrix},c\right)\right\rangle\) where \(K_2=\ker((a_2\varpi_2+\chi_1)|_{T_2\times \bbG_m})\) for some \(a_2\in \bbZ_{\geq 0}\), and \(c=(d,\varepsilon)\in T_2\times \bbG_m\) where \(d\) is a primitive \(2^l\)-th root of unity, and either \(\varepsilon=1\), \(2^{l-1}\) divides \(a_2\) but \(2^l\) does not, or \(\varepsilon=-1\) and \(2^{l}\) divides \(a_2\). 
\end{proof}

\subsection{Spherical homogeneous spaces with \(G^{sc}=\SL_3\)}

\begin{prop}
\label{prop_S(GL2xGL1)}
Let \(G=\SL_3\times \bbG_m^n\) and \(G/H\) satisfy Assumption~\ref{assumptions}, with \(\dim(G/H)\leq 4\).  
Then either \(n=0\) and \(H=S(\GL_2\times\GL_1)\subset \SL_3\) or \(G/H\) is obtained by parabolic induction. 
\end{prop}

\begin{proof}
Depending on the adapted parabolic subgroup \(P\), we have \(\dim(G/P)=2\) or \(3\). 
As a consequence, \(n\leq 2\), and if \(n=2\), by Proposition~\ref{horospherical_by_torus_dimension}, \(G/H\) is horospherical, hence obtained by parabolic induction. 

Lie subalgebras of \(\mathfrak{sl}_3\) up to conjugation have been classified in  \cite{Douglas_Repka_2016}. 
From examining the list of Lie subalgebras of dimension at least four, keeping in mind Remark~\ref{rem_quasi_algebraic}, we see that either \(G/H\) is obtained by parabolic induction, or \(\mathfrak{h}\simeq \mathfrak{sl}_2\oplus\bbC\) embedded as a Lie subalgebra of block-diagonal matrices in \(\mathfrak{sl}_3\) with a block of size two and a block of size one. 
Since the image of this subalgebra in \(\SL_3\) by the exponential map is \( S(\GL_2\times\GL_1)\) and this group is self-normalizing, we obtain the statement for \(n=0\). 

It remains to show that if \(n=1\), \(G/H\) is obtained by parabolic induction. 
Assume that \(H\) is a spherical subgroup of \(\SL_3\times \bbG_m\). 
If \(\dim(G/H)=3\) then \(1=n\leq \rank(G/H)\leq 3-\dim(G/P)\leq 1\) hence by Proposition~\ref{horospherical_by_torus_dimension}, \(G/H\) is horospherical. 
Assume now that \(\dim(G/H)=4\), hence \(\dim(H)=5\). 
Let \(H_1\) be the projection of \(H\) to the \(\SL_3\) factor. 
In view of the case \(n=0\), either \(H_1\) is a parabolic induction, in which case \(H\) is as well by Remark~\ref{rem_quasi_algebraic}, or \(H_1=S(\GL_2\times \GL_1)\) (up to conjugation). 
In the latter case, \(H\) is a \(5\)-dimensional subgroup of \(S(\GL_2\times \GL_1)\times \bbG_m\), hence equal to \(S(\GL_2\times \GL_1)\times \bbG_m\). 
This subgroup however, does not satisfy Assumption~\ref{assumptions}. 
\end{proof}

\begin{prop}
\label{prop_H_SL_3_PI}
Let \(G=\SL_3\times \bbG_m^n\) and let \(G/H\) satisfy Assumption~\ref{assumptions}, with \(\dim(G/H)\leq 4\).  
Assume that \(G/H\) is obtained by parabolic induction, and that the image of \(G\) in \(\Aut(G/H)\) contains \(\Aut^{G,0}(G/H)\). 
Then up to \(G\)-equivariant isomorphism, we are in one of the following cases:
\begin{enumerate}
    \item \label{prop_H_SL_3_PI:B} \(n=0\) and \(H=B^-\) 
    \item \label{prop_H_SL_3_PI:Q} \(n=0\) and \(H=Q\)
    \item \label{prop_H_SL_3_PI:horosym} \(n=0\) and \(H = \langle Q^u, T \rangle \)
    \item \label{prop_H_SL_3_PI:Nhorosym} \(n=0\) and \(H = N(\langle Q^u, T \rangle) \) 
    \item \label{prop_H_SL_3_PI:rk1horoQ} \(n=1\) and \(H=\ker(a_1\varpi_1+\chi_1)|_Q\) for some \(a_1\in \bbZ_{\geq 0}\)
    \item \label{prop_H_SL_3_PI:rk1horoB} \(n=1\) and \(H=\ker(a_1\varpi_1+a_2\varpi_2+\chi_1)|_{B^-}\) for some \(a_1\) and \(a_2\in \bbZ\) with \(a_1\geq\lvert a_2\rvert\)
    \item \label{prop_H_SL_3_PI:rk2horo} \(n=2\) and \(H=\ker((a_1\varpi_1+\chi_1)|_Q)\cap \ker(\chi_2|_Q)\) for some \(a_1\in \bbZ_{\geq 0}\),
\end{enumerate}
\end{prop}

\begin{rem}
\label{rem_SL3_non-uniqueness_of_chi}
As in Remark~\ref{rem_SL2_non-uniqueness_of_chi}, we note that, in case \ref{prop_H_SL_3_PI:rk2horo} of Proposition~\ref{prop_H_SL_3_PI}, we may replace \(\chi_1\) with \(\chi_1+\lambda \chi_2\) for any \(\lambda\in \bbZ\) without changing the subgroup \(H\). 
Again, if \(a_1=0\), the homogeneous space is the product of \(\bbP^2\) with a torus \(\bbG_m^2\). Then the basis \((\chi_1,\chi_2)\) of \(X^*(\bbG_m^2)\) can be chosen arbitrarily. 
\end{rem}

\begin{proof}
Up to \(G\)-equivariant isomorphism, there are two possibilities for the parabolic subgroup involved in the parabolic induction: \(B^-\) and \(Q\).  

Let us first deal with the case when it is \(B^-\). 
Since any reductive quotient of \(B^-\) is a torus, \(G/H\) is horospherical. 
Furthermore, since \(G/B^-\) is of dimension three, the rank of \(G/H\) is zero or one. 
If the rank is zero, then \(H=B^-\). 
If the rank is one, then \(n=1\) since \(\Aut^{G,0}(G/H)\) contains a one dimensional torus, and \(H\) is the kernel of a single character of \(B^-\), which may be written as \(a_1\varpi_1+a_2\varpi_2+c_1\chi_1\) for some integers \(a_1\), \(a_2\), \(c_1\). 
Assumption~\ref{assumptions} show that \(c_1=\pm 1\), and up to \(G\)-equivariant isomorphism, we can reorder the fundamental weights. 
Hence we may assume \(a_1\geq \lvert a_2\rvert\) and \(c_1=1\). 

Let us now deal with parabolic inductions with respect to the parabolic subgroup \(Q\). 
Since \(G/Q\simeq \bbP^2\), the rank may be zero (then \(H=Q\)), one, or two. 

If the rank is two, then since the fiber \(G_0/H_0\) of the projection \(G/H\to G/Q\) is of dimension two, it is a torus. 
Hence \(G/H\) is horospherical, and under our assumptions, \(n=2\). 
The subgroup \(H\) is then defined as the intersection of the kernels of two \(\bbZ\)-independent characters \(a\varpi_1+\chi_1\) and \(b\varpi_1+\chi_2\) of \(Q\), under Assumption~\ref{assumptions}. 
Since the action of \(\GL_2(\bbZ)\) is transitive on primitive vectors in \(\bbZ^2\), up to \(G\)-equivariant isomorphism, we may assume that \(a\in \bbN\) and \(b=0\). 

Assume now that \(\rank(G/H)=1\). 
If \(n=1\), then by Proposition~\ref{horospherical_by_torus_dimension}, \(G/H\) is horospherical, of dimension three. 
The subgroup \(H\) writes as the kernel of a character \(a_1\varpi_1 + \chi_1\) for some integer \(a_1\) and up to \(G\)-equivariant isomorphism, we may assume that \(a_1\in \bbN\). 

Finally, if \(\rank(G/H)=1\) and \(n=0\), then the fiber \(G_0/H_0\) under \(G/H\to G/Q\) is a rank one, non horospherical homogeneous space of dimension two under the action of a quotient of \(\SL_2\times \bbG_m\). 
In view of their classification, we deduce that we can choose \(G_0=\PSL_2\) and \(H_0\) a maximal torus of \(\PSL_2\), or its normalizer. 
This yields the description of the remaining items~(\ref{prop_H_SL_3_PI:horosym}) and (\ref{prop_H_SL_3_PI:Nhorosym}) in the statement. 
\end{proof}

\subsection{Remaining rank one spherical homogeneous spaces of dimension 4}

We now deal with the remaining rank one spherical homogeneous spaces of dimension 4. 
We need some specific conventions. 
For \(G=\Sp_4\times \bbG_m\) we let \(\alpha_1\) be the short simple root, and \(\alpha_2\) be the long simple root, so that \(G/Q_{\{\alpha_1\}}\simeq \bbP^3\) and \(G/Q_{\{\alpha_2\}}\simeq Q^3\). 
For \(G=\SL_4\times \bbG_m\) we let \(\alpha_1\) be a simple root such that \(G/Q_{\{\alpha_1\}} \simeq \bbP^3\). 
When there is only one choice up to \(G\)-equivariant isomorphism, we denote by \(Q_i\) a maximal parabolic subgroup of \(G_i\).

\begin{prop}
\label{prop_rk_one} 
Let \(G\) and \(G/H\) satisfy Assumption~\ref{assumptions}, with \(\dim(G/H)\leq 4\). 
Assume furthermore that the image of \(G\) in \(\Aut(G/H)\) contains \(\Aut^{G,0}(G/H)\), and that \(G^{sc}\notin\{\SL_2, \SL_2^2, \SL_3\}\). 
Then up to \(G\)-equivariant isomorphism, we are in one of the following cases:
\begin{enumerate}
    \item \label{prop_rk_one:Sp4_Q4} \(G=\Sp_4\), \(H=\SL_2\times\Sp_2\)
    \item \label{prop_rk_one:Sp4_P4} \(G=\Sp_4\),  \(H=N_{\Sp_4}(\SL_2\times\Sp_2)\)
    \item \label{prop_rk_one:SL23_P1xQ3} \(G=\SL_2^3\),  \(H=B_1^-\times \diag(\SL_2)\)
    \item \label{prop_rk_one:SL23_P1xP3} \(G=\SL_2^3\),  \(H=B_1^-\times N_{\SL_2^2}(\diag(\SL_2))\)
    \item \label{prop_rk_one:SL23_P14} \(G=\SL_2^3\),  \(H=B_1^-\times B_2^-\times T_3\)
    \item \label{prop_rk_one:SL23_P12xP2} \(G=\SL_2^3\),  \(H=B_1^-\times B_2^-\times N_{\SL_2}(T_3)\)
    \item \label{prop_rk_one:SL3xSL2_P2xP12} \(G=\SL_3\times \SL_2\),  \(H=Q_1\times T_2\)
    \item \label{prop_rk_one:SL3xSL2_P22} \(G=\SL_3\times \SL_2\),  \(H=Q_1\times N_{\SL_2}(T_2)\)
    \item \label{prop_rk_one:SL23_horo} \(G=\SL_2^3\times \bbG_m\),  \(H=\ker(a_1\varpi_1+a_2\varpi_2+a_3\varpi_3+\chi_1)|_{B^-}\) with \(a_1\geq \lvert a_2\rvert \geq \lvert a_3\rvert\)
    \item \label{prop_rk_one:SL3xSL2_horo} \(G=\SL_3\times \SL_2\times \bbG_m\),  \(H=\ker(a_1\varpi_1+a_3\varpi_3+\chi_1)|_{Q_{\{\alpha_1\}}}\) with \(a_1\geq 0\)
    \item \label{prop_rk_one:Sp4_horoshort} \(G=\Sp_4\times \bbG_m\),  \(H=\ker(a_1\varpi_1+\chi_1)|_{Q_{\{\alpha_1\}}}\) with \(a_1\geq 0\) 
    \item \label{prop_rk_one:Sp4_horolong} \(G=\Sp_4\times \bbG_m\),  \(H=\ker(a_2\varpi_2+\chi_1)|_{Q_{\{\alpha_2\}}}\) with \(a_2\geq 0\) 
    \item \label{prop_rk_one:SL4_horo} \(G=\SL_4\times \bbG_m\),  \(H=\ker(a_1\varpi_1+\chi_1)|_{Q_{\{\alpha_1\}}}\) with \(a_1\geq 0\)
\end{enumerate}
\end{prop}

\begin{proof}
There is a general classification result for rank one spherical homogeneous spaces by Akhiezer \cite{Akhiezer_1983}, which states that any such homogeneous space is obtained by parabolic induction from a homogeneous space \(G_0/H_0\) in a list of primitive cases described explicitly in \cite{Akhiezer_1983}. 
It is easy to extract the elements of this list of dimension less than four: they are \(\bbG_m\), \(\SL_2/\GL_1\), \(\SL_2/N_{\SL_2}(\GL_1)\), \(\SL_2^2/\diag\SL_2\), \(\SL_2^2/N_{\SL_2^2}(\diag\SL_2)\), \(\SL_3/\GL_2\), \(\Sp_4/\SL_2\times \Sp_2\) and \(\Sp_4/N_{\Sp_4}(\SL_2\times \Sp_2)\). 

The last three are of dimension 4, so cannot be involved in a parabolic induction. 
Furthermore, \(\SL_3/\GL_2\) already appeared in the previous section since in this case \(G^{sc}=\SL_3\). 
We now study the possible parabolic inductions for each of the other elements of the list. 
Note first that if \(\dim(G/H)\leq 3\) then \(G^{sc}\in \{\SL_2, \SL_2^2, \SL_3\}\) so the corresponding homogeneous spaces were already considered in previous sections. 

In the case \(G_0/H_0= \SL_2^2/\diag\SL_2\), or \(\SL_2^2/N_{\SL_2^2}(\diag \SL_2)\), the fiber of the parabolic induction has dimension three, hence the basis is \(\bbP^1\). 
But then one must have \(G=\SL_2^3\), the parabolic induction is with respect to a maximal parabolic subgroup \(Q\), containing two of the \(\SL_2\)-factors, and the morphism \(Q\to G_0\) must kill the action of the Borel subgroup of the remaining factor. The parabolic induction is thus actually a product of \(\bbP^1\) with an embedding of \(\SL_2^2/\diag\SL_2\), or \(\SL_2^2/N_{\SL_2^2}(\diag\SL_2)\). 

We now deal with the case when the fiber is \(\SL_2/\GL_1\) or \(\SL_2/N_{\SL_2}(\GL_1)\). 
The basis is then of dimension \(2\) so it is \(\bbP^1\times \bbP^1\) or \(\bbP^2\). 
Once again, if it is \(\bbP^1\times \bbP^1\), then the group \(G\) must be \(\SL_2^3\), and the parabolic induction is actually a product \(\bbP^1\times \bbP^1\times G_0/H_0\). 
If the basis is \(\bbP^2\), then \(G\) is either \(\SL_3\) or \(\SL_3\times \SL_2\), but in the latter case again, the parabolic induction is a product. 
If \(G=\SL_3\), then \(G/H\) already appeared in the previous section. 

Finally, we deal with the case when the fiber is \(\bbG_m\), equivalently, we consider horospherical homogeneous spaces of rank one. 
Recall from Remark~\ref{remark_horospherical_subgroups} that such a homogeneous space \(G/H\) is fully determined by the data of a parabolic subgroup \(Q\) of \(G\) and a character \(\lambda\in X^*(Q)\), by \(H=\ker(\lambda:Q\to \bbG_m)\). 
The possible simply connected semisimple groups \(G^{sc}\) acting were determined in Proposition~\ref{rank+3=dim}, using the fact that \(\dim(G^{sc}/Q)=3\).
We already dealt with the case \(G^{sc}=\SL_3\) in the previous section.
By considering the action of the equivariant automorphism group, we have \(G=G^{sc}\times \bbG_m\). 

Assume that \(G=\SL_2^3\times \bbG_m\). 
In that case, the parabolic subgroup \(Q\) is \(B^-\), so that \(G/Q=(\bbP^1)^3\). 
The character \(\lambda\) writes \(a_1\varpi_1+a_2\varpi_2+a_3\varpi_3+c\chi\) where the \(\varpi_i\) are fundamental weights for each factor, and by Assumption~\ref{assumptions}, \(c=\pm 1\). 
Up to rearranging the factors and replacing \(\chi\) with its opposite, we may assume that \(a_1\geq \lvert a_2\rvert \geq \lvert a_3\rvert\)  and \(c=1\). 

Assume that \(G=\SL_3\times \SL_2\times \bbG_m\). 
Then the parabolic subgroup is the product of a maximal parabolic subgroup \(Q_1\) of \(\SL_3\) with a Borel subgroup \(B_2^-\) of \(\SL_2\) (and with the \(\bbG_m\)-factor), so that \(G/Q=\bbP^2\times \bbP^1\). 
Then \(\lambda=a_1\varpi_1+a_3\varpi_3+\chi\) and we may only assume \(a_1\geq 0\) by changing \(\chi\) to \(-\chi\). 

Assume that \(G=\SL_4\times \bbG_m\). Then \(Q=Q_{\{\alpha_1\}}\), so that \(G/Q\simeq \bbP^3\), and \(\lambda=a_1\varpi_1+\chi\) for some \(a_1\geq 0\). 

Finally, assume that \(G=\Sp_4\times \bbG_m\). 
Then there are two possible choices for the parabolic subgroup \(Q\), which in any case is a maximal parabolic subgroup. 
For \(G/Q\) we have either \(\bbP^3\) or \(Q^3\). 
If \(\varpi\) denotes the fundamental weight of \(\Sp_4\) such that \(\varpi\in X^*(Q)\), then \(\lambda=a\varpi+\chi\) for some \(a\geq 0\). 
\end{proof}

\section{Conventions, and dimension two}

\label{section_dimension_2}

To make the conventions in the following sections precise, and warm up, we recall the spherical actions in dimension two (in dimension one, it is obvious that there are only the rank zero homogeneous structure on \(\bbP^1\) and the toric structure on \(\bbP^1\)).

In the following sections, we will work our way to the classification by considering each group successively. 
Then for each corresponding spherical homogeneous space \(G/H\), we will determine the combinatorial data (\(M\), \(\Sigma\), \((\mathcal{D},\rho,\zeta)\)), then the possible locally factorial \(G/H\)-reflexive polytopes. 
Concerning the combinatorial data, we will give a basis for the lattice \(M\), in terms of the basis of \(X^*(T)\) given in our conventions by the \(\varpi_i\) and \(\chi_j\). 
Given this basis, we will identify \(M\) with \(\bbZ^r\), the dual \(N\) with \(\bbZ^r\) using the dual \(\bbZ\)-basis, and \(M\otimes \bbR\) with \(\bbR^r\), equipped with coordinates \((x_1,\ldots, x_r)\). 
This will allow to describe the color map \(\rho:\mathcal{D}\to N\), and the Duistermaat-Heckman polynomial \(f:M\otimes \bbR \to \bbR\) concisely.  
Finally, the set of colors will be denoted abstractly as a subset of \(\{\clubsuit, \varheartsuit, \vardiamondsuit, \spadesuit\}\) to emphasize that only the abstract set and maps matter. 

We will next identify the possible locally factorial \(G/H\)-reflexive polytopes (up to equivalences corresponding to \(G\)-equivariant isomorphisms), and will draw these. 
In the pictures, we will draw the polytopes themselves, highlight the lattice, and indicate the images of the colors. 
Under each polytope, we will indicate an identifier \textbf{dim-rank-number}, and some geometrical data associated to the corresponding embedding will be summarized in the appendix in a big table. 

We first illustrate these conventions with the dimension two case. 
Assume that \(G/H\) is of dimension two, then the rank is at most two, and the group acting is \(\bbG_m^2\) if the rank is two, \(\SL_2\times \bbG_m^n\) if the rank is one, with \(n\in \{0,1\}\), and \(\SL_3\) or \(\SL_2\times \SL_2\) if the rank is zero, in which case, we have the rational homogeneous surfaces \(\bbP^2\) and \(\bbP^1\times \bbP^1\). 
We will not say more about the rank zero case. 
For the other cases, we have essentially gathered the information needed in Section~\ref{section_reminder}, we illustrate in the remainder of the section our conventions and summarize the classification. 

\begin{prop}
Let \(G=\bbG_m^2\) and \(H=\{1\}\), then the combinatorial data of \(G/H\) is as follows. 
\[\arraycolsep=12pt
\begin{array}{lll}
M=\langle \chi_1,\chi_2\rangle  & \Sigma = \emptyset &  \mathcal{D} =\emptyset  \\
\midrule
\kappa = 0 & f=1 &
\end{array}
\]
\end{prop}

The corresponding polytopes, up to equivalences, are given by the five smooth Fano polygons. 

\begin{center}
\begin{tikzpicture}[scale=0.6]
\draw[dotted] (-2,2) grid (2,-2);
\draw (0,0) node{+};
\draw[thick] (-1,1) -- (0,-1) -- (1,0) -- cycle;
\draw (0,-1) node[below]{\textbf{2-2-1}};
\end{tikzpicture}
\begin{tikzpicture}[scale=0.6]
\draw[dotted] (-2,2) grid (2,-2);
\draw (0,0) node{+};
\draw[thick] (-1,1) -- (0,-1) -- (1,-1) -- (1,0) -- cycle;
\draw (0,-1) node[below]{\textbf{2-2-2}};
\end{tikzpicture}
\begin{tikzpicture}[scale=0.6]
\draw[dotted] (-2,2) grid (2,-2);
\draw (0,0) node{+};
\draw[thick] (0,1) -- (-1,0) -- (0,-1) -- (1,0) -- cycle;
\draw (0,-1) node[below]{\textbf{2-2-3}};
\end{tikzpicture}
\begin{tikzpicture}[scale=0.6]
\draw[dotted] (-2,2) grid (2,-2);
\draw (0,0) node{+};
\draw[thick] (0,1) -- (-1,1) -- (-1,0) -- (0,-1) -- (1,0) -- cycle;
\draw (0,-1) node[below]{\textbf{2-2-4}};
\end{tikzpicture}
\begin{tikzpicture}[scale=0.6]
\draw[dotted] (-2,2) grid (2,-2);
\draw (0,0) node{+};
\draw[thick] (0,1) -- (-1,1) -- (-1,0) -- (0,-1) -- (1,-1) -- (1,0) -- cycle;
\draw (0,-1) node[below]{\textbf{2-2-5}};
\end{tikzpicture}
\end{center}

\begin{prop}
Let \(G=\SL_2\) and \(H=T\), then the combinatorial data of \(G/H\) is as follows. 
\[\arraycolsep=12pt
\begin{array}{lll}
M=\langle \alpha_1\rangle  & \zeta(\clubsuit) = \zeta(\varheartsuit) = \{\alpha_1\}  \\
\Sigma = \{\alpha_1\} & \rho(\clubsuit) = \rho(\varheartsuit) = 1 \\
\mathcal{D} =\{\clubsuit,\varheartsuit\} &  \\
\midrule
\kappa = \alpha_1 & f=2(1+x_1) & m_{\clubsuit} = m_{\varheartsuit} =1 \\
\end{array}
\]
\end{prop}

There is a unique locally factorial \(G/H\)-reflexive polytope:
\begin{center}
\begin{tikzpicture}
\draw[dotted] (-2,0) -- (2,0);
\draw (0,0) node{+};
\foreach \i in {0,...,4}
{
\draw[dotted] (-2+\i,.3) -- (-2+\i,-.3);
}
\draw[thick] (-1,0) -- (1,0);
\draw (1,0) node[above=-4, color=teal]{\(\clubsuit\)};
\draw (1,0) node[below=-4, color=purple]{\(\varheartsuit\)};
\draw (0,0) node[below]{\textbf{2-1-1}};
\end{tikzpicture}
\end{center}

\begin{prop}
Let \(G=\SL_2\) and \(H=N(T)\), then the combinatorial data of \(G/H\) is as follows. 
\[\arraycolsep=12pt
\begin{array}{lll}
M=\langle 2\alpha_1\rangle  & \zeta(\clubsuit) = \{\alpha_1\}  \\
\Sigma = \{2\alpha_1\} & \rho(\clubsuit) = 2 \\
\mathcal{D} =\{\clubsuit\} &  \\
\midrule
\kappa = \alpha_1 & f=2(1+2x_1) & m_{\clubsuit} = 1 \\
\end{array}
\]
\end{prop}

There is a unique locally factorial \(G/H\)-reflexive polytope:
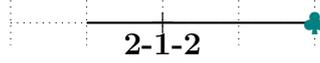
\begin{center}
\begin{tikzpicture}
\draw[dotted] (-2,0) -- (2,0);
\draw (0,0) node{+};
\foreach \i in {0,...,4}
{
\draw[dotted] (-2+\i,.3) -- (-2+\i,-.3);
}
\draw[thick] (-1,0) -- (2,0);
\draw (2,0) node[color=teal]{\(\clubsuit\)};
\draw (0,0) node[below]{\textbf{2-1-2}};
\end{tikzpicture}
\end{center}

There are finally three horospherical actions of \(\SL_2\times \bbG_m\), one is the product of the homogeneous \(\bbP^1\) under \(\SL_2\) with the toric \(\bbP^1\) under \(\bbG_m\), which we identify as \textbf{2-1-3}. 

\begin{prop}
Let \(G=\SL_2\times \bbG_m\) and \(H=\ker(\varpi_1+\chi_1)|_{B^-}\), then the combinatorial data of \(G/H\) is as follows. 
\[\arraycolsep=12pt
\begin{array}{lll}
M=\langle \varpi_1+\chi_1\rangle  & \zeta(\clubsuit) = \{\alpha_1\}  \\
\Sigma = \emptyset & \rho(\clubsuit) = 1 \\
\mathcal{D} =\{\clubsuit\} &  \\
\midrule
\kappa = \alpha_1 & f=2+x_1 & m_{\clubsuit} = 2 \\
\end{array}
\]
\end{prop}

The two possible locally factorial \(G/H\)-reflexive polytope are as follows. 
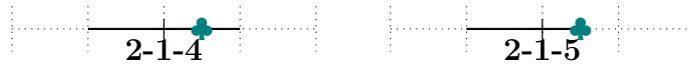
\begin{center}
\begin{tikzpicture}
\draw[dotted] (-2,0) -- (2,0);
\draw (0,0) node{+};
\foreach \i in {0,...,4}
{
\draw[dotted] (-2+\i,.3) -- (-2+\i,-.3);
}
\draw[thick] (-1,0) -- (1,0);
\draw (1/2,0) node[color=teal]{\(\clubsuit\)};
\draw (0,0) node[below]{\textbf{2-1-4}};
\end{tikzpicture}
\qquad
\begin{tikzpicture}
\draw[dotted] (-2,0) -- (2,0);
\draw (0,0) node{+};
\foreach \i in {0,...,4}
{
\draw[dotted] (-2+\i,.3) -- (-2+\i,-.3);
}
\draw[thick] (-1,0) -- (1/2,0);
\draw (1/2,0) node[color=teal]{\(\clubsuit\)};
\draw (0,0) node[below]{\textbf{2-1-5}};
\end{tikzpicture}
\end{center}

\begin{longtable}{CCCCCCC}
\text{identifier}&\pic&\degree&\text{KE?}&\text{description}&\text{group}&\text{type}\\
\toprule
\text{2-1-1}&2&8&y&\bbP^1\times\bbP^1&\SL_2&\text{symmetric}\\
\text{2-1-2}&1&9&y&\bbP^2&\SL_2&\text{symmetric}\\
\text{2-1-3}&2&8&y&\bbP^1\times\bbP^1&\SL_2\times\bbG_m&\text{horospherical}\\
\text{2-1-4}&2&8&n&\bbF_1&\SL_2\times\bbG_m&\text{horospherical}\\
\text{2-1-5}&1&9&y&\bbP^2&\SL_2\times\bbG_m&\text{horospherical}\\
\text{2-2-1} & 1 & 9 & y & \bbP^2 & \bbG_m^2 & \text{toric}\\
\text{2-2-2} & 2 & 8 & n & \bbF_1 & \bbG_m^2 & \text{toric}\\
\text{2-2-3} & 2 & 8 & y & \bbP^1\times\bbP^1 & \bbG_m^2 & \text{toric}\\
\text{2-2-4} & 3 & 7 & n & \Bl_{\text{2 pts}}\bbP^2 & \bbG_m^2 & \text{toric}\\
\text{2-2-5} & 4 & 6 & y & \Bl_{\text{3 pts}}\bbP^2 & \bbG_m^2 & \text{toric}\\
\end{longtable}

\section{Locally factorial Fano spherical \(\SL_2\times \bbG_m^n\)-manifolds of rank \(\leq 2\)}
\label{section_SL2}

As spherical fourfolds under \(\SL_2\times \bbG_m^n\) must be of rank three, we will not classify these in the present paper. 
We will however recover the full classification of (almost) faithful spherical actions of \(\SL_2\times \bbG_m^n\) on locally factorial Fano threefolds.

\subsection{Case when \(H_1=T_1\)}

\subsubsection{Combinatorial data}

\begin{prop}
\label{prop_combi_SL2:T}
Let \(G=\SL_2\times \bbG_m^n\) and \(H=\ker(a_1\varpi_1+\chi_1)|_{T_1\times \bbG_m}\times \{1_{\bbG_m^{n-1}}\}\) for some \(a_1\in \bbZ_{\geq 0}\). 
Then the combinatorial data of the spherical homogeneous space \(G/H\) is as follows. If \(a_1\) is even, then: 
\[\arraycolsep=12pt
\begin{array}{lll}
M=\langle \alpha_1,\chi_1,\ldots \chi_n\rangle  & \zeta(\clubsuit) = \zeta(\varheartsuit) = \{\alpha_1\}  \\
\Sigma = \{\alpha_1\} & \rho(\clubsuit) = (1,a_1/2) & \rho(\varheartsuit)=(1,-a_1/2) \\
\mathcal{D} =\{\clubsuit,\varheartsuit\} &  \\
\midrule
\kappa = \alpha_1 & f=2+2x_1 & m_{\clubsuit} = m_{\varheartsuit} = 1 \\
\end{array}
\]
If \(a_1\) is odd, then: 
\[\arraycolsep=12pt
\begin{array}{lll}
M=\langle \varpi_1+\chi_1, \varpi_1-\chi_1, \chi_2, \ldots, \chi_n \rangle &\rho(\varheartsuit)=(\frac{1-a_1}{2},\frac{a_1+1}{2})  &    \\
\Sigma = \{\alpha_1\} & \rho(\clubsuit) = (\frac{a_1+1}{2},\frac{1-a_1}{2}) &  \\
\mathcal{D} =\{\clubsuit,\varheartsuit\} & \zeta(\clubsuit) = \zeta(\varheartsuit) = \{\alpha_1\} \\
\midrule
m_{\clubsuit} = m_{\varheartsuit} = 1  & f=2+x_1+x_2 & \kappa = \alpha_1  \\
\end{array}
\]
\end{prop}

\begin{proof}
Since \(G/H\) is a product of a \(\SL_2\times \bbG_m\) homogeneous space with the tori \((\bbC^*)^{n-1}\), it suffices to prove the result when \(n=1\). 
When \(a_1=0\), the homogeneous space is a product, the result is immediate. 
Assume now that \(a_1\geq 1\). 
Let \(\bbP(1,1,a_1)\) denote the weighted projective space, defined as the quotient of \(\bbC^3\) under the action of \(\bbG_m\) given by \(t\cdot (x,y,z)=(tx,ty,t^{a_1}z)\).  
Consider the action of \(\SL_2\times \bbG_m\) on \(\bbP^1\times \bbP(1,1,a_1)\) given by 
\begin{align*} 
\left( \begin{bmatrix} a & b \\ c & d \end{bmatrix}, s\right) & \cdot \left([u:v],[x:y:z]\right) = \\
& \left([au+bv:cu+dv],[ax+by:cx+dy:sz]\right) 
\end{align*}
Then the stabilizer of \(\left([1:0],[0:1:1]\right)\) is precisely \(H\). 

In fact, we can easily describe the five \(G\)-orbits: 
\begin{align*}
O & = \left\{ ([u:v],[x:y:z]) \mid [x:y]\neq [u:v], z\neq 0 \right\} \\
I_1 & = \left\{ ([x:y],[x:y:z]) \mid z\neq 0 \right\} \\
I_2 & = \left\{ ([u:v],[x:y:0]) \mid [x:y]\neq [u:v] \right\} \\
C_1 & = \left\{ ([u:v],[0:0:1]) \right\} \\
C_2 & = \left\{ ([x:y],[x:y:0]) \right\} 
\end{align*}
The open orbit is \(O\). 
The codimension one orbits are \(I_1\) and \(I_2\) and the codimension two orbits are \(C_1\) and \(C_2\). 
The closure of \(I_1\) is \(I_1\cup C_1\cup C_2\), it is a prime \(G\)-stable divisor whose equation is \(xv-yu=0\). 
The closure of \(I_2\) is \(I_2\cup C_2\), it is a prime \(G\)-stable divisor whose equation is \(z=0\).  

Consider now the action of the Borel subgroup \(B=\left\{\left(\begin{bmatrix} a & b \\ 0 & 1/a \end{bmatrix}, s\right)\right\}\). 
It acts with two codimension one orbits in \(O\): 
\[ D_1 = \left\{ ([u:1],[1:0:z]) \mid z\neq 0 \right\} \] 
and 
\[ D_2 = \left\{ ([1:0],[x:y:z]) \mid [1:0]\neq [x,y], z\neq 0 \right\}. \]
The equation for (the closure of) \(D_1\) is \(y=0\) and the equation for \(D_2\) is \(v=0\).  
As a consequence, we have \(\mathcal{D}=\{D_1,D_2\}\) and \(\zeta(D_i)=\{\alpha_1\}\) for \(i=1\) or \(2\). 

The rational functions 
\[ f_1 : ([u:v],[x:y:z]) \longmapsto \frac{y^m}{z} \]
and 
\[ f_2 : ([u:v],[x:y:z]) \longmapsto \frac{vy}{xv-yu} \]
are well-defined and \(B\)-semi-invariant with \(B\)-weight \(a_1\varpi_1+\chi\) and \(2\varpi_1\). 
To prove that these two weights generate \(M\subset \varpi_1\bbZ \oplus \chi\bbZ\), it suffices to check that \(\varpi_1\notin M\). 
It is the case indeed: since \((-I_2,(-1)^{a_1})\) is central and acts trivially on \(G/H\), the weight of any \(B\)-semi-invariant rational function must vanish on it, hence \(\varpi_1\notin M\). 

By Corollary~\ref{cor_Sigma_SL_2}, \(\Sigma=\{\alpha_1\}=\{2\varpi_1\}\). 
The description of images of colors by \(\rho\) follows readily from our explicit description of equations of the colors and rational \(B\)-semi-invariant functions whose weights generate \(M\).  
\end{proof}

\subsubsection{Polytopes}

We denote the coordinates on \(N\otimes \bbR\) by \((y_1,\ldots, y_{n+1})\), in the dual basis to the given \(\bbZ\)-basis of \(M\), which depends on the parity of \(a_1\).
We will give, for later use, a purely combinatorial characterization of locally factorial \(G/H\)-reflexive polytopes which is closer to the usual notions for toric varieties. 
Before that, let us thus recall some of the relevant definitions when considering polytopes of singular Fano toric varieties. 

\begin{defn}
A polytope \(\Omega\) with integral vertices in \(\bbZ^{n+1}\), containing the origin in its interior, is called: 
\begin{itemize}
    \item \emph{smooth Fano} if for any facet \(F\) of \(\Omega\), the vertices of \(F\) form a basis of \(\bbZ^{n+1}\), 
    \item \emph{terminal Fano} if the only integral points in \(\Omega\) are its vertices and the origin,
    \item \emph{canonical Fano} if the only interior integral point in \(\Omega\) is the origin. 
\end{itemize}
A facet of a canonical polytope is called \emph{non-smooth} if its vertices do not form a basis of \(\bbZ^{n+1}\). 
\end{defn}

\begin{prop}
\label{prop_LFRP_T_0}
Assume that \(a_1=0\). 
Then \(\Omega\) is a locally factorial \(G/H\)-reflexive polytope if and only if \(\Omega\) is a smooth Fano polytope such that \((1,0,\ldots,0)\) is a vertex, all other vertices are in the half-space \(H_-=\{y_1\leq 0\}\) and all the facets containing \((1,0,\ldots,0)\) are fully contained in the half-space \(H_+=\{y_1\geq 0\}\).  
Locally factorial Fano embeddings of \(G/H\) correspond uniquely to such polytopes up to the action of \(\GL_n(\bbZ)\) on the last \(n\) coordinates. 
\end{prop}

\begin{proof}
The statement essentially follows from the definition of locally factorial \(G/H\)-reflexive polytopes, Proposition~\ref{prop_combi_SL2:T} and our identification of the lattice \(N\) with \(\bbZ^{n+1}\). 
There are two particular things to note. 
First, both colors have the same image, so the corresponding points \(\frac{\rho(D_i)}{m_{D_i}}=(1,0,\ldots,0)\) cannot be in a facet that intersects the interior of the valuation cone. 
Second, if a facet contains \((1,0,\ldots,0)\) (necessarily as a vertex), then the other vertices are on the hyperplane \(H_-\cap H_+\) and form a face of codimension two, which must belong to another facet, contained in \(H_-\). 
Since the vertices of that second facet must form a basis of \(\bbZ^{n+1}\), the vertices in \(H_-\cap H_+\) form a basis of \(\{0\}\times \bbZ^n\). 
As a consequence, the original facet containing \((1,0,\ldots,0)\) is a smooth facet. 

The last sentence of the statement follows from Remark~\ref{rem_SL2_non-uniqueness_of_chi}. 
\end{proof}

\begin{prop}
\label{prop_LFRP_T_even>0}
Assume that \(a_1\geq 2\) is even. 
Then \(\Omega\) is a locally factorial \(G/H\)-reflexive polytope if and only if \(\Omega\) is a canonical Fano polytope such that \((1,\frac{a_1}{2},0\ldots,0)\) and \((1,-\frac{a_1}{2},0\ldots,0)\) are vertices, all other vertices are in the half-space \(H_-=\{y_1\leq 0\}\), any non-smooth facet of \(\Omega\) contains both points \((1,\pm\frac{a_1}{2},0\ldots,0)\) and is fully contained in the half-space \(H_+=\{y_1\geq 0\}\). 

Furthermore, any non-smooth facet of \(\Omega\) has \(n+1 \leq m \leq n+2\) vertices, and \(m+a_1-1\) integral points, and integral points which are not vertices are in the interior of the one-dimensional face \([(1,\frac{a_1}{2},0\ldots,0),(1,-\frac{a_1}{2},0\ldots,0)]\).  

Locally factorial Fano embeddings of \(G/H\) correspond uniquely to such polytopes up to the action of the subgroup of \(\GL_{n+1}(\bbZ)\) that fixes \((1,0,\ldots,0)\) and stabilizes both \(\bbZ^2\times \{0\}\) and the hyperplane \(H_-\cap H_+\). 
\end{prop}

\begin{proof}
In this case, the colors have different images, so they can be involved in the polytope. 
If a facet has \((1,\frac{a_1}{2},0\ldots,0)\) as vertex but not \((1,-\frac{a_1}{2},0\ldots,0)\), then the same argument as in the previous proof shows that the facet is spanned by a basis of \(\bbZ^{n+1}\).
If a facet \(F\) has \((1,\frac{a_1}{2},0\ldots,0)\) and \((1,-\frac{a_1}{2},0\ldots,0)\) as vertices, then its other vertices are on \(H_-\cap H_+\) since the facet cannot be smooth. 
These other vertices lie in another facet of \(\Omega\), which cannot contain both \((1,\frac{m}{2},0\ldots,0)\) and \((1,-\frac{m}{2},0\ldots,0)\), so these vertices are part of a basis of \(\bbZ^{n+1}\). 
As a consequence, there cannot be any interior integral point in \(\Omega \cap H_-\cap H_+\). 
Finally, there cannot be any integral interior points in the facet \(F\), since it would have non integral coordinate \(x_0\). 
We deduce that \(\Omega\) is canonical. 
The other consequences are easily deduced from the above. 
The last sentence of the statement again follows from Remark~\ref{rem_SL2_non-uniqueness_of_chi}. 
\end{proof}

Adapting the arguments to the case when \(a_1\) is odd yields the following characterization. 

\begin{prop}
\label{prop_LFRP_T_odd}
Assume that \(a_1\) is odd. 
Then \(\Omega\) is a locally factorial \(G/H\)-reflexive polytope if and only if \(\Omega\) is a canonical Fano polytope such that \((\frac{a_1+1}{2},\frac{1-a_1}{2},0\ldots,0)\) and \((\frac{1-a_1}{2},\frac{a_1+1}{2},0\ldots,0)\) are vertices, all other vertices are in the half-space \(H_-=\{y_2\leq -y_1\}\), any non-smooth facet of \(\Omega\) contains both points \((\frac{a_1+1}{2},\frac{1-a_1}{2},0\ldots,0)\) and \((\frac{1-a_1}{2},\frac{a_1+1}{2},0\ldots,0)\) and is fully contained in the half-space \(H_+=\{y_2\geq -y_1\}\). 

Furthermore, any non-smooth facet of \(\Omega\) has \(n+1 \leq m \leq n+2\) vertices, and \(m+a_1-1\) integral points, and integral points which are not vertices are in the interior of the one-dimensional face \([(\frac{a_1+1}{2},\frac{1-a_1}{2},0\ldots,0),(\frac{1-a_1}{2},\frac{a_1+1}{2},0\ldots,0)]\).  

Locally factorial Fano embeddings of \(G/H\) correspond uniquely to such polytopes up to the action of the subgroup of \(\GL_{n+1}(\bbZ)\) that fixes \((1,1,0\ldots, 0)\), and stabilizes \(\bbZ^2\times \{0\}\) as well as the hyperplane \(H_+\cap H_-\). 
\end{prop}

\begin{cor}
\label{prop_LFRP_T_1}
Assume that \(a_1=1\). 
Then \(\Omega\) is furthermore a terminal Fano polytope, and any non-smooth facet has \(n+2\) vertices. 
\end{cor}

\begin{proof}
It follows from the proof in the general case that integral points which are not vertices or the origin can only lie on the dimension one face which is the segment between the images of the two colors. 
In the case \(a_1=1\), this is the segment between \((1,0\ldots,0)\) and \((0,1,0\ldots,0)\) which obviously does not contain any other integral points than its endpoints. 
Furthermore, since \((1,0\ldots,0)\) and \((0,1,0\ldots,0)\) form a basis of \(\bbZ^2\times \{0\}\), if a facet is contained in \(H_+\) and is a simplex, then it is actually smooth.  
\end{proof}

\begin{exa}
\label{exa_pic1_rk3}
Assume that \(n=2\), and that \(\Omega\) is a simplex. 
Then \(\Omega\) cannot be a locally factorial \(G/H\)-reflexive polytope if \(a_1\neq 1\). 
Indeed, note first that there cannot be two facets fully contained in a half-space (otherwise by vertex count, the simplex would not contain the origin in its interior). 
This directly prevents \(a_1=0\), as the color must be in at least three facets, all contained in the same half space. 
If \(a_1\geq 2\), this implies that at least one of the two facets containing the two colors must be smooth. 
But it is impossible since the facets contains the edge between the two colors, which contains an integral interior point. 

Now focus on the case \(a_1=1\). 
Since all the facets of a simplex are simplices, all the facets are smooth. 
Let us denote a third vertex of \(\Omega\) by \((x_1,y_1,1)\). 
The simplex being smooth, we know that the last vertex is \((-1-x_1,-1-y_1,-1)\). 
Up to the action of matrices of the form 
\[\begin{bmatrix} 1&0&a\\ 0&1&-a \\ 0&0&1 \end{bmatrix}\]
we may assume that \(y_1=0\). 
Finally, the condition that the vertices other than the colors are in \(H_-\) implies that \(-2\leq x_1 \leq 0\). 
The two cases \(x_1=0\) and \(x_1=2\) are related by the matrix 
\[\begin{bmatrix} 0&1&1 \\ 1&0&-1 \\ 0&0&-1 \end{bmatrix}\]
hence define the same embedding. 
In conclusion, we have found two locally factorial \(G/H\)-reflexive for \(a_1=1\), corresponding to two different locally factorial, Fano, Picard rank one embeddings.  
\end{exa}

\subsubsection{Rank 2 locally factorial Fano embeddings}

We now determine the polytopes when \(n=1\), that is, when \(\rank(G/H)=2\) and \(\dim(G/H)=3\). 

\textbf{When \(a_1=0\)}
by Proposition~\ref{prop_LFRP_T_0}, it suffices to go through the list of smooth Fano polygons up to \(\GL_2(\bbZ)\)-action (recalled in Section~\ref{section_dimension_2}), and for each class, determine which representative do satisfy the hypothesis. 
Assume that a representative \(\Omega_0\) has \((1,0)\) as a vertex, then the elements of \(\GL_2(\bbZ)\) that fix this vertex are of the form \(\begin{bmatrix} 1 & a \\ 0 & \pm 1 \end{bmatrix}\) for some \(a\in \bbZ\). 
Thanks to the last sentence of Proposition~\ref{prop_LFRP_T_0}, we can work up to reflection with respect to the \(y_1\)-axis, that is, we may assume \(\pm 1 = 1\). 
Note that since all facets containing \((1,0)\) must lie in \(H_+\), and the origin is in the interior, \((0,1)\) and \((0,-1)\) are vertices. 
Their images by \(\begin{bmatrix} 1 & a \\ 0 & \pm 1 \end{bmatrix}\) are \((a,1)\) and \((-a,-1)\), hence if \(a\neq 0\), one of these is in \(H_+\). 
As a consequence, for every pair of a smooth Fano polytope and a vertex, there is at most one representative which satisfy the hypotheses. 

Going through the list of smooth Fano polytopes and the list of their vertices yields the following list of three. 

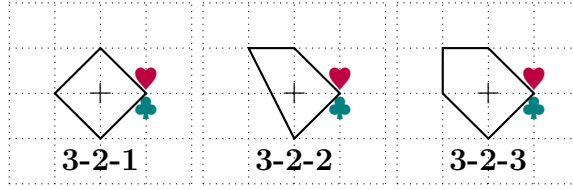
\begin{center}
\begin{tikzpicture}[scale=0.6]
\draw[dotted] (-2,2) grid (2,-2);
\draw (0,0) node{+};
\draw (1,0) node[below=-3, color=teal]{\(\clubsuit\)};
\draw (1,0) node[above=-3, color=purple]{\(\varheartsuit\)};
\draw[thick] (1,0) -- (0,1) -- (-1,0) -- (0, -1) -- cycle;
\draw (0,-1) node[below]{\textbf{3-2-1}};
\end{tikzpicture}
\begin{tikzpicture}[scale=0.6]
\draw[dotted] (-2,2) grid (2,-2);
\draw (0,0) node{+};
\draw (1,0) node[below=-3, color=teal]{\(\clubsuit\)};
\draw (1,0) node[above=-3, color=purple]{\(\varheartsuit\)};
\draw[thick] (1,0) -- (0,1) -- (-1,1) -- (0, -1) -- cycle;
\draw (0,-1) node[below]{\textbf{3-2-2}};
\end{tikzpicture}
\begin{tikzpicture}[scale=0.6]
\draw[dotted] (-2,2) grid (2,-2);
\draw (0,0) node{+};
\draw (1,0) node[below=-3, color=teal]{\(\clubsuit\)};
\draw (1,0) node[above=-3, color=purple]{\(\varheartsuit\)};
\draw[thick] (1,0) -- (0,1) -- (-1,1) -- (-1,0) -- (0, -1) -- cycle;
\draw (0,-1) node[below]{\textbf{3-2-3}};
\end{tikzpicture}
\end{center}

\begin{exa}
The embedding \textbf{3-2-1} is \((\bbP^1)^3\) as the product of the embedding \textbf{2-1-1} with the toric \(\bbP^1\), while the embedding \textbf{3-2-3} is the blowup of \((\bbP^1)^3\) along, for example, the closed orbit \(\diag \bbP^1\times \{0\}\), if \(0\) denotes a torus fixed point in \(\bbP^1\). 
\end{exa}

\textbf{When \(a_1=1\)}, using Proposition~\ref{prop_LFRP_T_1}, it suffices again to go through the list of smooth Fano polygons up to \(\GL_2(\bbZ)\)-action, and for each class, determine which representative do satisfy the hypothesis. 
Indeed, in dimension two, smooth and terminal Fano polygons coincide. 
We obtain eight polytopes. 

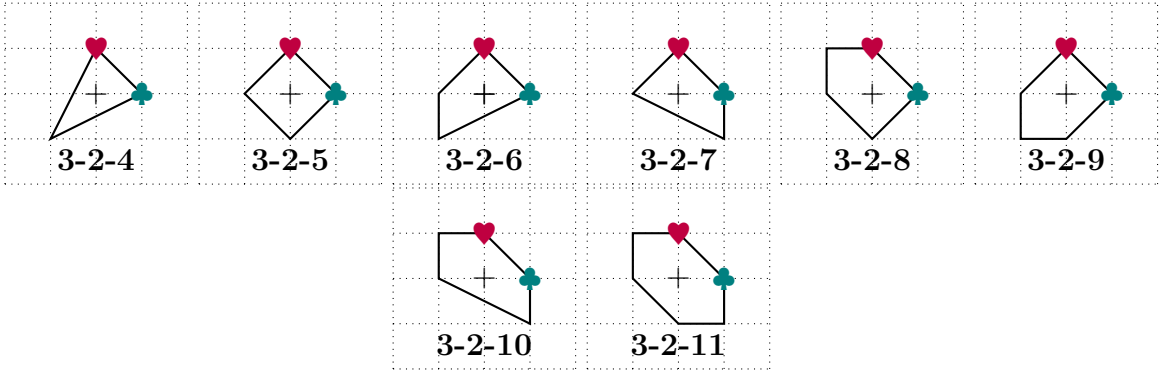
\begin{center}
\begin{tikzpicture}[scale=0.6]
\draw[dotted] (-2,2) grid (2,-2);
\draw (0,0) node{+};
\draw[thick] (1,0) -- (0,1) -- (-1,-1) -- cycle;
\draw (1,0) node[color=teal]{\(\clubsuit\)};
\draw (0,1) node[color=purple]{\(\varheartsuit\)};
\draw (0,-1) node[below]{\textbf{3-2-4}};
\end{tikzpicture}
\begin{tikzpicture}[scale=0.6]
\draw[dotted] (-2,2) grid (2,-2);
\draw (0,0) node{+};
\draw[thick] (1,0) -- (0,1) -- (-1,0) -- (0,-1) -- cycle;
\draw (1,0) node[color=teal]{\(\clubsuit\)};
\draw (0,1) node[color=purple]{\(\varheartsuit\)};
\draw (0,-1) node[below]{\textbf{3-2-5}};
\end{tikzpicture}
\begin{tikzpicture}[scale=0.6]
\draw[dotted] (-2,2) grid (2,-2);
\draw (0,0) node{+};
\draw (0,0) node{+};
\draw[thick] (1,0) -- (0,1) -- (-1,0) -- (-1,-1) -- cycle;
\draw (1,0) node[color=teal]{\(\clubsuit\)};
\draw (0,1) node[color=purple]{\(\varheartsuit\)};
\draw (0,-1) node[below]{\textbf{3-2-6}};
\end{tikzpicture}
\begin{tikzpicture}[scale=0.6]
\draw[dotted] (-2,2) grid (2,-2);
\draw (0,0) node{+};
\draw[thick] (1,0) -- (0,1) -- (-1,0) -- (1,-1) -- cycle;
\draw (1,0) node[color=teal]{\(\clubsuit\)};
\draw (0,1) node[color=purple]{\(\varheartsuit\)};
\draw (0,-1) node[below]{\textbf{3-2-7}};
\end{tikzpicture}
\begin{tikzpicture}[scale=0.6]
\draw[dotted] (-2,2) grid (2,-2);
\draw (0,0) node{+};
\draw[thick] (1,0) -- (0,1) -- (-1,1) -- (-1,0) -- (0,-1) -- cycle;
\draw (1,0) node[color=teal]{\(\clubsuit\)};
\draw (0,1) node[color=purple]{\(\varheartsuit\)};
\draw (0,-1) node[below]{\textbf{3-2-8}};
\end{tikzpicture}
\begin{tikzpicture}[scale=0.6]
\draw[dotted] (-2,2) grid (2,-2);
\draw (0,0) node{+};
\draw[thick] (1,0) -- (0,1) -- (-1,0) -- (-1,-1) -- (0,-1) -- cycle;
\draw (1,0) node[color=teal]{\(\clubsuit\)};
\draw (0,1) node[color=purple]{\(\varheartsuit\)};
\draw (0,-1) node[below]{\textbf{3-2-9}};
\end{tikzpicture}
\begin{tikzpicture}[scale=0.6]
\draw[dotted] (-2,2) grid (2,-2);
\draw (0,0) node{+};
\draw[thick] (1,0) -- (0,1) -- (-1,1) -- (-1,0) -- (1,-1) -- cycle;
\draw (1,0) node[color=teal]{\(\clubsuit\)};
\draw (0,1) node[color=purple]{\(\varheartsuit\)};
\draw (0,-1) node[below]{\textbf{3-2-10}};
\end{tikzpicture}
\begin{tikzpicture}[scale=0.6]
\draw[dotted] (-2,2) grid (2,-2);
\draw (0,0) node{+};
\draw[thick] (1,0) -- (0,1) -- (-1,1) -- (-1,0) -- (0,-1) -- (1,-1) -- cycle;
\draw (1,0) node[color=teal]{\(\clubsuit\)};
\draw (0,1) node[color=purple]{\(\varheartsuit\)};
\draw (0,-1) node[below]{\textbf{3-2-11}};
\end{tikzpicture}
\end{center}

\begin{exa}
The embedding \textbf{3-2-7} is the embedding \(\bbP^1\times \bbP(1,1,1)=\bbP^1\times \bbP^2\) used in the proof to determine the combinatorial data. 
With the notations of the proof, \textbf{3-2-10} is the blowup of \(\bbP^1\times \bbP^2\) along the curve \(C_2\) (which is isomorphic to \(\bbP^1\)), and \textbf{3-2-8} is the blowup of \(\bbP^1\times \bbP^2\) along the curve \(C_1\) (which is isomorphic to \(\bbP^1\) as well). 
The embedding \textbf{3-2-11} is the blowup along both curves. 
\end{exa}

\begin{exa}
The embedding \textbf{3-2-5} is the variety \(W\) of complete flags in \(\bbC^3\), equipped with the action of \(\SL_2\times\bbG_m\) induced by the action of \(\GL_2\) on the first two coordinates of \(\bbC^3\). 
There are three orbits of codimension two under this action: the orbit \(C_1\) of flags with line generated by \((0,0,1)\), the orbit \(C_2\) of flags with the plane stabilized by \(\GL_2\), and the orbit \(C_3\) of flags of the form \(\langle (x,y,0)\rangle \subset \langle (x,y,0),(0,0,1)\rangle\).  

The embedding \textbf{3-2-8} can alternatively be recovered as the blowup of \(W\) along \(C_1\) or \(C_2\), and \text{3-2-11} as the blowup of \(W\) along both orbits. 
The embedding \textbf{3-2-9} on the other hand, is obtained by blowing up \(C_3\). 
\end{exa}

\textbf{When \(a_1=2\)}, using Proposition~\ref{prop_LFRP_T_even>0}, it suffices to go through the list of canonical Fano polygons up to \(\GL_2(\bbZ)\)-action, and for each class, determine which representative do satisfy the hypothesis. 
Canonical Fano polygons in dimension two coincide with reflexive polygons, there are 16 such polytopes up to \(\GL_2(\bbZ)\)-action. 
Among these, only three have exactly one non-smooth facet with one integral interior point. 
We obtain the following possibilities for locally factorial \(G/H\)-reflexive polytopes. 

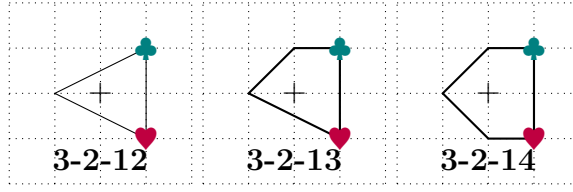
\begin{center}
\begin{tikzpicture}[scale=0.6]
\draw[dotted] (-2,2) grid (2,-2);
\draw (0,0) node{+};
\draw (1,-1) -- (1,1) -- (-1,0) -- cycle;
\draw (1,1) node[color=teal]{\(\clubsuit\)};
\draw (1,-1) node[color=purple]{\(\varheartsuit\)};
\draw (0,-1) node[below]{\textbf{3-2-12}};
\end{tikzpicture}
\begin{tikzpicture}[scale=0.6]
\draw[dotted] (-2,2) grid (2,-2);
\draw (0,0) node{+};
\draw[thick] (1,-1) -- (1,1) -- (0,1) -- (-1,0) -- cycle;
\draw (1,1) node[color=teal]{\(\clubsuit\)};
\draw (1,-1) node[color=purple]{\(\varheartsuit\)};
\draw (0,-1) node[below]{\textbf{3-2-13}};
\end{tikzpicture}
\begin{tikzpicture}[scale=0.6]
\draw[dotted] (-2,2) grid (2,-2);
\draw (0,0) node{+};
\draw[thick] (1,-1) -- (1,1) -- (0,1) -- (-1,0) -- (0,-1) -- cycle;
\draw (1,1) node[color=teal]{\(\clubsuit\)};
\draw (1,-1) node[color=purple]{\(\varheartsuit\)};
\draw (0,-1) node[below]{\textbf{3-2-14}};
\end{tikzpicture}
\end{center}

\begin{exa}
The embedding \textbf{3-2-12} is \(\bbP^3\), equipped with the following action: 
\[ \SL_2\times \bbG_m \times\bbC^4 \ni \left(\begin{bmatrix}a&b\\c&d\end{bmatrix}, z, \begin{pmatrix}x_0\\x_1\\x_2\\x_3\end{pmatrix}\right) \mapsto \begin{pmatrix}z(ax_0+bx_1)\\z(cx_0+dx_1)\\ax_2+bx_3\\cx_2+dx_3\end{pmatrix} \]
Indeed, the stabilizer of \((1,0,0,1)\) under this action is precisely the subgroup of \(T\) defined by \(za^2=1\). 

The embedding \textbf{3-2-13} is obtained from \(\bbP^3\) by blowing up one of the two closed orbits, and the embedding \textbf{3-2-14} by blowing up both closed orbits. 
\end{exa}

\textbf{When \(a_1\geq 3\)}, by Proposition~\ref{prop_LFRP_T_even>0}, Proposition~\ref{prop_LFRP_T_odd}, and the list of reflexive Fano polygons, there are no locally factorial \(G/H\)-reflexive embeddings. 

\subsection{Case when \(H_1=N(T_1)\)}

In this section, we will use the slightly abusive notation 
\(\diag(N(T_1))\times  \{1_{\bbG_m^{n-1}}\}=\langle T_1\times \{1\}, \left(\begin{bmatrix} 0&-1\\ 1&0\end{bmatrix}, -1 \right)\rangle \times \{1_{\bbC^{n-1}}\}\) for conciseness. 

\subsubsection{Combinatorial data}

\begin{cor}
For \(H=N(T_1)\times \{1\}\), the combinatorial data is:
\[\arraycolsep=12pt
\begin{array}{lll}
M=\langle 2\alpha_1,\chi_1, \ldots, \chi_n \rangle & \zeta(\clubsuit) = \{\alpha_1\} &    \\
\Sigma = \{2\alpha_1\} & \rho(\clubsuit) = (2,0,\ldots,0) &  \\
\mathcal{D} =\{\clubsuit\} &  \\
\midrule
m_{\clubsuit} = 1  & f=2+4x_1 & \kappa = \alpha_1  \\
\end{array}
\]
For \(H=\diag(N(T_1))\times  \{1_{\bbG_m^{n-1}}\}\), the combinatorial data is as follows: 
\[\arraycolsep=12pt
\begin{array}{lll}
M=\langle \alpha_1+\chi_1,\alpha_1-\chi_1, \chi_2 \ldots, \chi_n \rangle & \zeta(\clubsuit) = \{\alpha_1\} &    \\
\Sigma = \{2\alpha_1\} & \rho(\clubsuit) = (1,1,\ldots,0) &  \\
\mathcal{D} =\{\clubsuit\} &  \\
\midrule
m_{\clubsuit} = 1  & f=2+2x_1+2x_2 & \kappa = \alpha_1  \\
\end{array}
\]
\end{cor}

\begin{proof}
The spherical homogeneous spaces are obtained by quotienting the previous case with \(a_1=0\). 
It is straightforward to deduce the spherical lattice, and the set of colors (the two initial colors are exchanged by the action we quotient from). 
The spherical root is changed as the lattice changes, but since we have \(\alpha_1\in \frac{1}{2}\Sigma\), we still have \(m_{D_1}=1\). 
\end{proof}

\subsubsection{Polytopes}

We still denote by \((y_1,\ldots,y_{n+1})\) the coordinates on \(N\otimes \bbR\), in the basis dual to the \(\bbZ\)-basis of \(M\) given in the combinatorial data.  
We first consider the case when \(H=N(T_1)\times \{1_{\bbG_m^n}\}\), 

\begin{prop}
\label{prop_LFRP_N_product}
A polytope \(\Omega\) is a locally factorial \(G/H\)-reflexive polytope if and only if \(\Omega\) is a polytope such that \((2,0,\ldots,0)\) is a vertex, all other vertices are primitive elements in the half-space \(H_-=\{y_1\leq 0\}\), all the facets containing \((2,0,\ldots,0)\) are fully contained in the half-space \(H_+=\{y_1\geq 0\}\), and for any facet, the collection of primitive multiples of its vertices forms a basis of \(\bbZ^{n+1}\). 
Locally factorial Fano embeddings of \(G/H\) correspond uniquely to such polytopes up to the action of \(\GL_n(\bbZ)\) on the last \(n\) coordinates. 
\end{prop}

Note that if we remove the vertex \((2,0,\ldots,0)\) to replace it with \((1,0,\ldots,0)\), we obtain a smooth Fano polytope which satisfies the properties in Proposition~\ref{prop_LFRP_T_0}. 
It is not obvious that the reverse construction always works, although we will check that it is the case if \(n=1\). 

\begin{proof}
The statement essentially follows from the definition of locally factorial \(G/H\)-reflexive polytopes, Proposition~\ref{prop_combi_SL2:T} and our identification of the lattice \(N\) with \(\bbZ^{n+1}\), as in the previous cases. 
There are two particular things to note. 
First, the image of the color \(\clubsuit\) is not primitive, and \(m_{\clubsuit}=1\), so the corresponding points \(\frac{\rho(\clubsuit)}{m_{\clubsuit}}=(2,0,\ldots,0)\) cannot be in a facet that intersects the interior of the valuation cone. 
Second, if a facet contains \((2,0,\ldots,0)\) (necessarily as a vertex), then the other vertices are on the hyperplane \(H_-\cap H_+\) and form a face of codimension two, which must belong to another facet, contained in \(H_-\). 
Since the vertices of that second facet must form a basis of \(\bbZ^{n+1}\), the vertices in \(H_-\cap H_+\) form a basis of \(\{0\}\times \bbZ^n\). 
This basis, completed by \((1,0,\ldots,0)=\frac{1}{2}(2,0,\ldots,0)\) forms a basis of \(\bbZ^{n+1}\). 

The last sentence of the statement follows from Remark~\ref{rem_SL2_non-uniqueness_of_chi}. 
\end{proof}

We now turn to the case when \(H=\diag(N(T_1))\times \{1_{\bbG_m^{n-1}}\}\). 

\begin{prop}
\label{prop_LFRP_N_diag}
A polytope \(\Omega\) is a locally factorial \(G/H\)-reflexive polytope if and only if \(\Omega\) is a simplicial canonical Fano polytope such that \((1,1,0\ldots,0)\) is a vertex, all other vertices are in the half-space \(H_-=\{y_2\leq y_1\}\), any non-smooth facet of \(\Omega\) contains \((1,1,0\ldots,0)\) and is fully contained in the half-space \(H_+=\{y_2\geq y_1\}\). 
Furthermore, any non-smooth facet contains at most one integral point which is not a vertex. 
Locally factorial Fano embeddings of \(G/H\) correspond uniquely to such polytopes up to the action of the subgroup of \(\GL_{n+1}(\bbZ)\) that fixes \((1,1,0\ldots, 0)\), and stabilizes both \(\bbZ^2\times \{0\}\) and the hyperplane \(H_+\cap H_-\). 
\end{prop}

\begin{proof}
It suffices to describe the non-smooth facets, when they exist. 
If \(F\) is such a facet, then it follows directly from the definition that it must be fully contained in \(H_+\), and that \((1,1,0,\ldots, 0)\) is a vertex. 
By the same argument as in the previous proofs, the other vertices of \(F\) form a basis of \(\langle (1,-1,0\ldots, 0), (0,0)\times\bbZ^{n-1}\rangle\), hence the facet is a simplex and contains at most one integral point in its interior. 
The last sentence of the statement always follows from Remark~\ref{rem_SL2_non-uniqueness_of_chi}. 
\end{proof}

\begin{exa}
As in Example~\ref{exa_pic1_rk3}, there are no simplices that are locally factorial \(G/H\)-reflexive for \(H=N(T_1)\times \{1_{\bbG_m^n}\}\). 
In the case \(H=\diag(N(T_1))\times \{1_{\bbG_m^{n-1}}\}\), let us just give an example showing that there are such simplices: when \(n=2\), consider the simplex with vertices \((1,1,0)\), \((0,-1,0)\), \((0,0,1)\), \((-1,0,-1)\). 
It is a smooth simplex which obviously satisfies all the conditions, hence provides an example of a Fano, locally factorial, Picard rank one embedding of \(G/H\). 
The determination of all locally factorial \(G/H\)-reflexive simplices requires a finer analysis that we leave for subsequent work. 
Note that already in the rank two case that follows, there are two possible locally factorial \(G/H\)-reflexive simplices (3.18 and 3.20). 
\end{exa}

\subsubsection{Rank 2 locally factorial Fano embeddings}

We first deal with the case when \(H=N(T_1)\times \{1\}\subset \SL_2\times \bbG_m=G\). 
Using Proposition~\ref{prop_LFRP_N_product}, we note that, if \(\Omega\) is a locally factorial \(G/H\)-reflexive polytope, then by replacing the vertex \((2,0,\ldots, 0)\) with \((1,0,\ldots, 0)\), we obtain a bona fide smooth Fano polytope \(\Omega'\), such that all the facets containing \((1,0,\ldots,0)\) are fully contained in \(\{x_0\geq 0\}\).

We can thus run through the list of smooth Fano polygons to obtain the list of polytopes. 
We obtain three possible embeddings. 

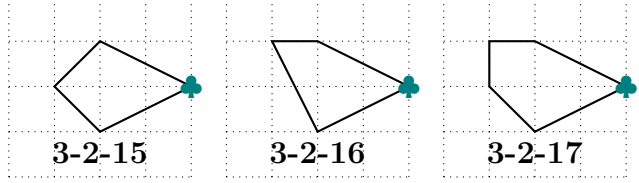
\begin{center}
\begin{tikzpicture}[scale=0.6]
\draw[dotted] (-2,2) grid (2,-2);
\draw[thick] (2,0) -- (0,1) -- (-1,0) -- (0,-1) -- cycle;
\draw (2,0) node[color=teal]{\(\clubsuit\)};
\draw (0,-1) node[below]{\textbf{3-2-15}};
\end{tikzpicture}
\begin{tikzpicture}[scale=0.6]
\draw[dotted] (-2,2) grid (2,-2);
\draw[thick] (2,0) -- (0,1) -- (-1,1) -- (0,-1) -- cycle;
\draw (2,0) node[color=teal]{\(\clubsuit\)};
\draw (0,-1) node[below]{\textbf{3-2-16}};
\end{tikzpicture}
\begin{tikzpicture}[scale=0.6]
\draw[dotted] (-2,2) grid (2,-2);
\draw[thick] (2,0) -- (0,1) -- (-1,1) -- (-1,0) -- (0,-1) -- cycle;
\draw (2,0) node[color=teal]{\(\clubsuit\)};
\draw (0,-1) node[below]{\textbf{3-2-17}};
\end{tikzpicture}
\end{center}

\begin{exa}
The embedding \textbf{3-2-15} is the product of embeddings \(\bbP^2\times \bbP^1\), and \textbf{3-2-17} is its blowup along one closed orbit, a product of a quadric in \(\bbP^2\) with a torus fixed point in \(\bbP^1\). 
\end{exa}

We now deal with the case when \(H=\diag(N(T_1))\subset \SL_2\times \bbG_m=G\).
By Proposition~\ref{prop_LFRP_N_diag}, it suffices to run through the list of 16 reflexive polygons up to \(\GL_2(\bbZ)\)-action, and determine which representatives satisfy the properties. 
We obtain:

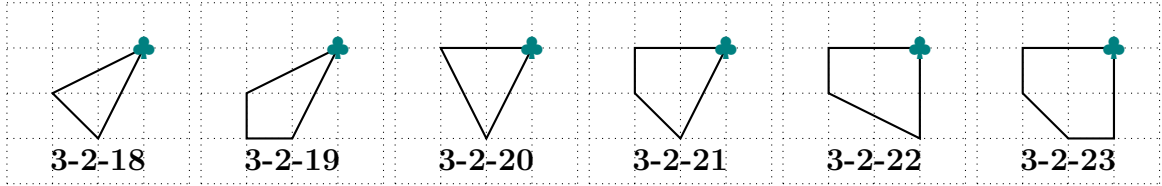
\begin{center}
\begin{tikzpicture}[scale=0.6]
\draw[dotted] (-2,2) grid (2,-2);
\draw[thick] (1,1) -- (-1,0) -- (0,-1) -- cycle;
\draw (1,1) node[color=teal]{\(\clubsuit\)};
\draw (0,-1) node[below]{\textbf{3-2-18}};
\end{tikzpicture}
\begin{tikzpicture}[scale=0.6]
\draw[dotted] (-2,2) grid (2,-2);
\draw[thick] (1,1) -- (-1,0) -- (-1,-1) -- (0,-1) -- cycle;
\draw (1,1) node[color=teal]{\(\clubsuit\)};
\draw (0,-1) node[below]{\textbf{3-2-19}};
\end{tikzpicture}
\begin{tikzpicture}[scale=0.6]
\draw[dotted] (-2,2) grid (2,-2);
\draw[thick] (1,1) -- (-1,1) -- (0,-1) -- cycle;
\draw (1,1) node[color=teal]{\(\clubsuit\)};
\draw (0,-1) node[below]{\textbf{3-2-20}};
\end{tikzpicture}
\begin{tikzpicture}[scale=0.6]
\draw[dotted] (-2,2) grid (2,-2);
\draw[thick] (1,1) -- (-1,1) -- (-1,0) -- (0,-1) -- cycle;
\draw (1,1) node[color=teal]{\(\clubsuit\)};
\draw (0,-1) node[below]{\textbf{3-2-21}};
\end{tikzpicture}
\begin{tikzpicture}[scale=0.6]
\draw[dotted] (-2,2) grid (2,-2);
\draw[thick] (1,1) -- (-1,1) -- (-1,0) -- (1,-1) -- cycle;
\draw (1,1) node[color=teal]{\(\clubsuit\)};
\draw (0,-1) node[below]{\textbf{3-2-22}};
\end{tikzpicture}
\begin{tikzpicture}[scale=0.6]
\draw[dotted] (-2,2) grid (2,-2);
\draw[thick] (1,1) -- (-1,1) -- (-1,0) -- (0,-1) -- (1,-1) -- cycle;
\draw (1,1) node[color=teal]{\(\clubsuit\)};
\draw (0,-1) node[below]{\textbf{3-2-23}};
\end{tikzpicture}
\end{center}

\begin{exa}
The embedding \textbf{3-2-18} is the quadric \(Q^3\), as explained in details in \cite{infinite} for example. 
The embedding \textbf{3-2-19} is obtained by blowing up a one dimensional subquadric, while \textbf{3-2-23} is obtained by blowing up a zero dimensional subquadric. Since a zero dimensional quadric consists of two distinct points, there is also an intermediate embedding given by blowing up only one of these two points, this is \textbf{3-2-21}. 
\end{exa}

\begin{exa}
The embedding \textbf{3-2-20} is \(\bbP^3\). 
To describe the action, one may consider the action of \(\SL_2\) on \(\bbC^3\) induced by the exceptional morphism \(\SL_2\to \SO_3\). 
The action of \(\SO_3\) preserves a non-degenerate quadratic form on \(\bbC^3\), preserving in particular its level sets. 
Throwing in the action of \(\bbG_m\) on \(\bbC^3\) by dilation, we obtain an action of \(\SL_2\times \bbG_m\) on \(\bbC_3\) with an open dense orbit formed by the union of non-zero level sets, and a closed orbit formed by the affine quadric defined by the quadratic form. 
Compactifying into \(\bbP^3\) yields  \textbf{3-2-20}. 

There are two closed orbits that may be blown up, yielding the embeddings  \textbf{3-2-21} and \textbf{3-2-22}.
\end{exa}

\subsection{Horospherical case (when \(H_1=B_1\))}
\label{section_SL_2_horospherical}

\subsubsection{Combinatorial data}

\begin{prop}
Let \(H=\ker (a_1\varpi_1+\chi_1)|_{B_1\times \bbG_m} \times \{1_{\bbG_m^{n-1}}\}\) for some \(a_1\in \bbZ_{\geq 0}\). 
Then the combinatorial data of the spherical homogeneous space \(G/H\) is as follows:
\[\arraycolsep=12pt
\begin{array}{lll}
M=\langle a_1\varpi_1+\chi_1, \chi_2, \ldots, \chi_n\rangle  &  \rho(\clubsuit)=(a_1,0,\ldots,0) 
\\  
\Sigma=\emptyset & \zeta(\clubsuit)=\{\alpha_1\} 
\\ 
\mathcal{D} = \{\clubsuit\} &  
\\
\midrule
\kappa = \alpha_1 & f = 2+a_1x_1 & m_{\clubsuit}=2
\end{array}
\]
\end{prop}

\begin{proof}
This follows straightforwardly from the definition of \(H\) and Proposition~\ref{prop_parabolic_induction}.
\end{proof}

\subsubsection{Locally factorial \(G/H\)-reflexive polytopes}

\begin{prop}
Let \(\Omega\) be a locally factorial \(G/H\)-reflexive polytope, where \(H=\ker (a_1\varpi_1+\chi_1)|_{B_1\times \bbG_m} \times \{1_{\bbG_m^{n-1}}\}\) for some \(a_1\in \bbZ_{\geq 0}\). 
Then there are two possibilities: 
\begin{enumerate}
    \item \(a_1\in \{0,1\}\), \(\Omega\) is a smooth Fano polytope and \((a_1/2,0,\ldots,0) \in \Int(\Omega)\), or 
    \item \(a_1=1\), the vertices of \(\Omega\) are in \(\bbZ^n\cup \{(1/2,0,\ldots,0)\}\), \(0\in \Int(\Omega)\), and for any facet \(F\) of \(\Omega\), the collection of primitive positive multiples of vertices of \(F\) forms a basis of \(N\).     
\end{enumerate}
\end{prop}

\begin{proof}
Since \(G/H\) is horospherical, the interior of any facet meets the valuation cone. 
It follows that the polytope cannot contain points of \(N\) except its vertices and the origin. 
In particular, since \(\frac{\rho(\clubsuit)}{m_{\clubsuit}}\in \Omega\), we must have \(a_1\leq 2\). 
Furthermore, if \(a_1=2\), then \(\frac{\rho(\alpha_1)}{2}\) is a vertex, but then \(\rho(\alpha_1)\) is not a primitive element of \(N\), so cannot be part of a basis of \(N\), a contradiction. 
\end{proof}

We determine the possible polytopes in rank two, that is, when \(\dim(G/H)=3\). 

\textbf{When} \(a_1=0\), we get all, and only get, the products of \(\bbP^1\) with one of the five smooth Fano toric surfaces, with the product action. 
We number these five examples \textbf{3-2-24} to \textbf{3-2-28}. 

\textbf{When} \(a_1=1\) \textbf{and} \((1/2,0,\ldots,0)\) \textbf{is not a vertex}, we have to determine where \((1/2,0,\ldots,0)\) can lie in the interior of the polytope. 
It is an element of the lattice \(\frac{1}{2}\bbZ^n\). 
It suffices to go through the list of smooth Fano polygons up to \(\GL_2(\bbZ)\) action. 
Once the element of \(\frac{1}{2}\bbZ^2\) corresponding to \(\frac{\rho(\clubsuit)}{m_{\clubsuit}}\) is fixed, the \(\GL_2(\bbZ)\) action reduces to the action of matrices of the form \(\begin{bmatrix} 1 & k \\ 0 & \pm 1 \end{bmatrix}\) for \(k\in \bbZ\). 
Thanks to Remark~\ref{rem_SL2_non-uniqueness_of_chi}, we can work up to this action. 
From the list of Fano polygons, we obtain the following nine possibilities. 

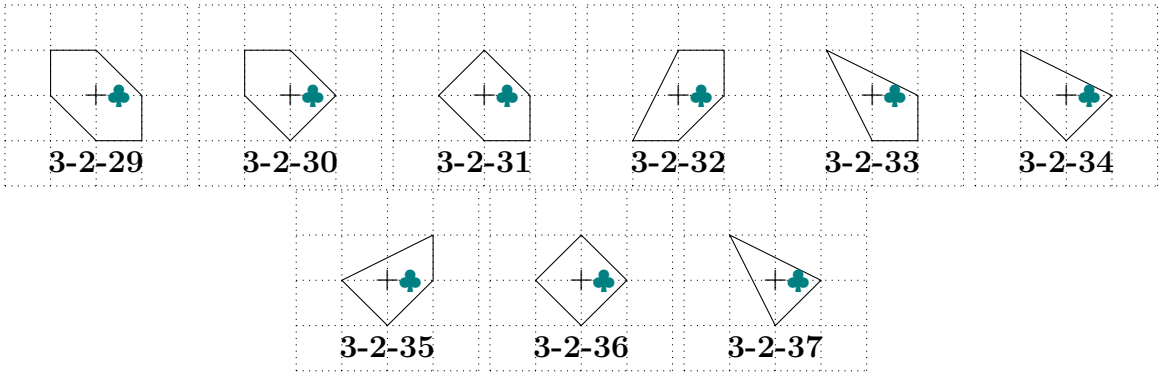
\begin{center}
\begin{tikzpicture}[scale=0.6]
\draw[dotted] (-2,2) grid (2,-2);
\draw (0,0) node{+};
\draw (0,1) -- (-1,1) -- (-1,0) -- (0,-1) -- (1,-1) -- (1,0) -- cycle;
\draw (.5,0) node[color=teal]{\(\clubsuit\)};
\draw (0,-1) node[below]{\textbf{3-2-29}};
\end{tikzpicture}
\begin{tikzpicture}[scale=0.6]
\draw[dotted] (-2,2) grid (2,-2);
\draw (0,0) node{+};
\draw (0,1) -- (-1,1) -- (-1,0) -- (0,-1) -- (1,0) -- cycle;
\draw (.5,0) node[color=teal]{\(\clubsuit\)};
\draw (0,-1) node[below]{\textbf{3-2-30}};
\end{tikzpicture}
\begin{tikzpicture}[scale=0.6]
\draw[dotted] (-2,2) grid (2,-2);
\draw (0,0) node{+};
\draw (0,1) -- (-1,0) -- (0,-1) -- (1,-1) -- (1,0) -- cycle;
\draw (.5,0) node[color=teal]{\(\clubsuit\)};
\draw (0,-1) node[below]{\textbf{3-2-31}};
\end{tikzpicture}
\begin{tikzpicture}[scale=0.6]
\draw[dotted] (-2,2) grid (2,-2);
\draw (0,0) node{+};
\draw (1,0) -- (1,1) -- (0,1) -- (-1,-1) -- (0,-1) -- cycle;
\draw (.5,0) node[color=teal]{\(\clubsuit\)};
\draw (0,-1) node[below]{\textbf{3-2-32}};
\end{tikzpicture}
\begin{tikzpicture}[scale=0.6]
\draw[dotted] (-2,2) grid (2,-2);
\draw (0,0) node{+};
\draw (-1,1) -- (0,-1) -- (1,-1) -- (1,0) -- cycle;
\draw (.5,0) node[color=teal]{\(\clubsuit\)};
\draw (0,-1) node[below]{\textbf{3-2-33}};
\end{tikzpicture}
\begin{tikzpicture}[scale=0.6]
\draw[dotted] (-2,2) grid (2,-2);
\draw (0,0) node{+};
\draw (-1,1) -- (-1,0) -- (0,-1) -- (1,0) -- cycle;
\draw (.5,0) node[color=teal]{\(\clubsuit\)};
\draw (0,-1) node[below]{\textbf{3-2-34}};
\end{tikzpicture}

\begin{tikzpicture}[scale=0.6]
\draw[dotted] (-2,2) grid (2,-2);
\draw (0,0) node{+};
\draw (1,0) -- (1,1) -- (-1,0) -- (0,-1) -- cycle;
\draw (.5,0) node[color=teal]{\(\clubsuit\)};
\draw (0,-1) node[below]{\textbf{3-2-35}};
\end{tikzpicture}
\begin{tikzpicture}[scale=0.6]
\draw[dotted] (-2,2) grid (2,-2);
\draw (0,0) node{+};
\draw (1,0) -- (0,1) -- (-1,0) -- (0,-1) -- cycle;
\draw (.5,0) node[color=teal]{\(\clubsuit\)};
\draw (0,-1) node[below]{\textbf{3-2-36}};
\end{tikzpicture}
\begin{tikzpicture}[scale=0.6]
\draw[dotted] (-2,2) grid (2,-2);
\draw (0,0) node{+};
\draw (1,0) -- (-1,1) -- (0,-1) -- cycle;
\draw (.5,0) node[color=teal]{\(\clubsuit\)};
\draw (0,-1) node[below]{\textbf{3-2-37}};
\end{tikzpicture}
\end{center}

\begin{exa}
The embedding \textbf{3-2-36} is the product of embeddings \(\bbP^1\times \bbF_1\). It may be blown up along one closed orbit to obtain either the embedding \textbf{3-2-31} or the embedding \textbf{3-2-30}. 
Blowing up wisely two closed orbits yields \textbf{3-2-29}.
\end{exa}

Finally, we deal with the case \textbf{when \((1/2,0,\ldots,0)\) is a vertex} of the polytope. 
In this case, we will show how to directly find the polytopes instead of relying on a known classification result. 
A variant of the method used can essentially be used in all rank two cases, and we will give less details in later occurrences. 

By Remark~\ref{rem_SL2_non-uniqueness_of_chi}, we may assume that \((0,1)\) is a vertex of a facet of \(\Omega\) that contains \((1/2,0,\ldots,0)\). 
An immediate consequence is that the whole polytope \(\Omega\) lies under the line \(y=1-2x\) if \((x,y)\) denote the coordinates in \(\bbQ^2\). 
We can be more precise by observing that the other vertex adjacent to \(\frac{\rho(\alpha_1)}{2}\) must be of the form \((l,-1)\) for some \(l\in \bbZ\). 
It lies strictly below the line \(y=1-2x\), hence \(l\leq 0\), and the whole polytope \(\Omega\) actually lies in the convex set defined by the two linear inequalities \(y\leq 1-2x\) and \(y\geq -1+2x\). 

\begin{multicols}{2}
We now consider the next vertex \((x_3,y_3)\) (in counter-clockwise order) adjacent to \((0,1)\). 
In order to form a basis with \((0,1)\), we must have \(x_3=-1\). 
In view of the restriction on \(\Omega\), its coordinate \(y_3\), which is integral, must satisfy \(-3\leq y\leq 2\). 
We thus have six possible choices for \((x_3,y_3)\). 

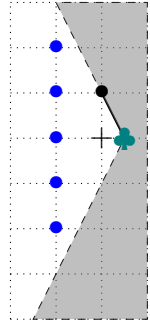
\begin{tikzpicture}[scale=0.6]
\draw[dotted] (-2,3) grid (1,-4);
\draw (0,0) node{+};
\draw[thick] (1/2,0) -- (0,1);
\draw[dashed] (1/2,0) -- (-1,3) -- (1,3) -- (1,-4) -- (-3/2,-4) -- cycle;
\fill[gray, opacity=.5] (1/2,0) -- (-1,3) -- (1,3) -- (1,-4) -- (-3/2,-4) -- cycle;
\draw (1/2,0) node[color=teal]{\(\clubsuit\)};
\draw (0,1) node{\(\bullet\)};
\draw[blue] (-1,2) node{\(\bullet\)};
\draw[blue] (-1,1) node{\(\bullet\)};
\draw[blue] (-1,0) node{\(\bullet\)};
\draw[blue] (-1,-1) node{\(\bullet\)};
\draw[blue] (-1,-2) node{\(\bullet\)};
\end{tikzpicture}
\end{multicols}

Assume that \((x_3,y_3)=(-1,2)\). 
Then consider the next adjacent vertex \((x_4,y_4)\). 
In view of the added known linear inequality \(y\leq 1-x\) defining \(\Omega\), with supporting facet the hull of \((0,1)\) and \((-1,2)\), and since \((-1,2)\) and \((x_4,y_4)\) must form a basis of \(\bbZ^2\), the only possibilities are \((x_4,y_4)=(-1,1)\) and \((x_4,y_4)=(0,-1)\). 
If \((x_4,y_4)=(0,-1)\), then it must also be the other vertex adjacent to \(\frac{\rho(\alpha_1)}{2}\), and we have exhausted the vertices of \(\Omega\). 
If \((x_4,y_4)=(-1,1)\), then there is no choice in adding a final integral vertex to obtain a locally factorial \(G/H\)-reflexive polytope, and that vertex is \((0,-1)\). 

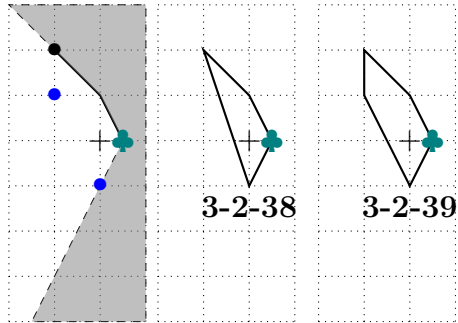
\begin{center}
\begin{tikzpicture}[scale=0.6]
\draw[dotted] (-2,3) grid (1,-4);
\draw (0,0) node{+};
\draw[thick] (1/2,0) -- (0,1) -- (-1,2);
\draw[dashed] (1/2,0) -- (0,1) -- (-2,3) -- (1,3) -- (1,-4) -- (-3/2,-4) -- cycle;
\fill[gray, opacity=.5] (1/2,0) -- (0,1) -- (-2,3) -- (1,3) -- (1,-4) -- (-3/2,-4) -- cycle;
\draw (1/2,0) node[color=teal]{\(\clubsuit\)};
\draw (-1,2) node{\(\bullet\)};
\draw[blue] (-1,1) node{\(\bullet\)};
\draw[blue] (0,-1) node{\(\bullet\)};
\end{tikzpicture}
\begin{tikzpicture}[scale=0.6]
\draw[dotted] (-2,3) grid (1,-4);
\draw (0,0) node{+};
\draw[thick] (1/2,0) -- (0,1) -- (-1,2) -- (0,-1) -- cycle;
\draw (1/2,0) node[color=teal]{\(\clubsuit\)};
\draw (0,-1) node[below]{\textbf{3-2-38}};
\end{tikzpicture}
\begin{tikzpicture}[scale=0.6]
\draw[dotted] (-2,3) grid (1,-4);
\draw (0,0) node{+};
\draw[thick] (1/2,0) -- (0,1) -- (-1,2) -- (-1,1) -- (0,-1) -- cycle;
\draw (1/2,0) node[color=teal]{\(\clubsuit\)};
\draw (0,-1) node[below]{\textbf{3-2-39}};
\end{tikzpicture}
\end{center}

Similarly, if \((x_3,y_3)=(-1,1)\), then the next vertex \((x_4,y_4)\) is either \((-1,0)\), in which case, once must add the fifth vertex \((0,-1)\) to get a locally factorial \(G/H\)-reflexive polytope, or it is \((0,-1)\) and one cannot add any vertex. 

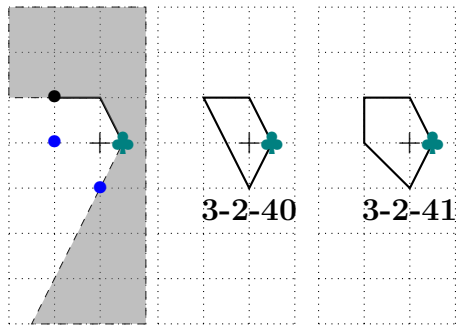
\begin{center}
\begin{tikzpicture}[scale=0.6]
\draw[dotted] (-2,3) grid (1,-4);
\draw (0,0) node{+};
\draw[thick] (1/2,0) -- (0,1) -- (-1,1);
\draw[dashed] (1/2,0) -- (0,1) -- (-2,1) -- (-2,3) -- (1,3) -- (1,-4) -- (-3/2,-4) -- cycle;

\fill[gray, opacity=.5] (1/2,0) -- (0,1) -- (-2,1) -- (-2,3) -- (1,3) -- (1,-4) -- (-3/2,-4) -- cycle;
\draw (1/2,0) node[color=teal]{\(\clubsuit\)};
\draw (-1,1) node{\(\bullet\)};
\draw[blue] (-1,0) node{\(\bullet\)};
\draw[blue] (0,-1) node{\(\bullet\)};
\end{tikzpicture}
\begin{tikzpicture}[scale=0.6]
\draw[dotted] (-2,3) grid (1,-4);
\draw (0,0) node{+};
\draw[thick] (1/2,0) -- (0,1) -- (-1,1) -- (0,-1) -- cycle;
\draw (1/2,0) node[color=teal]{\(\clubsuit\)};
\draw (0,-1) node[below]{\textbf{3-2-40}};
\end{tikzpicture}
\begin{tikzpicture}[scale=0.6]
\draw[dotted] (-2,3) grid (1,-4);
\draw (0,0) node{+};
\draw[thick] (1/2,0) -- (0,1) -- (-1,1) -- (-1,0) -- (0,-1) -- cycle;
\draw (1/2,0) node[color=teal]{\(\clubsuit\)};
\draw (0,-1) node[below]{\textbf{3-2-41}};
\end{tikzpicture}
\end{center}

If \((x_3,y_3)=(-1,0)\), then the next vertex \((x_4,y_4)\) is either \((-1,-1)\) or it is \((0,-1)\) and one cannot add any vertex. 
If it is \((-1,-1)\), the next vertex \((x_5,y_5)\) is either \((0,-1)\) or \(\frac{\rho(\alpha_1)}{2}\), yielding two possible polytopes. 

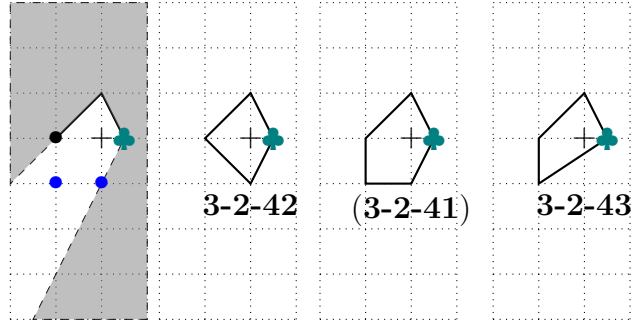
\begin{center}
\begin{tikzpicture}[scale=0.6]
\draw[dotted] (-2,3) grid (1,-4);
\draw (0,0) node{+};
\draw[thick] (1/2,0) -- (0,1) -- (-1,0);
\draw[dashed] (1/2,0) -- (0,1) -- (-2,-1) -- (-2,3) -- (1,3) -- (1,-4) -- (-3/2,-4) -- cycle;

\fill[gray, opacity=.5] (1/2,0) -- (0,1) -- (-2,-1) -- (-2,3) -- (1,3) -- (1,-4) -- (-3/2,-4) -- cycle;
\draw (1/2,0) node[color=teal]{\(\clubsuit\)};
\draw (-1,0) node{\(\bullet\)};
\draw[blue] (-1,-1) node{\(\bullet\)};
\draw[blue] (0,-1) node{\(\bullet\)};
\end{tikzpicture}
\begin{tikzpicture}[scale=0.6]
\draw[dotted] (-2,3) grid (1,-4);
\draw (0,0) node{+};
\draw[thick] (1/2,0) -- (0,1) -- (-1,0) -- (0,-1) -- cycle;
\draw (1/2,0) node[color=teal]{\(\clubsuit\)};
\draw (0,-1) node[below]{\textbf{3-2-42}};
\end{tikzpicture}
\begin{tikzpicture}[scale=0.6]
\draw[dotted] (-2,3) grid (1,-4);
\draw (0,0) node{+};
\draw[thick] (1/2,0) -- (0,1) -- (-1,0) -- (-1,-1) -- (0,-1) -- cycle;
\draw (1/2,0) node[color=teal]{\(\clubsuit\)};
\draw (0,-1) node[below]{(\textbf{3-2-41})};
\end{tikzpicture}
\begin{tikzpicture}[scale=0.6]
\draw[dotted] (-2,3) grid (1,-4);
\draw (0,0) node{+};
\draw[thick] (1/2,0) -- (0,1) -- (-1,0) -- (-1,-1) -- cycle;
\draw (1/2,0) node[color=teal]{\(\clubsuit\)};
\draw (0,-1) node[below]{\textbf{3-2-43}};
\end{tikzpicture}
\end{center}

If \((x_3,y_3)=(-1,-1)\), then the next vertex \((x_4,y_4)\) is either \((-1,-2)\), \((0,-1)\) or \(\frac{\rho(\alpha_1)}{2}\). 
A quick analysis of the next vertices shows that there is only one corresponding polytope in each case. 

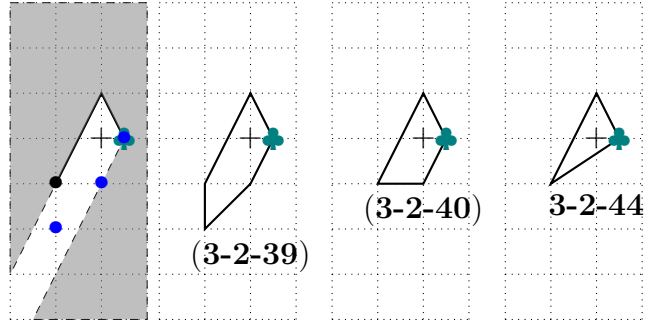
\begin{center}
\begin{tikzpicture}[scale=0.6]
\draw[dotted] (-2,3) grid (1,-4);
\draw (0,0) node{+};
\draw[thick] (1/2,0) -- (0,1) -- (-1,-1);
\draw[dashed] (1/2,0) -- (0,1) -- (-2,-3) -- (-2,3) -- (1,3) -- (1,-4) -- (-3/2,-4) -- cycle;

\fill[gray, opacity=.5] (1/2,0) -- (0,1) -- (-2,-3) -- (-2,3) -- (1,3) -- (1,-4) -- (-3/2,-4) -- cycle;
\draw (1/2,0) node[color=teal]{\(\clubsuit\)};
\draw (-1,-1) node{\(\bullet\)};
\draw[blue] (-1,-2) node{\(\bullet\)};
\draw[blue] (0,-1) node{\(\bullet\)};
\draw[blue] (1/2,0) node{\(\bullet\)};
\end{tikzpicture}
\begin{tikzpicture}[scale=0.6]
\draw[dotted] (-2,3) grid (1,-4);
\draw (0,0) node{+};
\draw[thick] (1/2,0) -- (0,1) -- (-1,-1) -- (-1,-2) -- (0,-1) -- cycle;
\draw (1/2,0) node[color=teal]{\(\clubsuit\)};
\draw (0,-2) node[below]{(\textbf{3-2-39})};
\end{tikzpicture}
\begin{tikzpicture}[scale=0.6]
\draw[dotted] (-2,3) grid (1,-4);
\draw (0,0) node{+};
\draw[thick] (1/2,0) -- (0,1) -- (-1,-1) -- (0,-1) -- cycle;
\draw (1/2,0) node[color=teal]{\(\clubsuit\)};
\draw (0,-1) node[below]{(\textbf{3-2-40})};
\end{tikzpicture}
\begin{tikzpicture}[scale=0.6]
\draw[dotted] (-2,3) grid (1,-4);
\draw (0,0) node{+};
\draw[thick] (1/2,0) -- (0,1) -- (-1,-1) -- cycle;
\draw (1/2,0) node[color=teal]{\(\clubsuit\)};
\draw (0,-1) node[below]{\textbf{3-2-44}};
\end{tikzpicture}
\end{center}

Finally, if \((x_3,y_3)=(-1,-2)\), there is no choice but to add the final fourth vertex \((0,-1)\). 

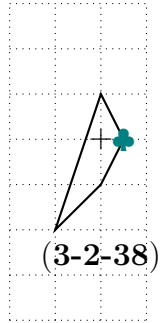
\begin{center}
\begin{tikzpicture}[scale=0.6]
\draw[dotted] (-2,3) grid (1,-4);
\draw (0,0) node{+};
\draw[thick] (1/2,0) -- (0,1) -- (-1,-2) -- (0,-1) -- cycle;
\draw (1/2,0) node[color=teal]{\(\clubsuit\)};
\draw (0,-2) node[below]{(\textbf{3-2-38})};
\end{tikzpicture}
\end{center}

In the above list of polytopes, we should work up to reflexion with respect to the \(x\)-axis, hence we obtain only seven different locally factorial Fano embeddings of \(G/H\). 

\begin{exa}
The embedding \textbf{3-2-42} is the product of embeddings \(\bbP^1\times \bbP^2\). 
The embedding \textbf{3-2-41} is its blowup along a closed orbit which is the product of a torus fixed point in \(\bbP^1\) with a fixed line in \(\bbP^2\). 
\end{exa}

\begin{exa}
The embedding \textbf{3-2-44} is \(\bbP^3\), equipped with the action 
\[ \left(\begin{bmatrix} a&b\\c&d\end{bmatrix},z\right) \cdot [x_0:x_1:x_2:x_3] = [ax_0+bx_1:cx_0+dx_1:z^{-1}x_2:x_3] \]
By blowing up closed orbits, one may for example recover the embeddings \textbf{3-2-40} and \textbf{3-2-43}. 
\end{exa}

The polytopes obtained in this section coincide with the list obtained by Pasquier in his PhD thesis \cite{Pasquier_2006}.

\section{Locally factorial Fano spherical \(\SL_2^2\times \bbG_m^n\)-fourfolds}
\label{section_SL22}

We assume throughout that \(G=\SL_2^2\times \bbG_m^n\). 

\subsection{Group compactifications}

\subsubsection{Combinatorial data}

\begin{prop}
The combinatorial data associated to the homogeneous space \(\SL_2^2/\diag \SL_2\) is as follows. 
\[\arraycolsep=12pt
\begin{array}{lll}
M=\langle \varpi_1+\varpi_2 \rangle  &  \rho(\clubsuit)=1 
\\  
\Sigma=\{\varpi_1+\varpi_2\} & \zeta(\clubsuit)=\{\alpha_1,\alpha_2\} 
\\ 
\mathcal{D} = \{\clubsuit\} &  
\\
\midrule
\kappa = \alpha_1+\alpha_2 & f = (2+x_1)^2 & m_{\clubsuit} = 2
\end{array}
\]
\end{prop}

\begin{proof}
The rank of \(\SL_2^2\) is two, and the homogeneous space is obviously not horospherical, hence it must have rank one, and \(\Sigma\) is a singleton consisting of a generator of \(M\). 

Consider the standard action of \(\SL_2^2\) on the space \(\Mat_{2\times 2}\) of square two by two matrices \(\begin{bmatrix} a& b\\ c & d\end{bmatrix}\) by left and right multiplication, in other words, the irreducible representation with highest weight \(\varpi_1+\varpi_2\). 
Consider the affine quadric \(\{ad-bc=1\}\) in \(\Mat_{2\times 2}\), that is, the set \(\SL_2\). 
As homogeneous space, it coincides with \(\SL_2^2/\diag \SL_2\). 
Under the action of the Borel subgroup of pairs of upper triangular matrices, there are two orbits: the set of upper triangular matrices and its complement. 
The set of upper triangular matrices, defined by the equation \(c=0\), is thus the unique color in the homogeneous space. We have obtained \(\mathcal{D}\) and deduce easily \(\zeta\) and \(\kappa\).  

There is a natural \(B\)-semi-invariant rational function, given by \(\begin{bmatrix} a& b\\ c & d\end{bmatrix}\mapsto c\). 
It has \(B\)-weight \(\varpi_1+\varpi_2\), and obviously vanishes to the order 1 on the color. 
As a consequence, it is indivisible and we have our description of \(M\), \(\Sigma\) and \(\rho\). 
From \(\Sigma\), we see that \(m_{D_1}=\langle \alpha_1^{\vee},\kappa \rangle = 2\). 
\end{proof}

\begin{cor}
The combinatorial data associated to the homogeneous space \(\PGL_2^2/\diag \PGL_2\) is as follows. 
\[\arraycolsep=12pt
\begin{array}{lll}
M=\langle \alpha_1+\alpha_2 \rangle  &  \rho(\clubsuit)=2 
\\  
\Sigma=\{\alpha_1+\alpha_2\} & \zeta(\clubsuit)=\{\alpha_1,\alpha_2\} 
\\ 
\mathcal{D} = \{\clubsuit\} &  
\\
\midrule
\kappa = \alpha_1+\alpha_2 & f = 4(1+x_1)^2 & m_{\clubsuit} = 2
\end{array}
\]
\end{cor}

\begin{proof}
Since it is obtained by quotienting the previous homogeneous space by the action of \((I_2,-I_2)\), and the unique color is necessarily stable under this action, we recover easily the combinatorial data. 
One could also consider the projective model \(\bbP(\Mat_{2\times 2})\).  
\end{proof}

\begin{cor}
The combinatorial data associated to the homogeneous space \(\GL_2^2/\diag \GL_2\) is as follows. 
\[\arraycolsep=12pt
\begin{array}{lll}
M=\langle \varpi_1+\varpi_2+\chi_1, \varpi_1+\varpi_2-\chi_1 \rangle  &  \rho(\clubsuit)=(1,1)
\\  
\Sigma=\{\varpi_1+\varpi_2\} & \zeta(\clubsuit)=\{\alpha_1,\alpha_2\} 
\\ 
\mathcal{D} = \{\clubsuit\} &  
\\
\midrule
\kappa = \alpha_1+\alpha_2 & f = (2+x_1+x_2)^2 & m_{\clubsuit} = 2
\end{array}
\]
\end{cor}

\begin{proof}
This is the only possibility between \(\diag \SL_2 \times \{1\}\) and \(N(\diag \SL_2) \times \{\pm 1\}\), whose combinatorial data we described above.  
\end{proof}

\subsubsection{Polytopes}

\textbf{In rank one} there is only one complete equivariant embedding for each, and it is smooth and Fano. We almost described them already: for \(\SL_2^2/\diag(\SL_2)\), it is the three dimensional quadric defined by \(ad-bc=e^2\) in \(\bbP^4\) with homogeneous coordinates \([a:b:c:d:e]\), where the affine chart \(e=1\) is identified with \(\Mat_{2\times 2}\) equipped with the standard action of \(\SL_2^2\). There are only two orbits, the closed orbit being the smooth quadric \(\{ad-bc=0\}\) in the hyperplane at infinity \(\{e=0\}\), equivariantly isomorphic to \(\bbP^1\times \bbP^1\). The corresponding polytope is: 
\begin{center}
\begin{tikzpicture}
\draw[dotted] (-2,0) -- (2,0);
\draw (0,0) node{+};
\foreach \i in {0,...,4}
{
\draw[dotted] (-2+\i,.3) -- (-2+\i,-.3);
}
\draw[thick] (-1,0) -- (1/2,0);
\draw (1/2,0) node[color=teal]{\(\clubsuit\)};
\draw (0,0) node[below]{\textbf{3-1-1}};
\end{tikzpicture}
\end{center}

For \(\SL_2^2/N(\diag\SL_2)\) this is \(\bbP(\Mat_{2\times 2})\) equipped with the standard action. There are only two orbits, the closed orbit being the projectivization of rank one matrices, also equivariantly isomorphic to \(\bbP^1\times \bbP^1\). The corresponding polytope is: 
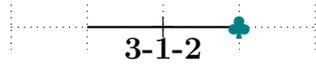
\begin{center}
\begin{tikzpicture}
\draw[dotted] (-2,0) -- (2,0);
\draw (0,0) node{+};
\foreach \i in {0,...,4}
{
\draw[dotted] (-2+\i,.3) -- (-2+\i,-.3);
}
\draw[thick] (-1,0) -- (1,0);
\draw (1,0) node[color=teal]{\(\clubsuit\)};
\draw (0,0) node[below]{\textbf{3-1-2}};
\end{tikzpicture}
\end{center}

\textbf{In rank two, for \(H=\diag\SL_2\times \{1\}\)}, 
we are considering equivariant compactification of the group \(\SL_2\times \bbG_m\). 
We determine the locally factorial \(G/H\)-reflexive polytopes using the same method as in Section~\ref{section_SL_2_horospherical}. 
The point \(\frac{\rho(\clubsuit)}{m_{\clubsuit}}=(1/2,0)\) must be a vertex for the origin to be an interior point. 
Working up to reflection by Remark~\ref{rem_SL_22_non-uniqueness_of_chi}, we obtain seven locally factorial Fano embeddings, corresponding to the following seven polytopes. 
Another way to obtain these polytopes, and a more general result for Fano embeddings of rank two symmetric spaces, was obtained by Ruzzi in \cite{Ruzzi_2012}.  
In particular, it follows from Ruzzi's article that the last two are locally factorial and not smooth. 
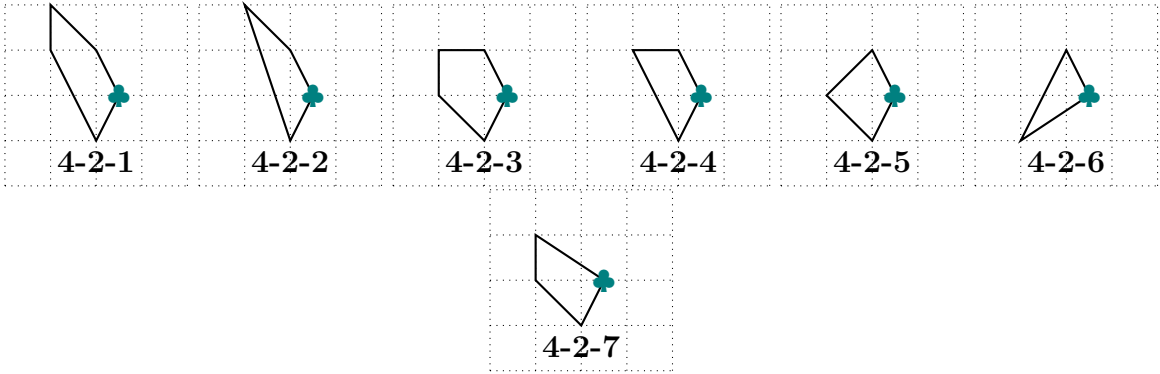
\begin{center}
\begin{tikzpicture}[scale=0.6]
\draw[dotted] (-2,2) grid (2,-2);
\draw[thick] (1/2,0) -- (0,1) -- (-1,2) -- (-1,1) -- (0,-1) -- cycle;
\draw (1/2,0) node[color=teal]{\(\clubsuit\)};
\draw (0,-1) node[below]{\textbf{4-2-1}};
\end{tikzpicture}
\begin{tikzpicture}[scale=0.6]
\draw[dotted] (-2,2) grid (2,-2);
\draw[thick] (1/2,0) -- (0,1) -- (-1,2) -- (0,-1) -- cycle;
\draw (1/2,0) node[color=teal]{\(\clubsuit\)};
\draw (0,-1) node[below]{\textbf{4-2-2}};
\end{tikzpicture}
\begin{tikzpicture}[scale=0.6]
\draw[dotted] (-2,2) grid (2,-2);
\draw[thick] (1/2,0) -- (0,1) -- (-1,1) -- (-1,0) -- (0,-1) -- cycle;
\draw (1/2,0) node[color=teal]{\(\clubsuit\)};
\draw (0,-1) node[below]{\textbf{4-2-3}};
\end{tikzpicture}
\begin{tikzpicture}[scale=0.6]
\draw[dotted] (-2,2) grid (2,-2);
\draw[thick] (1/2,0) -- (0,1) -- (-1,1) -- (0,-1) -- cycle;
\draw (1/2,0) node[color=teal]{\(\clubsuit\)};
\draw (0,-1) node[below]{\textbf{4-2-4}};
\end{tikzpicture}
\begin{tikzpicture}[scale=0.6]
\draw[dotted] (-2,2) grid (2,-2);
\draw[thick] (1/2,0) -- (0,1) -- (-1,0) -- (0,-1) -- cycle;
\draw (1/2,0) node[color=teal]{\(\clubsuit\)};
\draw (0,-1) node[below]{\textbf{4-2-5}};
\end{tikzpicture}
\begin{tikzpicture}[scale=0.6]
\draw[dotted] (-2,2) grid (2,-2);
\draw[thick] (1/2,0) -- (0,1) -- (-1,-1) -- cycle;
\draw (1/2,0) node[color=teal]{\(\clubsuit\)};
\draw (0,-1) node[below]{\textbf{4-2-6}};
\end{tikzpicture}
\begin{tikzpicture}[scale=0.6]
\draw[dotted] (-2,2) grid (2,-2);
\draw[thick] (1/2,0) -- (-1,1) -- (-1,0) -- (0,-1) -- cycle;
\draw (1/2,0) node[color=teal]{\(\clubsuit\)};
\draw (0,-1) node[below]{\textbf{4-2-7}};
\end{tikzpicture}
\end{center}

\begin{exa}
The embedding \textbf{4-2-5} is the product of embeddings \(Q^3\times \bbP^1\). 
Blowing up a closed orbit yields \textbf{4-2-3}. 
\end{exa}

\textbf{In rank two, for \(H=N(\diag\SL_2)\times \{1\}\)}, 
we are considering equivariant compactification of the group \(\PGL_2\times \bbG_m\). 
We now determine the locally factorial \(G/H\)-reflexive polytopes. 
The point \(\frac{\rho(\clubsuit)}{m_{\clubsuit}}= (1,0)\) must be a vertex for the origin to be an interior point. 
Furthermore, since \(\rho(\clubsuit)\) is not primitive, any facet containing \((1,0\) must be fully contained in the half-space \(\{(y_1,y_2)\mid y_1\geq 0\}\). 
As in previous cases, we observe that a locally factorial \(G/H\)-reflexive polytope is also a smooth Fano polytope, so that we can find all of these by going through the list of smooth Fano polytopes up to \(\GL_2(\bbZ)\)-action. 
We obtain the following three polytopes, and corresponding embeddings.  
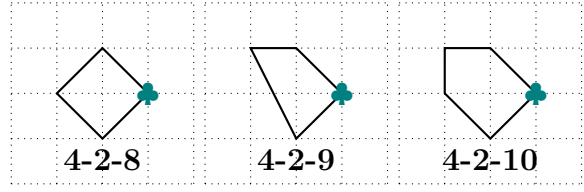
\begin{center}
\begin{tikzpicture}[scale=0.6]
\draw[dotted] (-2,2) grid (2,-2);
\draw[thick] (1,0) -- (0,1) -- (-1,0) -- (0,-1) -- cycle;
\draw (1,0) node[color=teal]{\(\clubsuit\)};
\draw (0,-1) node[below]{\textbf{4-2-8}};
\end{tikzpicture}
\begin{tikzpicture}[scale=0.6]
\draw[dotted] (-2,2) grid (2,-2);
\draw[thick] (1,0) -- (0,1) -- (-1,1) -- (0,-1) -- cycle;
\draw (1,0) node[color=teal]{\(\clubsuit\)};
\draw (0,-1) node[below]{\textbf{4-2-9}};
\end{tikzpicture}
\begin{tikzpicture}[scale=0.6]
\draw[dotted] (-2,2) grid (2,-2);
\draw[thick] (1,0) -- (0,1) -- (-1,1) -- (-1,0) -- (0,-1) -- cycle;
\draw (1,0) node[color=teal]{\(\clubsuit\)};
\draw (0,-1) node[below]{\textbf{4-2-10}};
\end{tikzpicture}
\end{center}

\begin{exa}
The embedding \textbf{4-2-8} is the product of embeddings \(\bbP^3\times \bbP^1\). Blowing up a closed orbit yields the embedding \textbf{4-2-10}. 
\end{exa}

\textbf{In rank two, for \(H=\diag(N(\diag\SL_2))\)}, 
we are considering equivariant compactification of the group \(\GL_2\). 
We now determine the locally factorial \(G/H\)-reflexive polytopes using the same method as in Section~\ref{section_SL_2_horospherical}. 
The point \(\frac{\rho(\clubsuit)}{m_{\clubsuit}}= (1,0)\) must be a vertex for the origin to be an interior point. 
Working up to reflection by Remark~\ref{rem_SL_22_non-uniqueness_of_chi}, we obtain eight locally factorial Fano embeddings, corresponding to the following eight polytopes. 

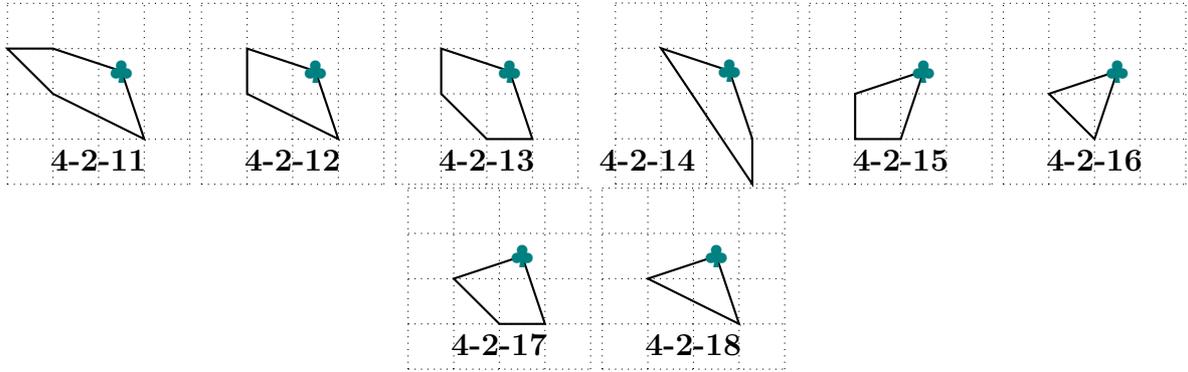
\begin{center}
\begin{tikzpicture}[scale=0.6]
\draw[dotted] (-2,2) grid (2,-2);
\draw[thick] (1/2,1/2) -- (-1,1) -- (-2,1) -- (-1,0) -- (1,-1) -- cycle;
\draw (1/2,1/2) node[color=teal]{\(\clubsuit\)};
\draw (0,-1) node[below]{\textbf{4-2-11}};
\end{tikzpicture}
\begin{tikzpicture}[scale=0.6]
\draw[dotted] (-2,2) grid (2,-2);
\draw[thick] (1/2,1/2) -- (-1,1) -- (-1,0) -- (1,-1) -- cycle;
\draw (1/2,1/2) node[color=teal]{\(\clubsuit\)};
\draw (0,-1) node[below]{\textbf{4-2-12}};
\end{tikzpicture}
\begin{tikzpicture}[scale=0.6]
\draw[dotted] (-2,2) grid (2,-2);
\draw[thick] (1/2,1/2) -- (-1,1) -- (-1,0) -- (0,-1) -- (1,-1) -- cycle;
\draw (1/2,1/2) node[color=teal]{\(\clubsuit\)};
\draw (0,-1) node[below]{\textbf{4-2-13}};
\end{tikzpicture}
\begin{tikzpicture}[scale=0.6]
\draw[dotted] (-2,2) grid (2,-2);
\draw[thick] (1/2,1/2) -- (-1,1) -- (1,-2) -- (1,-1) -- cycle;
\draw (1/2,1/2) node[color=teal]{\(\clubsuit\)};
\draw (0,-1) node[below left]{\textbf{4-2-14}};
\end{tikzpicture}
\begin{tikzpicture}[scale=0.6]
\draw[dotted] (-2,2) grid (2,-2);
\draw[thick] (1/2,1/2) -- (-1,0) -- (-1,-1) -- (0,-1) -- cycle;
\draw (1/2,1/2) node[color=teal]{\(\clubsuit\)};
\draw (0,-1) node[below]{\textbf{4-2-15}};
\end{tikzpicture}
\begin{tikzpicture}[scale=0.6]
\draw[dotted] (-2,2) grid (2,-2);
\draw[thick] (1/2,1/2) -- (-1,0) -- (0,-1) -- cycle;
\draw (1/2,1/2) node[color=teal]{\(\clubsuit\)};
\draw (0,-1) node[below]{\textbf{4-2-16}};
\end{tikzpicture}
\begin{tikzpicture}[scale=0.6]
\draw[dotted] (-2,2) grid (2,-2);
\draw[thick] (1/2,1/2) -- (-1,0) -- (0,-1) -- (1,-1) -- cycle;
\draw (1/2,1/2) node[color=teal]{\(\clubsuit\)};
\draw (0,-1) node[below]{\textbf{4-2-17}};
\end{tikzpicture}
\begin{tikzpicture}[scale=0.6]
\draw[dotted] (-2,2) grid (2,-2);
\draw[thick] (1/2,1/2) -- (-1,0) -- (1,-1) -- cycle;
\draw (1/2,1/2) node[color=teal]{\(\clubsuit\)};
\draw (0,-1) node[below]{\textbf{4-2-18}};
\end{tikzpicture}
\end{center}

\begin{exa}
The embedding \textbf{4-2-16} is the quadric \(Q^4\), as described in details in \cite{infinite}. The action factors through \(\SO_4\times \SO_2\). 
Blowing up a two dimensional subquadric yields the embedding \textbf{4-2-15}, while blowing up one or two points of a zero dimensional subquadric yield \textbf{4-2-17} and \textbf{4-2-13}. 
\end{exa}

\begin{exa}
The embedding \textbf{4-2-18} is \(\bbP^4\). 
It is easily seen as a group compactification: identify one affine chart \(\bbC^4\subset \bbP^4\) with the space of two by two matrices, equipped with the action of \(\GL_2\times \GL_2\) by equivalences. 
This linear action extends to \(\bbP^4\) and provides the equivariant group compactification. 
By blowing up smaller orbits, one can also recover the embeddings \textbf{4-2-12}, \textbf{4-2-17} and \textbf{4-2-13}.
\end{exa}

\subsection{Diagonal Borel case}

\begin{prop}
The combinatorial data for \(\PGL_2^2/\diag B_1\) is 
\[\arraycolsep=12pt
\begin{array}{lll}
M=\langle \alpha_1, \alpha_2 \rangle  &  \rho(\clubsuit)=(1,-1) & \zeta(\clubsuit)=\{\alpha_1\}
\\  
\Sigma=\{\alpha_1,\alpha_2\} & \rho(\varheartsuit)=(1,1) & \zeta(\clubsuit)=\{\alpha_1,\alpha_2\}
\\ 
\mathcal{D} = \{\clubsuit,\varheartsuit,\vardiamondsuit\} &  \rho(\vardiamondsuit)=(-1,1) & \zeta(\vardiamondsuit)=\{\alpha_2\}
\\
\midrule
\kappa = \alpha_1+\alpha_2 & f = 4(1+x_1)(1+x_2)  & m_{\clubsuit} = m_{\varheartsuit} = m_{\vardiamondsuit} = 1
\end{array}
\]
and the combinatorial data for \(\SL_2/\diag B_1\) is 
\[\arraycolsep=12pt
\begin{array}{lll}
M=\langle \varpi_1+\varpi_2, \varpi_1-\varpi_2 \rangle  &  \rho(\clubsuit)=(0,1) & \zeta(\clubsuit)=\{\alpha_1\}
\\  
\Sigma=\{\alpha_1,\alpha_2\} & \rho(\varheartsuit)=(1,0) & \zeta(\clubsuit)=\{\alpha_1,\alpha_2\}
\\ 
\mathcal{D} = \{\clubsuit,\varheartsuit,\vardiamondsuit\} &  \rho(\vardiamondsuit)=(0,-1) & \zeta(\vardiamondsuit)=\{\alpha_2\}
\\
\midrule
\kappa = \alpha_1+\alpha_2 & f = (2+x_1+x_2)(2+x_1-x_2) & m_{\clubsuit} = m_{\varheartsuit} = m_{\vardiamondsuit} = 1
\end{array}
\]
\end{prop}

\begin{proof}
We write the proof in detail in the case of \(\PGL_2\), then deduce the result for \(\SL_2\). For now we thus assume \(G=\PGL_2\). 

\textbf{Geometric realization.}
One geometric realization of the homogeneous space \(G/H\) may be obtained by considering the space \(\bbP^3\times \bbP^1= \bbP(\Mat_{2})\times\bbP^1\) of pairs of a two by two matrix and a line in \(\bbC^2\). 
The action of \(\PGL_2^2\) to consider is then the action by 
\((g_1,g_2)\cdot (M,d) = (g_1Mg_2^{-1}, g_1d)\). 

There are three \(\PGL_2^2\) orbits in \(\bbP^3\times \bbP^1\). 
The obvious open orbit is that of pairs \((M,d)\) where \(M\) is invertible. 
The stabilizer of \((I_2,[1:0])\) is the subgroup \(H\), hence \(\bbP^3\times \bbP^1\) is an equivariant embedding of \(G/H\). 
The complement to the open orbit is the projective space of rank one matrices times the projective line. 
Note that the projective space of rank one matrices may be identified with \(\bbP^1\times \bbP^1\) by sending a rank one matrix to the pair of lines in \(\bbC^2\) formed by its kernel and its image. 
Under this identification, the action of \(\PGL_2^2\) simply splits on the two factors \((g_1,g_2)\cdot (k,i) = (g_2k,g_1i)\). 
It is now easy to check that there are two orbits in \(\bbP^1\times \bbP^1\times \bbP^1\) under the \(\PGL_2^2\) action \((g_1,g_2)\cdot (k,i,d)=(g_2k,g_1i,g_1d)\): the diagonal in the last two entries \(\{(k,i,i)\}\) (of dimension two) and its complement, of dimension three. 

\textbf{Colors.}
Let us now determine the colors of \(G/H\), considered as the space of projectivized two by two invertible matrices times \(\bbP^1\): 
\[ G/H = \PGL_2\times \bbP^1 \subset \bbP^3\times \bbP^1 \]
with the action 
\((g_1,g_2)\cdot (g, d) = (g_1gg_2^{-1}, g_1d) \). 

Let \(B_0\times B_0\) denote the Borel subgroup of \(\PGL_2^2\) formed by pairs of upper triangular matrices. 
The \(B_0\) orbits in \(\bbP^1\) are \(\{[1:0]\}\) and its complement, and the \(B_0\times B_0\) orbits in \(\PGL_2\) are \(B_0\) and its complement.  
As a consequence, the following subsets are stable under the action of \(B_0\times B_0\), and it is not hard to check that they are actually orbits: 
\begin{itemize}
    \item \(B_0\times \{[1:0]\}\) of dimension two
    \item \(D_1:= \PGL_2\setminus B_0 \times \{[1:0]\}\) of dimension 3 
    \item \(D_2:= B_0\times \bbP^1\setminus \{[1:0]\}\) of dimension 3
    \item \(D_3:= \{(g,g([1:0]))\mid g \in \PGL_2\setminus B_0\}\) of dimension 3
    \item \(\{(g,d)\mid g\in \PGL_2\setminus B_0, d\in \bbP^1\setminus \{[1:0],g([1:0])\} \}\) the open orbit. 
\end{itemize}
In particular, we have \(\mathcal{D}=\{D_1,D_2,D_3\}\). 

\textbf{Weight lattice.}
The weight lattice is the full weight lattice of \(B_0\times B_0 \subset \PGL_2\times \PGL_2\), identified with \(\bbZ^2\) via the basis \((\alpha_1,\alpha_2)\).
Indeed, we can find two rational functions whose weights form a basis of this lattice. 
These functions are easily defined in homogeneous coordinates on \(\bbP^3\times \bbP^1\) as follows. 
\[ f_{(-1,0)} : \left(\begin{bmatrix} a & b \\ c & d \end{bmatrix},[x:y]\right)\mapsto \frac{xc-ay}{cy} \] 
and 
\[ f_{(1,1)} : \left(\begin{bmatrix} a & b \\ c & d \end{bmatrix},[x:y]\right)\mapsto \frac{c^2}{cb-ad}. \] 
Let more generally \(f_{(m_1,m_2)}=f_{(-1,0)}^{m_2-m_1}f_{(1,1)}^{m_2}\) for \((m_1,m_2)\in \bbZ^2\). 

Since 
\[ \begin{bmatrix} a_1 & b_1 \\ 0 & 1 \end{bmatrix} \begin{bmatrix} a & b \\ c & d \end{bmatrix} \begin{bmatrix} \frac{1}{a_2} & \frac{-b_2}{a_2} \\ 0 & 1 \end{bmatrix} = \begin{bmatrix} \frac{a_1a}{a_2}+\frac{b_1c}{a_2} & \frac{-a_1ab_2}{a_2}+\frac{-b_1cb_2}{a_2}+a_1b+b_1d \\ \frac{c}{a_2} & \frac{-cb_2}{a_2}+d \end{bmatrix} \]
one can check easily that for \(X\in \PGL_2\times \bbP^1\), 
\[ f_{(m_1,m_2)}\left(\left(\begin{bmatrix} a_1 & b_1 \\ 0 & 1\end{bmatrix},\begin{bmatrix} a_2 & b_2 \\ 0 & 1\end{bmatrix}\right)^{-1}\cdot X\right) = a_1^{m_1}a_2^{m_2} f_{(m_1,m_2)}(X) \]
Hence \(f_{(m_1,m_2)}\) is an eigenfunction for the action of \(B\), with weight \(m_1\alpha_1+m_2\alpha_2\). 

\textbf{Images of colors.}
To determine the color map, we consider the order of vanishing of the functions \(f_{(m_1,m_2)}\) along the \(B\)-stable divisors \(D_i\). 
Again, it is convenient to work in homogeneous coordinates on \(\bbP^3\times \bbP^1\), associating to the class of a matrix \(\begin{bmatrix}a&b\\c&d\end{bmatrix}\) the homogeneous coordinates \([a:b:c:d]\) and to the line \(\bbC\begin{bmatrix}x\\y\end{bmatrix}\) the homogeneous coordinates \([x:y]\). 
One can then write down the equations of the divisors \(D_i\): 
\begin{itemize}
    \item \(D_1=\{y=0\}\)
    \item \(D_2=\{c=0\}\)
    \item \(D_3=\{xc-ay=0\}\)
\end{itemize}
From the expression of \(f_{(m_1,m_2)}\), one deduces that 
\begin{itemize}
    \item \(\ord(D_1)(f_{(m_1,m_2)})=m_1-m_2\)
    \item \(\ord(D_2)(f_{(m_1,m_2)})=m_1+m_2\)
    \item \(\ord(D_3)(f_{(m_1,m_2)})=-m_1+m_2\)
\end{itemize}
Hence the description of \(\rho(D_i)\) in the statement. 

\textbf{Valuation cone.}

We use Proposition~\ref{prop_valuation_cone} to determine the valuation cone. 
Since \(\PGL_2^2/\diag(B_0)\) is the quotient of \(\SL_2^2/\diag(B_0)\) by a central subgroup of order two, the valuation cones of the two homogeneous spaces are the same. 
We compute it for \(\SL_2^2/\diag(B_0)\) for simpler notation. 

Consider the functions 
\[ 
f_1:\SL_2^2 \to \bbC, \qquad \left( \begin{bmatrix}a_1&b_1\\ c_1&d_1\end{bmatrix}, \begin{bmatrix}a_2&b_2\\ c_2&d_2\end{bmatrix} \right) \mapsto c_1c_2 \]
and 
\[ 
f_2:\SL_2^2 \to \bbC, \qquad \left( \begin{bmatrix}a_1&b_1\\ c_1&d_1\end{bmatrix}, \begin{bmatrix}a_2&b_2\\ c_2&d_2\end{bmatrix} \right) \mapsto c_1d_2-d_1c_2 \]
It is easy to check that they are regular functions  that are \(B\)-semi-invariant on the left with weight \(\varpi_1+\varpi_2\) and \(\diag(B_0)\)-semi-invariant on the right. 
Let \(M_1\) and \(M_2\) be the simple \(G\)-modules generated by \(f_1\) and \(f_2\). 
The functions 
\[f:=\diag(\begin{bmatrix} 0& 1\\-1&0\end{bmatrix})\cdot f_1 \times f_2 - \left(I_2,\begin{bmatrix} 0& 1\\-1&0\end{bmatrix}\right)\cdot f_1 \times \left(\begin{bmatrix} 0& 1\\-1&0\end{bmatrix}, I_2\right) \cdot f_2\]
and 
\[\tilde{f}:=\diag(\begin{bmatrix} 0& 1\\-1&0\end{bmatrix})\cdot f_1 \times f_2 - \left(\begin{bmatrix} 0& 1\\-1&0\end{bmatrix}, I_2\right)\cdot f_1 \times \left(I_2, \begin{bmatrix} 0& 1\\-1&0\end{bmatrix}\right) \cdot f_2\]
are respectively in simple \(G\)-submodules in \(M_1M_2\) of highest weight \(2\varpi_2\) and \(2\varpi_1\). 

Since \(2\varpi_1+2\varpi_2 - 2\varpi_2= \alpha_1\) and \(2\varpi_1+2\varpi_2 - 2\varpi_1= \alpha_2\), by Proposition~\ref{prop_valuation_cone}, we deduce that the valuation cone \(\mathcal{V}\) is included in \(\{x\in N\otimes \bbQ\mid \langle \alpha_1,x\rangle\leq 0, \langle \alpha_2,x\rangle\leq 0\}\). 
This is in fact an equality, by Proposition~\ref{prop_spherical_root_as_sum_of_positive_roots}, and \(\Sigma=\{\alpha_1,\alpha_2\}\). 

\textbf{The case \(G_0=\SL_2\) follows directly.} 
Since \(\diag(-I_2)\in \diag(B_0)\), we know that \(M\subset (\varpi_1+\varpi_2)\bbZ\oplus (\varpi_1-\varpi_2)\bbZ\). 
But since both homogeneous spaces are related by the quotient by \((I_2,-I_2)\), we have \(\alpha_1\bbZ\oplus \alpha_2\bbZ\subset M\).
This shows that \(\Sigma=\{\alpha_1,\alpha_2\}\). 
Using the properties of colors by type (see e.g. \cite{Gagliardi_Hofscheier_2015_homogeneous}), we see that necessarily all the other combinatorial data are the same. 
To distinguish the two homogeneous spaces, the weight lattices must be different, hence the result. 
\end{proof}

We now determine the possible locally factorial \(G/H\)-reflexive polytopes. We first note that, in view of the combinatorial data, we already know three vertices given by the \(\rho(D)/m_{D}=\rho(D)\) for \(D\in \mathcal{D}\). 

\textbf{Case \(H=\diag(B)\).}
By considering the successive next vertices similarly as in Section~\ref{section_SL_2_horospherical}, we quickly obtain the classification of locally factorial \(G/H\)-reflexive polytopes: there are only three, as follows. 

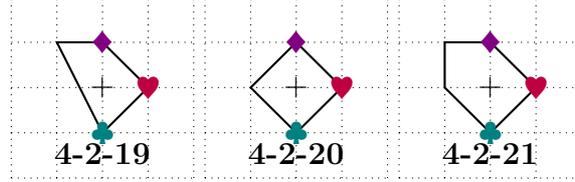
\begin{center}
\begin{tikzpicture}[scale=0.6]
\draw[dotted] (-2,2) grid (2,-2);
\draw (0,0) node{+};
\draw[thick] (0,-1) -- (1,0) -- (0,1) -- (-1,1) -- cycle;
\draw (1,0) node[color=purple]{\(\varheartsuit\)};
\draw (0,-1) node[color=teal]{\(\clubsuit\)};
\draw (0,1) node[color=violet]{\(\vardiamondsuit\)};
\draw (0,-1) node[below]{\textbf{4-2-19}};
\end{tikzpicture}
\begin{tikzpicture}[scale=0.6]
\draw[dotted] (-2,2) grid (2,-2);
\draw (0,0) node{+};
\draw[thick] (0,-1) -- (1,0) -- (0,1) -- (-1,0) -- cycle;
\draw (1,0) node[color=purple]{\(\varheartsuit\)};
\draw (0,-1) node[color=teal]{\(\clubsuit\)};
\draw (0,1) node[color=violet]{\(\vardiamondsuit\)};
\draw (0,-1) node[below]{\textbf{4-2-20}};
\end{tikzpicture}
\begin{tikzpicture}[scale=0.6]
\draw[dotted] (-2,2) grid (2,-2);
\draw (0,0) node{+};
\draw[thick] (0,-1) -- (1,0) -- (0,1) -- (-1,1) -- (-1,0) -- cycle;
\draw (1,0) node[color=purple]{\(\varheartsuit\)};
\draw (0,-1) node[color=teal]{\(\clubsuit\)};
\draw (0,1) node[color=violet]{\(\vardiamondsuit\)};
\draw (0,-1) node[below]{\textbf{4-2-21}};
\end{tikzpicture}
\end{center}

\textbf{Case \(H=N(\diag(B))\).}
We obtain the following two additional possibilities. 

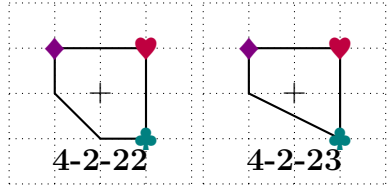
\begin{center}
\begin{tikzpicture}[scale=0.6]
\draw[dotted] (-2,2) grid (2,-2);
\draw (0,0) node{+};
\draw[thick] (1,-1) -- (1,1) -- (-1,1) -- (-1,0) -- (0,-1) -- cycle;
\draw (1,-1) node[color=teal]{\(\clubsuit\)};
\draw (1,1) node[color=purple]{\(\varheartsuit\)};
\draw (-1,1) node[color=violet]{\(\vardiamondsuit\)};
\draw (0,-1) node[below]{\textbf{4-2-22}};
\end{tikzpicture}
\begin{tikzpicture}[scale=0.6]
\draw[dotted] (-2,2) grid (2,-2);
\draw (0,0) node{+};
\draw[thick] (1,-1) -- (1,1) -- (-1,1) -- (-1,0) -- cycle;
\draw (1,-1) node[color=teal]{\(\clubsuit\)};
\draw (1,1) node[color=purple]{\(\varheartsuit\)};
\draw (-1,1) node[color=violet]{\(\vardiamondsuit\)};
\draw (0,-1) node[below]{\textbf{4-2-23}};
\end{tikzpicture}
\end{center}

\begin{exa}
As follows from the proof for the combinatorial data, or its obvious variant, the embeddings \textbf{4-2-23} and \textbf{4-2-19} are \(\bbP^3\times \bbP^1\) and \(Q^3\times \bbP^1\). 
Blowing up their unique closed orbits yields the embeddings \textbf{4-2-22} and \textbf{4-2-21}. 
\end{exa}

\subsection{Symmetric varieties of type \(A_1\times A_1\)}

We consider the cases when \(n=0\) and the Lie algebra of \(H\) is \(\mathfrak{t}_1\oplus\mathfrak{t}_2\). 
There are four possible corresponding subgroups (satisfying our assumptions). 
All but one are products, so that we can recover their combinatorial data directly. 

\begin{prop}
The combinatorial data for \(H=T_1\times T_2\) is as follows: 
\[\arraycolsep=12pt
\begin{array}{lll}
M=\langle \alpha_1, \alpha_2 \rangle  &  \rho(\clubsuit)=\rho(\varheartsuit)=(1,0) & \rho(\spadesuit)=\rho(\vardiamondsuit)=(0,1)
\\  
\Sigma=\{\alpha_1,\alpha_2\} & \zeta(\clubsuit)=\zeta(\varheartsuit)=\{\alpha_1\} &  \zeta(\spadesuit)=\zeta(\vardiamondsuit)=\{\alpha_2\}
\\ 
\mathcal{D} = \{\clubsuit,\varheartsuit,\spadesuit,\vardiamondsuit\} &  
\\
\midrule
\kappa = \alpha_1+\alpha_2 & f = 4(1+x_1)(1+x_2) & m_{\clubsuit} = m_{\varheartsuit} = m_{\spadesuit} = m_{\vardiamondsuit} = 1
\end{array}
\]
\end{prop}

\begin{prop}
The combinatorial data for \(H=N(T_1)\times T_2\) is as follows: 
\[\arraycolsep=12pt
\begin{array}{lll}
M=\langle 2\alpha_1, \alpha_2 \rangle  &  \rho(\clubsuit)=(2,0) & \rho(\spadesuit)=\rho(\vardiamondsuit)=(0,1)
\\  
\Sigma=\{2\alpha_1,\alpha_2\} & \zeta(\clubsuit)=\{\alpha_1\} &  \zeta(\spadesuit)=\zeta(\vardiamondsuit)=\{\alpha_2\}
\\ 
\mathcal{D} = \{\clubsuit,\spadesuit,\vardiamondsuit\} &  
\\
\midrule
\kappa = \alpha_1+\alpha_2 & f = 4(1+2x_1)(1+x_2) & m_{\clubsuit} = m_{\spadesuit} = m_{\vardiamondsuit} = 1
\end{array}
\]
\end{prop}

\begin{prop}
The combinatorial data for \(H=N(T_1)\times N(T_2)\) is as follows: 
\[\arraycolsep=12pt
\begin{array}{lll}
M=\langle 2\alpha_1, 2\alpha_2 \rangle  &  \rho(\clubsuit)=(2,0) & \rho(\spadesuit)=(0,2)
\\  
\Sigma=\{2\alpha_1,2\alpha_2\} & \zeta(\clubsuit)=\{\alpha_1\} &  \zeta(\spadesuit)=\{\alpha_2\}
\\ 
\mathcal{D} = \{\clubsuit,\spadesuit\} &  
\\
\midrule
\kappa = \alpha_1+\alpha_2 & f = 4(1+2x_1)(1+2x_2) & m_{\clubsuit} = m_{\spadesuit} = 1
\end{array}
\]
\end{prop}

Finally, for the remaining case \(H=\diag(N(T_1))\), we can deduce its combinatorial data from that the homogeneous space is obtained as the quotient of \(\SL_2^2/T_1\times T_2\) by the group generated by \(\left(\begin{bmatrix}0&1\\-1&0\end{bmatrix},\begin{bmatrix}0&1\\-1&0\end{bmatrix}\right)\). 

\begin{prop}
The combinatorial data for \(H=\diag N(T_1)\) is as follows: 
\[\arraycolsep=12pt
\begin{array}{lll}
M=\langle \alpha_1+\alpha_2, \alpha_1-\alpha_2 \rangle  &  \rho(\clubsuit)=(1,1) & \rho(\spadesuit)=(1,-1)
\\  
\Sigma=\{2\alpha_1,2\alpha_2\} & \zeta(\clubsuit)=\{\alpha_1\} &  \zeta(\spadesuit)=\{\alpha_2\}
\\ 
\mathcal{D} = \{\clubsuit,\spadesuit\} &  
\\
\midrule
\kappa = \alpha_1+\alpha_2 & f = 4(1+x_1+x_2)(1+x_1-x_2) & m_{\clubsuit} = m_{\spadesuit}  = 1
\end{array}
\]
\end{prop}

\subsubsection{Polytopes}

The conditions on locally factorial \(G/H\)-reflexive polytopes for these four homogeneous spaces are very strong, yielding few embeddings. 
Note that there can be facets containing a \(\frac{\rho(D)}{m_D}\) only in the case \(H=\diag(N(T_1))\), otherwise either \(\frac{\rho(D)}{m_D}\) is not primitive, or the color map has several antecedents of \(\rho(D)\). 
By considering successively the possible next vertices, we arrive quickly to the following possibilities. 

\textbf{ For \(H=T_1\times T_2\)}, we obtain two polytopes, and two corresponding locally factorial Fano embeddings. 

\begin{center}
\begin{tikzpicture}[scale=0.6]
\draw[dotted] (-2,2) grid (2,-2);
\draw[thick] (1,0) -- (0,1) -- (-1,0) -- (0,-1) -- cycle;
\draw (1,0) node[below=-4, color=teal]{\(\clubsuit\)};
\draw (1,0) node[right=-4, color=purple]{\(\varheartsuit\)};
\draw (0,1) node[above=-4, color=violet]{\(\vardiamondsuit\)};
\draw (0,1) node[left=-4]{\(\spadesuit\)};
\draw (0,-1) node[below]{\textbf{4-2-24}};
\end{tikzpicture}
\begin{tikzpicture}[scale=0.6]
\draw[dotted] (-2,2) grid (2,-2);
\draw[thick] (1,0) -- (0,1) -- (-1,0) -- (-1,-1) -- (0,-1) -- cycle;
\draw (1,0) node[below=-4, color=teal]{\(\clubsuit\)};
\draw (1,0) node[right=-4, color=purple]{\(\varheartsuit\)};
\draw (0,1) node[above=-4, color=violet]{\(\vardiamondsuit\)};
\draw (0,1) node[left=-4]{\(\spadesuit\)};
\draw (0,-1) node[below]{\textbf{4-2-25}};
\end{tikzpicture}
\end{center}

\textbf{ For \(H=N(T_1)\times T_2\)}, we obtain two polytopes, and two corresponding locally factorial Fano embeddings. 

\begin{center}
\begin{tikzpicture}[scale=0.6]
\draw[dotted] (-2,2) grid (2,-2);
\draw[thick] (2,0) -- (0,1) -- (-1,0) -- (0,-1) -- cycle;
\draw (2,0) node[color=teal]{\(\clubsuit\)};
\draw (0,1) node[above=-4, color=violet]{\(\vardiamondsuit\)};
\draw (0,1) node[left=-4]{\(\spadesuit\)};
\draw (0,-1) node[below]{\textbf{4-2-26}};
\end{tikzpicture}
\begin{tikzpicture}[scale=0.6]
\draw[dotted] (-2,2) grid (2,-2);
\draw[thick] (2,0) -- (0,1) -- (-1,0) -- (-1,-1) -- (0,-1) -- cycle;
\draw (2,0) node[color=teal]{\(\clubsuit\)};
\draw (0,1) node[above=-4, color=violet]{\(\vardiamondsuit\)};
\draw (0,1) node[left=-4]{\(\spadesuit\)};
\draw (0,-1) node[below]{\textbf{4-2-27}};
\end{tikzpicture}
\end{center}

\textbf{ For \(H=N(T_1)\times N(T_2)\)}, we obtain two polytopes, and two corresponding locally factorial Fano embeddings. 

\begin{center}
\begin{tikzpicture}[scale=0.6]
\draw[dotted] (-2,2) grid (2,-2);
\draw[thick] (2,0) -- (0,2) -- (-1,0) -- (0,-1) -- cycle;
\draw (2,0) node[color=teal]{\(\clubsuit\)};
\draw (0,2) node{\(\spadesuit\)};
\draw (0,-1) node[below]{\textbf{4-2-28}};
\end{tikzpicture}
\begin{tikzpicture}[scale=0.6]
\draw[dotted] (-2,2) grid (2,-2);
\draw[thick] (2,0) -- (0,2) -- (-1,0) -- (-1,-1) -- (0,-1) -- cycle;
\draw (2,0) node[color=teal]{\(\clubsuit\)};
\draw (0,2) node{\(\spadesuit\)};
\draw (0,-1) node[below]{\textbf{4-2-29}};
\end{tikzpicture}
\end{center}

\begin{exa}
In the three previous choices of homogeneous spaces, the embeddings obtained are the obvious product embeddings, and their blowups along the unique closed orbit. 
\end{exa}

\textbf{ For \(H=\diag(N(T_1))\)}, we obtain two polytopes, and two corresponding locally factorial Fano embeddings. 

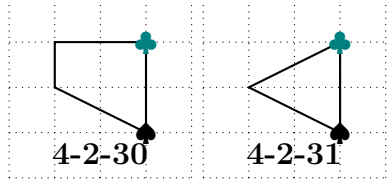
\begin{center}
\begin{tikzpicture}[scale=0.6]
\draw[dotted] (-2,2) grid (2,-2);
\draw[thick] (1,-1) -- (1,1) -- (-1,1) -- (-1,0) -- cycle;
\draw (1,1) node[color=teal]{\(\clubsuit\)};
\draw (1,-1) node{\(\spadesuit\)};
\draw (0,-1) node[below]{\textbf{4-2-30}};
\end{tikzpicture}
\begin{tikzpicture}[scale=0.6]
\draw[dotted] (-2,2) grid (2,-2);
\draw[thick] (1,-1) -- (1,1) -- (-1,0) -- cycle;
\draw (1,1) node[color=teal]{\(\clubsuit\)};
\draw (1,-1) node{\(\spadesuit\)};
\draw (0,-1) node[below]{\textbf{4-2-31}};
\end{tikzpicture}
\end{center}

\begin{exa}
The embedding \textbf{4-2-31} is the quadric \(Q^4\), and the action factors through \(\SO_3\times \SO_3\subset \SO_6\). 
This embedding is described in details in \cite{infinite}. 
Blowing up a one-dimensional subquadric yields the other embedding \textbf{4-2-30}.
\end{exa}

\subsection{Parabolic inductions of type T}

Let \(Q=Q_{\alpha_2}\), \(G_0=\SL_2\times \bbG_m\) and let \(\pi:Q\to G_0\) be the epimorphism given by \((\id_{\SL_2},a_2\varpi_2+\chi_1)\) . 
Letting, as in the paragraph on parabolic induction, \(B_0=\pi(B\cap Q)\), we have \(X^*(B_0)=\bbZ\varpi_1 \oplus \bbZ(a_2\varpi_2+\chi_1)\). 

\begin{prop}
Assume that \(G/H\) is obtained by parabolic induction from \(G_0/H_0\) with respect to \(\pi:Q\to G_0\) as above, with \(H_0=\ker(a_1\varpi_1+a_2\varpi_2+\chi)|_{T_0}\). 
If \(a_1\) is even, the combinatorial data of \(G/H\) is as follows:
\[\arraycolsep=12pt
\begin{array}{lll}
M=\langle \alpha_1, a_2\varpi_2+\chi_1 \rangle  &  \rho(\clubsuit)=(1,a_1/2) & \zeta(\clubsuit)=\{\alpha_1\}
\\  
\Sigma=\{\alpha_1\} & \rho(\varheartsuit)=(1,-a_1/2) & \zeta(\varheartsuit)=\{\alpha_1\}
\\ 
\mathcal{D} = \{\clubsuit,\varheartsuit,\vardiamondsuit\} &  \rho(\vardiamondsuit)=(0,a_2) & \zeta(\vardiamondsuit)=\{\alpha_2\}
\\
\midrule
\kappa = \alpha_1+\alpha_2 & f = 2(1+x_1)(2+a_2x_2) & m_{\clubsuit} = m_{\varheartsuit} = 1,  m_{\vardiamondsuit} = 2
\end{array}
\]
If \(a_1\) is odd, the combinatorial data of \(G/H\) is as follows:
\[\arraycolsep=12pt
\begin{array}{lll}
M=\langle \varpi_1+a_2\varpi_2+\chi_1, \varpi_1-a_2\varpi_2-\chi_1 \rangle  &  \rho(\clubsuit)=(\frac{a_1+1}{2},\frac{1-a_1}{2}) & \zeta(\clubsuit)=\{\alpha_1\}
\\  
\Sigma=\{\alpha_1\} & \rho(\varheartsuit)=(\frac{1-a_1}{2},\frac{a_1+1}{2}) & \zeta(\varheartsuit)=\{\alpha_1\}
\\ 
\mathcal{D} = \{\clubsuit,\varheartsuit,\vardiamondsuit\} &  \rho(\vardiamondsuit)=a_2(1,-1) & \zeta(\vardiamondsuit)=\{\alpha_2\}
\\
\midrule
f = (2+x_1+x_2)(2+a_2(x_1-x_2)) & \kappa = \alpha_1+\alpha_2 & m_{\clubsuit} = m_{\varheartsuit} = 1,  m_{\vardiamondsuit} = 2
\end{array}
\]
\end{prop}

\begin{proof}
It follows straightforwardly from the Proposition~\ref{prop_parabolic_induction}, and the combinatorial data for subgroups of \(\SL_2\times \bbG_m^n\). 
\end{proof}

According to our classifications, we may assume that \(a_1\) is a non-negative integer. 
We now determine the possible locally factorial \(G/H\)-reflexive polytopes. 
Let \(\Omega\) be such a polytope. 
Note first that the two points \(\rho(\clubsuit)\) and \(\rho(\varheartsuit)\) are vertices of \(\Omega\). 
If \(\rho(\vardiamondsuit)/2\) is a vertex, then \(\rho(\vardiamondsuit)\) must be primitive, hence \(a_2=1\). 
If it is not a vertex, then we must have \(a_2\leq 1\). 

\textbf{Step~0: \(a_2=0\)} In this case, any embedding is a product of \(\bbP^1\) with a type T, dimension three \(\SL_2\times \bbG_m\)-spherical variety, with the product action. 
There are 14 such products, that we identify as \textbf{4-2-32} to \textbf{4-2-45}. 

From now on we assume that \(a_2=1\). 
As will be clear from the proof of the first few cases, there are no possible \(\Omega\) if \(a_1\geq 3\), we thus deal successively with the cases \(a_1=0\), \(1\) and \(2\). 

\textbf{Step~1: \(a_1=0\), \(\rho(\vardiamondsuit)/2\) is not a vertex}
Since \(\rho(\clubsuit)=\rho(\varheartsuit)\), these vertices cannot be vertices of a facet intersecting the relative interior of the valuation cone. 
As a consequence, the primitive generators \((0,1)\) and \((0,-1)\) of the valuation cone are vertices. 
The vertex following \((0,1)\) in counter-clockwise sense must have first coordinate \(-1\) (to form a basis with \((0,1)\)), and must have second coordinate less than \(1\) (so that \((0,1)\) remains a vertex of \(\Delta\)). If it is \((-1,1)\), then the next vertex in counterclockwise sense can only be \((-1,0)\) or \((0,-1)\). 
If it is \((-1,0)\), then the next vertex in counterclockwise sense is either \((-1,-1)\) or \((0,-1)\). 
It is now easy to exhaust the rest of the possibilities. 
Finally, there are five polytopes

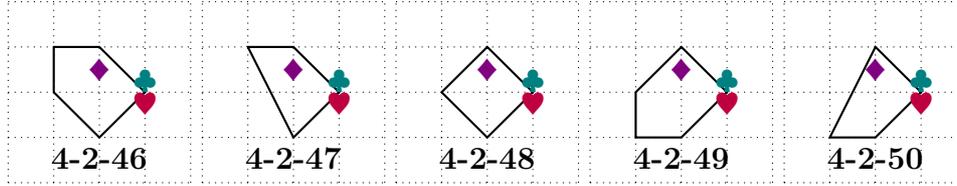
\begin{center}
\begin{tikzpicture}[scale=0.6]
\draw[dotted] (-2,2) grid (2,-2);
\draw[thick] (1,0) -- (0,1) -- (-1,1) -- (-1,0) -- (0,-1) -- cycle;
\draw (1,0) node[above=-4, color=teal]{\(\clubsuit\)};
\draw (1,0) node[below=-4, color=purple]{\(\varheartsuit\)};
\draw (0,1/2) node[color=violet]{\(\vardiamondsuit\)};
\draw (0,-1) node[below]{\textbf{4-2-46}};
\end{tikzpicture}
\begin{tikzpicture}[scale=0.6]
\draw[dotted] (-2,2) grid (2,-2);
\draw[thick] (1,0) -- (0,1) -- (-1,1) -- (0,-1) -- cycle;
\draw (1,0) node[above=-4, color=teal]{\(\clubsuit\)};
\draw (1,0) node[below=-4, color=purple]{\(\varheartsuit\)};
\draw (0,1/2) node[color=violet]{\(\vardiamondsuit\)};
\draw (0,-1) node[below]{\textbf{4-2-47}};
\end{tikzpicture}
\begin{tikzpicture}[scale=0.6]
\draw[dotted] (-2,2) grid (2,-2);
\draw[thick] (1,0) -- (0,1) -- (-1,0) -- (0,-1) -- cycle;
\draw (1,0) node[above=-4, color=teal]{\(\clubsuit\)};
\draw (1,0) node[below=-4, color=purple]{\(\varheartsuit\)};
\draw (0,1/2) node[color=violet]{\(\vardiamondsuit\)};
\draw (0,-1) node[below]{\textbf{4-2-48}};
\end{tikzpicture}
\begin{tikzpicture}[scale=0.6]
\draw[dotted] (-2,2) grid (2,-2);
\draw[thick] (1,0) -- (0,1) -- (-1,0) -- (-1,-1) -- (0,-1) -- cycle;
\draw (1,0) node[above=-4, color=teal]{\(\clubsuit\)};
\draw (1,0) node[below=-4, color=purple]{\(\varheartsuit\)};
\draw (0,1/2) node[color=violet]{\(\vardiamondsuit\)};
\draw (0,-1) node[below]{\textbf{4-2-49}};
\end{tikzpicture}
\begin{tikzpicture}[scale=0.6]
\draw[dotted] (-2,2) grid (2,-2);
\draw[thick] (1,0) -- (0,1) -- (-1,-1) -- (0,-1) -- cycle;
\draw (1,0) node[above=-4, color=teal]{\(\clubsuit\)};
\draw (1,0) node[below=-4, color=purple]{\(\varheartsuit\)};
\draw (0,1/2) node[color=violet]{\(\vardiamondsuit\)};
\draw (0,-1) node[below]{\textbf{4-2-50}};
\end{tikzpicture}
\end{center}

\begin{exa}
The embedding \textbf{4-2-48} is the product of embeddings \((\bbP^1)^2\times \bbF_1\). 
Blowing up one of the two closed orbits yields either \textbf{4-2-46} or \textbf{4-2-49}. 
\end{exa}

\textbf{Step~2: \(a_1=0\) and \((0,\frac{1}{2})\) is a vertex}
The same method yields the following polytopes (with respect to the previous case, some vertices are not allowed as they would lie on the same facet as the previous facet in counterclockwise order). 
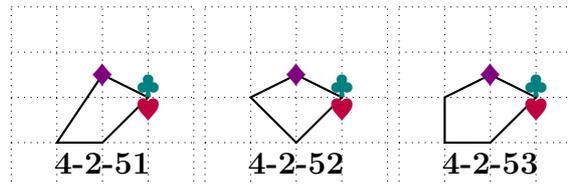
\begin{center}
\begin{tikzpicture}[scale=0.6]
\draw[dotted] (-2,2) grid (2,-2);
\draw[thick] (1,0) -- (0,1/2) -- (-1,-1) -- (0,-1) -- cycle;
\draw (1,0) node[above=-4, color=teal]{\(\clubsuit\)};
\draw (1,0) node[below=-4, color=purple]{\(\varheartsuit\)};
\draw (0,1/2) node[color=violet]{\(\vardiamondsuit\)};
\draw (0,-1) node[below]{\textbf{4-2-51}};
\end{tikzpicture}
\begin{tikzpicture}[scale=0.6]
\draw[dotted] (-2,2) grid (2,-2);
\draw[thick] (1,0) -- (0,1/2) -- (-1,0) -- (0,-1) -- cycle;
\draw (1,0) node[above=-4, color=teal]{\(\clubsuit\)};
\draw (1,0) node[below=-4, color=purple]{\(\varheartsuit\)};
\draw (0,1/2) node[color=violet]{\(\vardiamondsuit\)};
\draw (0,-1) node[below]{\textbf{4-2-52}};
\end{tikzpicture}
\begin{tikzpicture}[scale=0.6]
\draw[dotted] (-2,2) grid (2,-2);
\draw[thick] (1,0) -- (0,1/2) -- (-1,0) -- (-1,-1) -- (0,-1) -- cycle;
\draw (1,0) node[above=-4, color=teal]{\(\clubsuit\)};
\draw (1,0) node[below=-4, color=purple]{\(\varheartsuit\)};
\draw (0,1/2) node[color=violet]{\(\vardiamondsuit\)};
\draw (0,-1) node[below]{\textbf{4-2-53}};
\end{tikzpicture}
\end{center}

\begin{exa}
The embedding \textbf{4-2-52} is the product of embeddings \((\bbP^1)^2\times \bbP^2\). 
Blowing up \((\diag \bbP^1)\times H\) where \(H\) is the stable line in \(\bbP^2\) yields the embedding \textbf{4-2-53}. 
\end{exa}

\textbf{Step~3: \(a_1=1\) and \((1/2,-1/2)\) is not a vertex}
We obtain the following list of five polytopes.
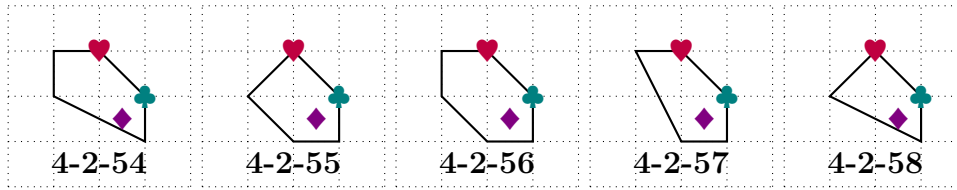
\begin{center}
\begin{tikzpicture}[scale=0.6]
\draw[dotted] (-2,2) grid (2,-2);
\draw[thick] (1,0) -- (0,1) -- (-1,1) -- (-1,0) -- (1,-1) -- cycle;
\draw (1,0) node[color=teal]{\(\clubsuit\)};
\draw (0,1) node[color=purple]{\(\varheartsuit\)};
\draw (1/2,-1/2) node[color=violet]{\(\vardiamondsuit\)};
\draw (0,-1) node[below]{\textbf{4-2-54}};
\end{tikzpicture}
\begin{tikzpicture}[scale=0.6]
\draw[dotted] (-2,2) grid (2,-2);
\draw[thick] (1,0) -- (0,1) -- (-1,0) -- (0,-1) -- (1,-1) -- cycle;
\draw (1,0) node[color=teal]{\(\clubsuit\)};
\draw (0,1) node[color=purple]{\(\varheartsuit\)};
\draw (1/2,-1/2) node[color=violet]{\(\vardiamondsuit\)};
\draw (0,-1) node[below]{\textbf{4-2-55}};
\end{tikzpicture}
\begin{tikzpicture}[scale=0.6]
\draw[dotted] (-2,2) grid (2,-2);
\draw[thick] (1,0) -- (0,1) -- (-1,1) -- (-1,0) -- (0,-1) -- (1,-1) -- cycle;
\draw (1,0) node[color=teal]{\(\clubsuit\)};
\draw (0,1) node[color=purple]{\(\varheartsuit\)};
\draw (1/2,-1/2) node[color=violet]{\(\vardiamondsuit\)};
\draw (0,-1) node[below]{\textbf{4-2-56}};
\end{tikzpicture}
\begin{tikzpicture}[scale=0.6]
\draw[dotted] (-2,2) grid (2,-2);
\draw[thick] (1,0) -- (0,1) -- (-1,1) -- (0,-1) -- (1,-1) -- cycle;
\draw (1,0) node[color=teal]{\(\clubsuit\)};
\draw (0,1) node[color=purple]{\(\varheartsuit\)};
\draw (1/2,-1/2) node[color=violet]{\(\vardiamondsuit\)};
\draw (0,-1) node[below]{\textbf{4-2-57}};
\end{tikzpicture}
\begin{tikzpicture}[scale=0.6]
\draw[dotted] (-2,2) grid (2,-2);
\draw[thick] (1,0) -- (0,1) -- (-1,0) -- (1,-1) -- cycle;
\draw (1,0) node[color=teal]{\(\clubsuit\)};
\draw (0,1) node[color=purple]{\(\varheartsuit\)};
\draw (1/2,-1/2) node[color=violet]{\(\vardiamondsuit\)};
\draw (0,-1) node[below]{\textbf{4-2-58}};
\end{tikzpicture}
\end{center}

\textbf{Step~4: \(a_1=1\) and \((1/2,-1/2)\) is a vertex}
Going the arguments above, we obtain only two possibilities. 

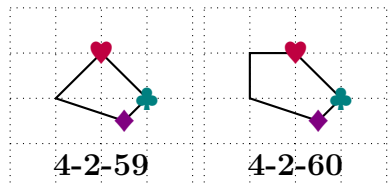
\begin{center}
\begin{tikzpicture}[scale=0.6]
\draw[dotted] (-2,2) grid (2,-2);
\draw[thick] (1,0) -- (0,1) -- (-1,0) -- (1/2,-1/2) -- cycle;
\draw (1,0) node[color=teal]{\(\clubsuit\)};
\draw (0,1) node[color=purple]{\(\varheartsuit\)};
\draw (1/2,-1/2) node[color=violet]{\(\vardiamondsuit\)};
\draw (0,-1) node[below]{\textbf{4-2-59}};
\end{tikzpicture}
\begin{tikzpicture}[scale=0.6]
\draw[dotted] (-2,2) grid (2,-2);
\draw[thick] (1,0) -- (0,1) -- (-1,1) -- (-1,0) -- (1/2,-1/2) -- cycle;
\draw (1,0) node[color=teal]{\(\clubsuit\)};
\draw (0,1) node[color=purple]{\(\varheartsuit\)};
\draw (1/2,-1/2) node[color=violet]{\(\vardiamondsuit\)};
\draw (0,-1) node[below]{\textbf{4-2-60}};
\end{tikzpicture}
\end{center}

\begin{exa}
The embedding \textbf{4-2-59} is \(\bbP^3\times \bbP^1\). 
The action of \(\SL_2^2\) to consider is given by considering \(\bbP^3\times \bbP^1\) as \(\bbP(\bbC^2_1\oplus \bbC^2_2)\times \bbP(\bbC^2_1)\) where \(\bbC^2_1\) and \(\bbC^2_2\) denote the standard representations of each of the \(\SL_2\) factor. 
By successive well-chosen blowups, one can obtain a precise description for all embeddings with \(a_1=1\) except \textbf{4-2-57}. 
\end{exa}

\textbf{Step~5: \(a_1=2\)}
Following the same steps as before, we easily see that \((0,1/2)\) cannot be a vertex, and obtain only two polytopes. 

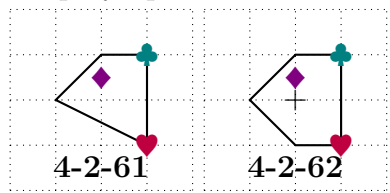
\begin{center}
\begin{tikzpicture}[scale=0.6]
\draw[dotted] (-2,2) grid (2,-2);
\draw[thick] (1,1) -- (0,1) -- (-1,0) --  (1,-1) -- cycle;
\draw (1,1) node[color=teal]{\(\clubsuit\)};
\draw (1,-1) node[color=purple]{\(\varheartsuit\)};
\draw (0,1/2) node[color=violet]{\(\vardiamondsuit\)};
\draw (0,-1) node[below]{\textbf{4-2-61}};
\end{tikzpicture}
\begin{tikzpicture}[scale=0.6]
\draw[dotted] (-2,2) grid (2,-2);
\draw (0,0) node{+}; 
\draw[thick] (1,1) -- (0,1) -- (-1,0) -- (0,-1) -- (1,-1) -- cycle;
\draw (1,1) node[color=teal]{\(\clubsuit\)};
\draw (1,-1) node[color=purple]{\(\varheartsuit\)};
\draw (0,1/2) node[color=violet]{\(\vardiamondsuit\)};
\draw (0,-1) node[below]{\textbf{4-2-62}};
\end{tikzpicture}
\end{center}

\textbf{Step~6: \(a_1\geq 3\)}
From the same method as before, we obtain that there are no locally factorial \(G/H\)-reflexive polytopes. 

\subsection{Parabolic inductions of type N}

Let \(Q=Q_{\alpha_2}\), \(G_0=\SL_2\times \bbG_m\) and let \(\pi:Q\to G_0\) be the epimorphism given by \((\id_{\SL_2},a_2\varpi_2+\chi_1)\) . 
Letting, as in the paragraph on parabolic induction, \(B_0=\pi(B\cap Q)\), we have \(X^*(B_0)=\bbZ\varpi_1 \oplus \bbZ(a_2\varpi_2+\chi_1)\). 

\begin{prop}
Assume that \(G/H\) is obtained by parabolic induction from \(G_0/H_0\) with respect to \(\pi:Q\to G_0\) as above.
If \(H_0=N(T_1)\times \{1\}\), then the combinatorial data is:
\[\arraycolsep=12pt
\begin{array}{lll}
M=\langle 2\alpha_1, a_2\varpi_2+\chi_1 \rangle  &  \rho(\clubsuit)=(2,0) & \zeta(\clubsuit)=\{\alpha_1\}
\\  
\Sigma=\{2\alpha_1\} & \rho(\vardiamondsuit)=(0,a_2) & \zeta(\vardiamondsuit)=\{\alpha_2\}
\\ 
\mathcal{D} = \{\clubsuit,\vardiamondsuit\} 
\\
\midrule
\kappa = \alpha_1+\alpha_2 & f = 2(1+2x_1)(2+a_2x_2) & m_{\clubsuit} = 1,  m_{\vardiamondsuit} = 2
\end{array}
\]
If \(H_0=\diag(N(T_1))\), then the combinatorial data is:
\[\arraycolsep=12pt
\begin{array}{lll}
M=\langle \alpha_1+a_2\varpi_2+\chi_1, \alpha_1-a_2\varpi_2-\chi_1 \rangle  &  \rho(\clubsuit)=(1,1) & \zeta(\clubsuit)=\{\alpha_1\}
\\  
\Sigma=\{2\alpha_1\} & \rho(\vardiamondsuit)=a_2(1,-1) & \zeta(\vardiamondsuit)=\{\alpha_2\}
\\ 
\mathcal{D} = \{\clubsuit,\vardiamondsuit\} 
\\
\midrule
f = 2(1+x_1+x_2)(2+a_2(x_1-x_2)) & \kappa = \alpha_1+\alpha_2 &  m_{\clubsuit} = 1,  m_{\vardiamondsuit} = 2
\end{array}
\]
\end{prop}

Let us now classify the possible locally factorial \(G/H\)-reflexive polytopes. 
As previously, we observe that \(a_2\in \{0,1\}\), and if \(a_2=0\), then any embedding is a product of \(\bbP^1\) with a type N, dimension three spherical variety under \(\SL_2\times \bbG_m\), with the product action. 
There are nine such products, that we number as \textbf{4-2-63} to \textbf{4-2-71}.
We now deal with the cases where \(a_2=1\).

\textbf{We first assume that \(H_0=\diag(N(T_1))\)} 

\textbf{If \((1/2,-1/2)\) is not a vertex}
By applying the same exhaustion argument as before, we obtain the following list of five polytopes.

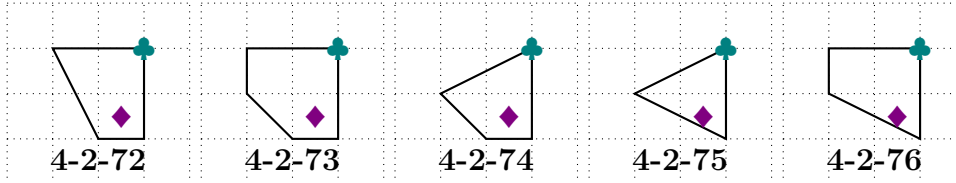
\begin{center}
\begin{tikzpicture}[scale=0.6]
\draw[dotted] (-2,2) grid (2,-2);
\draw[thick] (1,1) -- (-1,1) -- (0,-1) --  (1,-1) -- cycle;
\draw (1,1) node[color=teal]{\(\clubsuit\)};
\draw (1/2,-1/2) node[color=violet]{\(\vardiamondsuit\)};
\draw (0,-1) node[below]{\textbf{4-2-72}};
\end{tikzpicture} 
\begin{tikzpicture}[scale=0.6]
\draw[dotted] (-2,2) grid (2,-2);
\draw[thick] (1,1) -- (-1,1) -- (-1,0) -- (0,-1) --  (1,-1) -- cycle;
\draw (1,1) node[color=teal]{\(\clubsuit\)};
\draw (1/2,-1/2) node[color=violet]{\(\vardiamondsuit\)};
\draw (0,-1) node[below]{\textbf{4-2-73}};
\end{tikzpicture} 
\begin{tikzpicture}[scale=0.6]
\draw[dotted] (-2,2) grid (2,-2);
\draw[thick] (1,1) -- (-1,0) -- (0,-1) --  (1,-1) -- cycle;
\draw (1,1) node[color=teal]{\(\clubsuit\)};
\draw (1/2,-1/2) node[color=violet]{\(\vardiamondsuit\)};
\draw (0,-1) node[below]{\textbf{4-2-74}};
\end{tikzpicture} 
\begin{tikzpicture}[scale=0.6]
\draw[dotted] (-2,2) grid (2,-2);
\draw[thick] (1,1) -- (-1,0) --  (1,-1) -- cycle;
\draw (1,1) node[color=teal]{\(\clubsuit\)};
\draw (1/2,-1/2) node[color=violet]{\(\vardiamondsuit\)};
\draw (0,-1) node[below]{\textbf{4-2-75}};
\end{tikzpicture} 
\begin{tikzpicture}[scale=0.6]
\draw[dotted] (-2,2) grid (2,-2);
\draw[thick] (1,1) -- (-1,1) -- (-1,0) --  (1,-1) -- cycle;
\draw (1,1) node[color=teal]{\(\clubsuit\)};
\draw (1/2,-1/2) node[color=violet]{\(\vardiamondsuit\)};
\draw (0,-1) node[below]{\textbf{4-2-76}};
\end{tikzpicture} 
\end{center}

\textbf{If \((1/2, -1/2)\) is a vertex}
By applying the same exhaustion argument as before, we obtain the following list of six polytopes. 

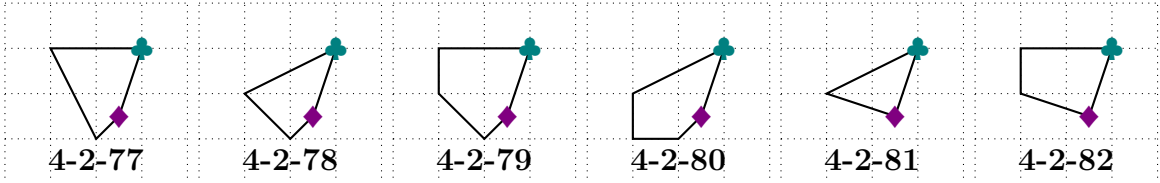
\begin{center}
\begin{tikzpicture}[scale=0.6]
\draw[dotted] (-2,2) grid (2,-2);
\draw[thick] (1,1) -- (-1,1) -- (0,-1) --  (1/2,-1/2) -- cycle;
\draw (1,1) node[color=teal]{\(\clubsuit\)};
\draw (1/2,-1/2) node[color=violet]{\(\vardiamondsuit\)};
\draw (0,-1) node[below]{\textbf{4-2-77}};
\end{tikzpicture} 
\begin{tikzpicture}[scale=0.6]
\draw[dotted] (-2,2) grid (2,-2);
\draw[thick] (1,1) -- (-1,0) -- (0,-1) --  (1/2,-1/2) -- cycle;
\draw (1,1) node[color=teal]{\(\clubsuit\)};
\draw (1/2,-1/2) node[color=violet]{\(\vardiamondsuit\)};
\draw (0,-1) node[below]{\textbf{4-2-78}};
\end{tikzpicture} 
\begin{tikzpicture}[scale=0.6]
\draw[dotted] (-2,2) grid (2,-2);
\draw[thick] (1,1) -- (-1,1) -- (-1,0) -- (0,-1) --  (1/2,-1/2) -- cycle;
\draw (1,1) node[color=teal]{\(\clubsuit\)};
\draw (1/2,-1/2) node[color=violet]{\(\vardiamondsuit\)};
\draw (0,-1) node[below]{\textbf{4-2-79}};
\end{tikzpicture} 
\begin{tikzpicture}[scale=0.6]
\draw[dotted] (-2,2) grid (2,-2);
\draw[thick] (1,1) -- (-1,0) -- (-1,-1) -- (0,-1) --  (1/2,-1/2) -- cycle;
\draw (1,1) node[color=teal]{\(\clubsuit\)};
\draw (1/2,-1/2) node[color=violet]{\(\vardiamondsuit\)};
\draw (0,-1) node[below]{\textbf{4-2-80}};
\end{tikzpicture} 
\begin{tikzpicture}[scale=0.6]
\draw[dotted] (-2,2) grid (2,-2);
\draw[thick] (1,1) -- (-1,0) --  (1/2,-1/2) -- cycle;
\draw (1,1) node[color=teal]{\(\clubsuit\)};
\draw (1/2,-1/2) node[color=violet]{\(\vardiamondsuit\)};
\draw (0,-1) node[below]{\textbf{4-2-81}};
\end{tikzpicture} 
\begin{tikzpicture}[scale=0.6]
\draw[dotted] (-2,2) grid (2,-2);
\draw[thick] (1,1) -- (-1,1) -- (-1,0) --  (1/2,-1/2) -- cycle;
\draw (1,1) node[color=teal]{\(\clubsuit\)};
\draw (1/2,-1/2) node[color=violet]{\(\vardiamondsuit\)};
\draw (0,-1) node[below]{\textbf{4-2-82}};
\end{tikzpicture} 
\end{center}

\begin{exa}
The embedding \textbf{4-2-81} is \(\bbP^4\). 
The action of \(\SL_2\times \SL_2\) factors through \(\SO_3\times \SL_2\) and is given by the standard action of this group on \(\bbC^3\times \bbC^2\). 
Blowing up well-chosen successive orbits allows to recover all the above embeddings but \textbf{4-2-72} and \textbf{4-2-77}.
\end{exa}

\textbf{We now assume that \(H_0=N(T_1)\times \{1\}\)}
Note that in this case, \(\rho(\clubsuit)\) is not primitive, so it cannot be in a facet whose relative interior intersects the valuation cone. 

\textbf{If \((0, \frac{1}{2})\) is not a vertex}
In this case, \((0,1)\) and \((0,-1)\) must be vertices. By exhaustion, we get the following list. 

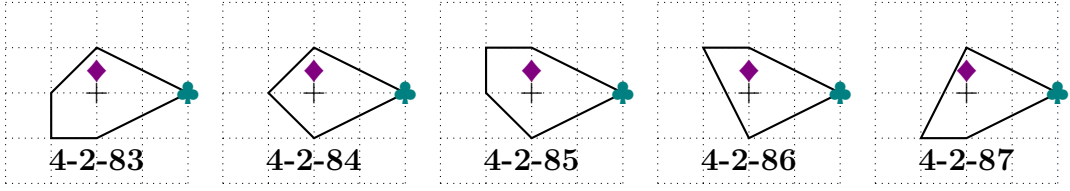
\begin{center}
\begin{tikzpicture}[scale=0.6]
\draw[dotted] (-2,2) grid (2,-2);
\draw (0,0) node{+};
\draw[thick] (2,0) -- (0,1) -- (-1,0) -- (-1,-1) -- (0,-1) -- cycle;
\draw (2,0) node[color=teal]{\(\clubsuit\)};
\draw (0,1/2) node[color=violet]{\(\vardiamondsuit\)};
\draw (0,-1) node[below]{\textbf{4-2-83}};
\end{tikzpicture} 
\begin{tikzpicture}[scale=0.6]
\draw[dotted] (-2,2) grid (2,-2);
\draw (0,0) node{+};
\draw[thick] (2,0) -- (0,1) -- (-1,0) -- (0,-1) -- cycle;
\draw (2,0) node[color=teal]{\(\clubsuit\)};
\draw (0,1/2) node[color=violet]{\(\vardiamondsuit\)};
\draw (0,-1) node[below]{\textbf{4-2-84}};
\end{tikzpicture} 
\begin{tikzpicture}[scale=0.6]
\draw[dotted] (-2,2) grid (2,-2);
\draw (0,0) node{+};
\draw[thick] (2,0) -- (0,1) -- (-1,1) -- (-1,0) -- (0,-1) -- cycle;
\draw (2,0) node[color=teal]{\(\clubsuit\)};
\draw (0,1/2) node[color=violet]{\(\vardiamondsuit\)};
\draw (0,-1) node[below]{\textbf{4-2-85}};
\end{tikzpicture} 
\begin{tikzpicture}[scale=0.6]
\draw[dotted] (-2,2) grid (2,-2);
\draw (0,0) node{+};
\draw[thick] (2,0) -- (0,1) -- (-1,1) -- (0,-1) -- cycle;
\draw (2,0) node[color=teal]{\(\clubsuit\)};
\draw (0,1/2) node[color=violet]{\(\vardiamondsuit\)};
\draw (0,-1) node[below]{\textbf{4-2-86}};
\end{tikzpicture} 
\begin{tikzpicture}[scale=0.6]
\draw[dotted] (-2,2) grid (2,-2);
\draw (0,0) node{+};
\draw[thick] (2,0) -- (0,1) -- (-1,-1) -- (0,-1) -- cycle;
\draw (2,0) node[color=teal]{\(\clubsuit\)};
\draw (0,1/2) node[color=violet]{\(\vardiamondsuit\)};
\draw (0,-1) node[below]{\textbf{4-2-87}};
\end{tikzpicture} 
\end{center}

\textbf{If \((0, \frac{1}{2})\) is a vertex}
In this case, \((0,1/2)\) and \((0,-1)\) must be vertices. By exhaustion, we get the following list. 

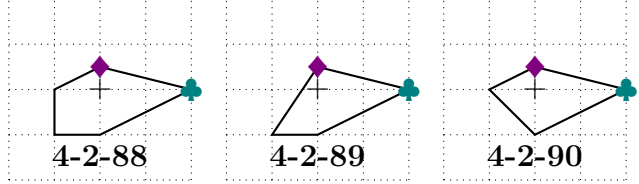
\begin{center}
\begin{tikzpicture}[scale=0.6]
\draw[dotted] (-2,2) grid (2,-2);
\draw (0,0) node{+};
\draw[thick] (2,0) -- (0,1/2) -- (-1,0) -- (-1,-1) -- (0,-1) -- cycle;
\draw (2,0) node[color=teal]{\(\clubsuit\)};
\draw (0,1/2) node[color=violet]{\(\vardiamondsuit\)};
\draw (0,-1) node[below]{\textbf{4-2-88}};
\end{tikzpicture} 
\begin{tikzpicture}[scale=0.6]
\draw[dotted] (-2,2) grid (2,-2);
\draw (0,0) node{+};
\draw[thick] (2,0) -- (0,1/2) -- (-1,-1) -- (0,-1) -- cycle;
\draw (2,0) node[color=teal]{\(\clubsuit\)};
\draw (0,1/2) node[color=violet]{\(\vardiamondsuit\)};
\draw (0,-1) node[below]{\textbf{4-2-89}};
\end{tikzpicture} 
\begin{tikzpicture}[scale=0.6]
\draw[dotted] (-2,2) grid (2,-2);
\draw (0,0) node{+};
\draw[thick] (2,0) -- (0,1/2) -- (-1,0) -- (0,-1) -- cycle;
\draw (2,0) node[color=teal]{\(\clubsuit\)};
\draw (0,1/2) node[color=violet]{\(\vardiamondsuit\)};
\draw (0,-1) node[below]{\textbf{4-2-90}};
\end{tikzpicture} 
\end{center}

\begin{exa}
The embeddings \textbf{4-2-84} and \textbf{4-2-90} are the product of embeddings \(\bbP^2\times \bbF_1\) and \(\bbP^2\times \bbP^2\). 
Blowing up closed orbits allows to recover \textbf{4-2-83}, \textbf{4-2-85} and \textbf{4-2-88}.
\end{exa}

\subsection{Dimension three, horospherical rank one}

We note that all horospherical structures under \(\SL_2\times \SL_2\times \bbG_m^n\) upgrade a toric structure under \(\bbG_m^{n+2}\). 

\begin{prop}
Let \(H=\ker(a_1\varpi_1+a_2\varpi_2+\chi_1)|_{B^-}\). 
The combinatorial data of \(G/H\) is as follows. 
\[\arraycolsep=12pt
\begin{array}{lll}
M=\langle a_1\varpi_1+a_2\varpi_2+\chi_1 \rangle  &  \rho(\clubsuit)=a_1 & \zeta(\clubsuit)=\{\alpha_1\} 
\\  
\Sigma=\emptyset & \rho(\varheartsuit)=a_2 &  \zeta(\varheartsuit)=\{\alpha_2\}
\\ 
\mathcal{D} = \{\clubsuit,\varheartsuit\} &  
\\
\midrule
\kappa = \alpha_1+\alpha_2 & f = (2+a_1x_1)(2+a_2x_2) & m_{\clubsuit} = m_{\varheartsuit}  = 2
\end{array}
\]
\end{prop}

The conditions on \(\frac{\rho(\alpha_i)}{2}\) show that if there exists a locally factorial \(G/H\)-reflexive polytope, then \(a_i\in \{-1,0,1\}\). 
If \(a_2=0\) then any embedding is a product of \(\bbP^1\) with an embedding of a rank one horospherical homogeneous space under \(\SL_2\times \bbG_m\). 
There are three such products, numbered as \textbf{3-1-3} to \textbf{3-1-5}. 
Finally, if both are non-zero, we obtain the following four polytopes.

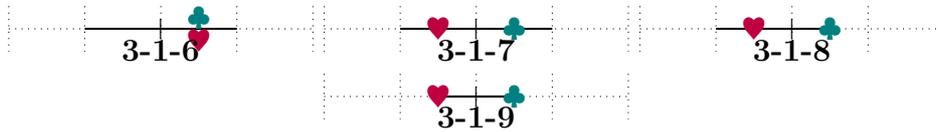
\begin{center}
\begin{tikzpicture}
\draw[dotted] (-2,0) -- (2,0);
\draw (0,0) node{+};
\foreach \i in {0,...,4}
{
\draw[dotted] (-2+\i,.3) -- (-2+\i,-.3);
}
\draw[thick] (-1,0) -- (1,0);
\draw (1/2,0) node[above=-4,color=teal]{\(\clubsuit\)};
\draw (1/2,0) node[below=-4,color=purple]{\(\varheartsuit\)};
\draw (0,0) node[below]{\textbf{3-1-6}};
\end{tikzpicture}
\begin{tikzpicture}
\draw[dotted] (-2,0) -- (2,0);
\draw (0,0) node{+};
\foreach \i in {0,...,4}
{
\draw[dotted] (-2+\i,.3) -- (-2+\i,-.3);
}
\draw[thick] (-1,0) -- (1,0);
\draw (1/2,0) node[color=teal]{\(\clubsuit\)};
\draw (-1/2,0) node[color=purple]{\(\varheartsuit\)};
\draw (0,0) node[below]{\textbf{3-1-7}};
\end{tikzpicture}
\begin{tikzpicture}
\draw[dotted] (-2,0) -- (2,0);
\draw (0,0) node{+};
\foreach \i in {0,...,4}
{
\draw[dotted] (-2+\i,.3) -- (-2+\i,-.3);
}
\draw[thick] (-1,0) -- (1/2,0);
\draw (1/2,0) node[color=teal]{\(\clubsuit\)};
\draw (-1/2,0) node[color=purple]{\(\varheartsuit\)};
\draw (0,0) node[below]{\textbf{3-1-8}};
\end{tikzpicture}
\begin{tikzpicture}
\draw[dotted] (-2,0) -- (2,0);
\draw (0,0) node{+};
\foreach \i in {0,...,4}
{
\draw[dotted] (-2+\i,.3) -- (-2+\i,-.3);
}
\draw[thick] (-1/2,0) -- (1/2,0);
\draw (1/2,0) node[color=teal]{\(\clubsuit\)};
\draw (-1/2,0) node[color=purple]{\(\varheartsuit\)};
\draw (0,0) node[below]{\textbf{3-1-9}};
\end{tikzpicture}
\end{center}

\begin{exa}
The first polytope corresponds to the embedding \(\bbP_{\bbP^1\times \bbP^1}(\mathcal{O}\oplus \mathcal{O}(1,1))\). 
The last one corresponds to \(\bbP^3\) equipped with the action of \(S(\GL_2\times \GL_2)\) induced by the block diagonal embedding of this group in \(\SL_4\) and the linear action of \(\SL_4\) on \(\bbC^4\). 
There are three orbits under this action: the two projective lines given by vanishing of the two first, or two last homogeneous coordinates, and the complement, which is our spherical homogeneous space. 
Blowing up one line yields the embedding associated to the third polytope above, while blowing up both lines yields the second, which may alternatively be described as \(\bbP_{\bbP^1\times \bbP^1}(\mathcal{O}\oplus \mathcal{O}(1,-1))\). 
\end{exa}

\subsection{Dimension four, horospherical rank two}

We note that all horospherical structures under \(\SL_2\times \SL_2\times \bbG_m^n\) upgrade a toric structure under \(\bbG_m^{n+2}\). 

\begin{prop}
Let \(H=\ker(a_1\varpi_1+a_2\varpi_2+\chi_1)\cap \ker(b_2\varpi_2+\chi_2) \subset B^-\). 
The combinatorial data of \(G/H\) is as follows. 
\[\arraycolsep=12pt
\begin{array}{lll}
M=\langle a_1\varpi_1+a_2\varpi_2+\chi_1, b_2\varpi_2+\chi_2 \rangle  &  \rho(\clubsuit)=(a_1,0) & \zeta(\clubsuit)=\{\alpha_1\} 
\\  
\Sigma=\emptyset & \rho(\varheartsuit)=(a_2,b_2) &  \zeta(\varheartsuit)=\{\alpha_2\}
\\ 
\mathcal{D} = \{\clubsuit,\varheartsuit\} &  
\\
\midrule
\kappa = \alpha_1+\alpha_2 & f = (2+a_1x_1)(2+a_2x_1+b_2x_2) & m_{\clubsuit} = m_{\varheartsuit}  = 2
\end{array}
\]
\end{prop}

Note that if \(a_2=b_2=0\), then any embedding is the product of \(\bbP^1\) with a rank two, dimension three embedding of a horospherical homogeneous space under \(\SL_2\times \bbG_m^2\). 
There are 21 such products, numbered from \textbf{4-2-91} to \textbf{4-2-111}.

We now assume that \(0<a_1\leq a_2\wedge b_2\) and \(0\leq a_2<b_2\). 
The conditions on the points \(\frac{\rho(\clubsuit)}{2}\) imply that \(a_1=1\), and the condition on \(\frac{\rho(\varheartsuit)}{2}\) imply that \((a_2,b_2)\) is primitive. 
When none of these points are vertices, it suffices to identify where these (half-integral) points may be located in a smooth Fano rank two polygon. 
We obtain the following list of twelve possibilities. 

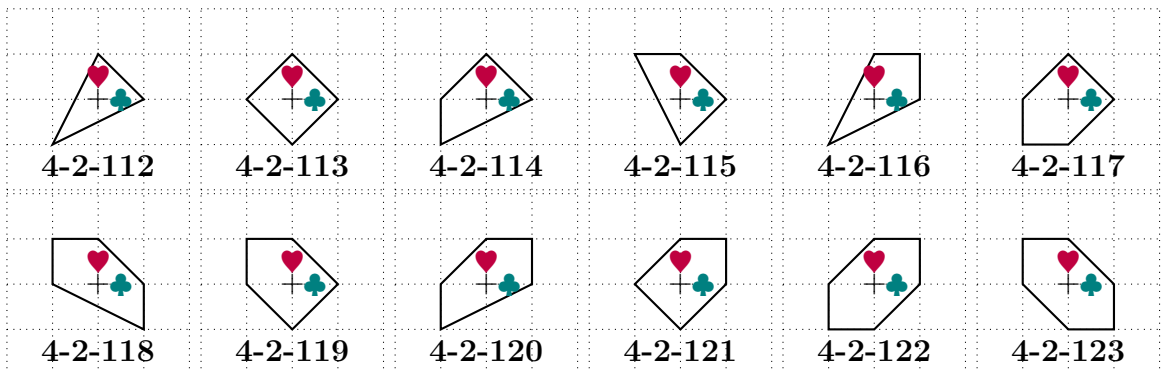
\begin{center}
\begin{tikzpicture}[scale=0.6]
\draw[dotted] (-2,2) grid (2,-2);
\draw (0,0) node{+};
\draw[thick] (1,0) -- (0,1) -- (-1,-1) -- cycle;
\draw (1/2,0) node[color=teal]{\(\clubsuit\)};
\draw (0,1/2) node[color=purple]{\(\varheartsuit\)};
\draw (0,-1) node[below]{\textbf{4-2-112}};
\end{tikzpicture}
\begin{tikzpicture}[scale=0.6]
\draw[dotted] (-2,2) grid (2,-2);
\draw (0,0) node{+};
\draw[thick] (1,0) -- (0,1) -- (-1,0) -- (0,-1) -- cycle;
\draw (1/2,0) node[color=teal]{\(\clubsuit\)};
\draw (0,1/2) node[color=purple]{\(\varheartsuit\)};
\draw (0,-1) node[below]{\textbf{4-2-113}};
\end{tikzpicture}
\begin{tikzpicture}[scale=0.6]
\draw[dotted] (-2,2) grid (2,-2);
\draw (0,0) node{+};
\draw[thick] (1,0) -- (0,1) -- (-1,0) -- (-1,-1) -- cycle;
\draw (1/2,0) node[color=teal]{\(\clubsuit\)};
\draw (0,1/2) node[color=purple]{\(\varheartsuit\)};
\draw (0,-1) node[below]{\textbf{4-2-114}};
\end{tikzpicture}
\begin{tikzpicture}[scale=0.6]
\draw[dotted] (-2,2) grid (2,-2);
\draw (0,0) node{+};
\draw[thick] (1,0) -- (0,1) -- (-1,1) -- (0,-1) -- cycle;
\draw (1/2,0) node[color=teal]{\(\clubsuit\)};
\draw (0,1/2) node[color=purple]{\(\varheartsuit\)};
\draw (0,-1) node[below]{\textbf{4-2-115}};
\end{tikzpicture}
\begin{tikzpicture}[scale=0.6]
\draw[dotted] (-2,2) grid (2,-2);
\draw (0,0) node{+};
\draw[thick] (1,0) -- (1,1) -- (0,1) -- (-1,-1) -- cycle;
\draw (1/2,0) node[color=teal]{\(\clubsuit\)};
\draw (0,1/2) node[color=purple]{\(\varheartsuit\)};
\draw (0,-1) node[below]{\textbf{4-2-116}};
\end{tikzpicture}
\begin{tikzpicture}[scale=0.6]
\draw[dotted] (-2,2) grid (2,-2);
\draw (0,0) node{+};
\draw[thick] (1,0) -- (0,1) -- (-1,0) -- (-1,-1) -- (0,-1) -- cycle;
\draw (1/2,0) node[color=teal]{\(\clubsuit\)};
\draw (0,1/2) node[color=purple]{\(\varheartsuit\)};
\draw (0,-1) node[below]{\textbf{4-2-117}};
\end{tikzpicture}
\begin{tikzpicture}[scale=0.6]
\draw[dotted] (-2,2) grid (2,-2);
\draw (0,0) node{+};
\draw[thick] (1,0) -- (0,1) -- (-1,1)  -- (-1,0) -- (1,-1) -- cycle;
\draw (1/2,0) node[color=teal]{\(\clubsuit\)};
\draw (0,1/2) node[color=purple]{\(\varheartsuit\)};
\draw (0,-1) node[below]{\textbf{4-2-118}};
\end{tikzpicture}
\begin{tikzpicture}[scale=0.6]
\draw[dotted] (-2,2) grid (2,-2);
\draw (0,0) node{+};
\draw[thick] (1,0) -- (0,1) -- (-1,1) -- (-1,0) -- (0,-1) -- cycle;
\draw (1/2,0) node[color=teal]{\(\clubsuit\)};
\draw (0,1/2) node[color=purple]{\(\varheartsuit\)};
\draw (0,-1) node[below]{\textbf{4-2-119}};
\end{tikzpicture}
\begin{tikzpicture}[scale=0.6]
\draw[dotted] (-2,2) grid (2,-2);
\draw (0,0) node{+};
\draw[thick] (1,0) -- (1,1) -- (0,1) -- (-1,0) -- (-1,-1) -- cycle;
\draw (1/2,0) node[color=teal]{\(\clubsuit\)};
\draw (0,1/2) node[color=purple]{\(\varheartsuit\)};
\draw (0,-1) node[below]{\textbf{4-2-120}};
\end{tikzpicture}
\begin{tikzpicture}[scale=0.6]
\draw[dotted] (-2,2) grid (2,-2);
\draw (0,0) node{+};
\draw[thick] (1,0) -- (1,1) -- (0,1) -- (-1,0) -- (0,-1) -- cycle;
\draw (1/2,0) node[color=teal]{\(\clubsuit\)};
\draw (0,1/2) node[color=purple]{\(\varheartsuit\)};
\draw (0,-1) node[below]{\textbf{4-2-121}};
\end{tikzpicture}
\begin{tikzpicture}[scale=0.6]
\draw[dotted] (-2,2) grid (2,-2);
\draw (0,0) node{+};
\draw[thick] (1,0) -- (1,1) -- (0,1) -- (-1,0) -- (-1,-1) -- (0,-1) -- cycle;
\draw (1/2,0) node[color=teal]{\(\clubsuit\)};
\draw (0,1/2) node[color=purple]{\(\varheartsuit\)};
\draw (0,-1) node[below]{\textbf{4-2-122}};
\end{tikzpicture}
\begin{tikzpicture}[scale=0.6]
\draw[dotted] (-2,2) grid (2,-2);
\draw (0,0) node{+};
\draw[thick] (1,0) -- (0,1) -- (-1,1) -- (-1,0) -- (0,-1) -- (1,-1) -- cycle;
\draw (1/2,0) node[color=teal]{\(\clubsuit\)};
\draw (0,1/2) node[color=purple]{\(\varheartsuit\)};
\draw (0,-1) node[below]{\textbf{4-2-123}};
\end{tikzpicture}
\end{center}

Assume now that there is only one non-integral vertex. 
Without loss of generality, we assume that \(\frac{\rho(\clubsuit)}{2}\) is a vertex. 
Then the conditions on locally factorial \(G/H\)-reflexive polytopes are the same as for rank two horospherical threefolds under \(\SL_2\times \bbG_m^2\), with the added condition that \((a_2/2,b_2/2)\) is in the interior of the polytope. 
We can thus go through the list obtained previously, and obtain the following 16 possibilities. 

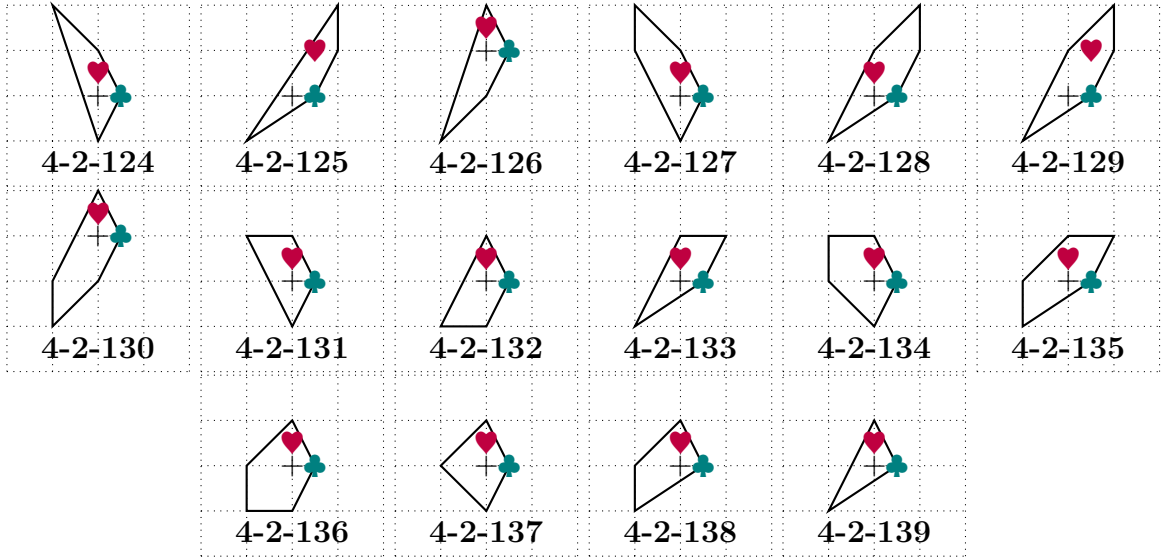
\begin{center}
\begin{tikzpicture}[scale=0.6]
\draw[dotted] (-2,2) grid (2,-2);
\draw (0,0) node{+};
\draw[thick] (1/2,0) -- (0,1) -- (-1,2) -- (0,-1) -- cycle;
\draw (1/2,0) node[color=teal]{\(\clubsuit\)};
\draw (0,1/2) node[color=purple]{\(\varheartsuit\)};
\draw (0,-1) node[below]{\textbf{4-2-124}};
\end{tikzpicture}
\begin{tikzpicture}[scale=0.6]
\draw[dotted] (-2,2) grid (2,-2);
\draw (0,0) node{+};
\draw[thick] (1/2,0) -- (1,1) -- (1,2) -- (-1,-1) -- cycle;
\draw (1/2,0) node[color=teal]{\(\clubsuit\)};
\draw (1/2,1) node[color=purple]{\(\varheartsuit\)};
\draw (0,-1) node[below]{\textbf{4-2-125}};
\end{tikzpicture}
\begin{tikzpicture}[scale=0.6]
\draw[dotted] (-2,1) grid (2,-3);
\draw (0,0) node{+};
\draw[thick] (1/2,0) -- (0,1) -- (-1,-2) -- (0,-1) -- cycle;
\draw (1/2,0) node[color=teal]{\(\clubsuit\)};
\draw (0,1/2) node[color=purple]{\(\varheartsuit\)};
\draw (0,-2) node[below]{\textbf{4-2-126}};
\end{tikzpicture}
\begin{tikzpicture}[scale=0.6]
\draw[dotted] (-2,2) grid (2,-2);
\draw (0,0) node{+};
\draw[thick] (1/2,0) -- (0,1) -- (-1,2) -- (-1,1) -- (0,-1) -- cycle;
\draw (1/2,0) node[color=teal]{\(\clubsuit\)};
\draw (0,1/2) node[color=purple]{\(\varheartsuit\)};
\draw (0,-1) node[below]{\textbf{4-2-127}};
\end{tikzpicture}
\begin{tikzpicture}[scale=0.6]
\draw[dotted] (-2,2) grid (2,-2);
\draw (0,0) node{+};
\draw[thick] (1/2,0) -- (1,1) -- (1,2) -- (0,1) -- (-1,-1) -- cycle;
\draw (1/2,0) node[color=teal]{\(\clubsuit\)};
\draw (0,1/2) node[color=purple]{\(\varheartsuit\)};
\draw (0,-1) node[below]{\textbf{4-2-128}};
\end{tikzpicture}
\begin{tikzpicture}[scale=0.6]
\draw[dotted] (-2,2) grid (2,-2);
\draw (0,0) node{+};
\draw[thick] (1/2,0) -- (1,1) -- (1,2) -- (0,1) -- (-1,-1) -- cycle;
\draw (1/2,0) node[color=teal]{\(\clubsuit\)};
\draw (1/2,1) node[color=purple]{\(\varheartsuit\)};
\draw (0,-1) node[below]{\textbf{4-2-129}};
\end{tikzpicture}
\begin{tikzpicture}[scale=0.6]
\draw[dotted] (-2,1) grid (2,-3);
\draw (0,0) node{+};
\draw[thick] (1/2,0) -- (0,1) -- (-1,-1) -- (-1,-2) -- (0,-1) -- cycle;
\draw (1/2,0) node[color=teal]{\(\clubsuit\)};
\draw (0,1/2) node[color=purple]{\(\varheartsuit\)};
\draw (0,-2) node[below]{\textbf{4-2-130}};
\end{tikzpicture}
\begin{tikzpicture}[scale=0.6]
\draw[dotted] (-2,2) grid (2,-2);
\draw (0,0) node{+};
\draw[thick] (1/2,0) -- (0,1) -- (-1,1) -- (0,-1) -- cycle;
\draw (1/2,0) node[color=teal]{\(\clubsuit\)};
\draw (0,1/2) node[color=purple]{\(\varheartsuit\)};
\draw (0,-1) node[below]{\textbf{4-2-131}};
\end{tikzpicture}
\begin{tikzpicture}[scale=0.6]
\draw[dotted] (-2,2) grid (2,-2);
\draw (0,0) node{+};
\draw[thick] (1/2,0) -- (0,1) -- (-1,-1) -- (0,-1) -- cycle;
\draw (1/2,0) node[color=teal]{\(\clubsuit\)};
\draw (0,1/2) node[color=purple]{\(\varheartsuit\)};
\draw (0,-1) node[below]{\textbf{4-2-132}};
\end{tikzpicture}
\begin{tikzpicture}[scale=0.6]
\draw[dotted] (-2,2) grid (2,-2);
\draw (0,0) node{+};
\draw[thick] (1/2,0) -- (1,1) -- (0,1) -- (-1,-1) -- cycle;
\draw (1/2,0) node[color=teal]{\(\clubsuit\)};
\draw (0,1/2) node[color=purple]{\(\varheartsuit\)};
\draw (0,-1) node[below]{\textbf{4-2-133}};
\end{tikzpicture}
\begin{tikzpicture}[scale=0.6]
\draw[dotted] (-2,2) grid (2,-2);
\draw (0,0) node{+};
\draw[thick] (1/2,0) -- (0,1) -- (-1,1) -- (-1,0) -- (0,-1) -- cycle;
\draw (1/2,0) node[color=teal]{\(\clubsuit\)};
\draw (0,1/2) node[color=purple]{\(\varheartsuit\)};
\draw (0,-1) node[below]{\textbf{4-2-134}};
\end{tikzpicture}
\begin{tikzpicture}[scale=0.6]
\draw[dotted] (-2,2) grid (2,-2);
\draw (0,0) node{+};
\draw[thick] (1/2,0) -- (1,1) -- (0,1) -- (-1,0) -- (-1,-1) -- cycle;
\draw (1/2,0) node[color=teal]{\(\clubsuit\)};
\draw (0,1/2) node[color=purple]{\(\varheartsuit\)};
\draw (0,-1) node[below]{\textbf{4-2-135}};
\end{tikzpicture}
\begin{tikzpicture}[scale=0.6]
\draw[dotted] (-2,2) grid (2,-2);
\draw (0,0) node{+};
\draw[thick] (1/2,0) -- (0,1) -- (-1,0) -- (-1,-1) -- (0,-1) -- cycle;
\draw (1/2,0) node[color=teal]{\(\clubsuit\)};
\draw (0,1/2) node[color=purple]{\(\varheartsuit\)};
\draw (0,-1) node[below]{\textbf{4-2-136}};
\end{tikzpicture}
\begin{tikzpicture}[scale=0.6]
\draw[dotted] (-2,2) grid (2,-2);
\draw (0,0) node{+};
\draw[thick] (1/2,0) -- (0,1) -- (-1,0) -- (0,-1) -- cycle;
\draw (1/2,0) node[color=teal]{\(\clubsuit\)};
\draw (0,1/2) node[color=purple]{\(\varheartsuit\)};
\draw (0,-1) node[below]{\textbf{4-2-137}};
\end{tikzpicture}
\begin{tikzpicture}[scale=0.6]
\draw[dotted] (-2,2) grid (2,-2);
\draw (0,0) node{+};
\draw[thick] (1/2,0) -- (0,1) -- (-1,0) -- (-1,-1) -- cycle;
\draw (1/2,0) node[color=teal]{\(\clubsuit\)};
\draw (0,1/2) node[color=purple]{\(\varheartsuit\)};
\draw (0,-1) node[below]{\textbf{4-2-138}};
\end{tikzpicture}
\begin{tikzpicture}[scale=0.6]
\draw[dotted] (-2,2) grid (2,-2);
\draw (0,0) node{+};
\draw[thick] (1/2,0) -- (0,1) -- (-1,-1) -- cycle;
\draw (1/2,0) node[color=teal]{\(\clubsuit\)};
\draw (0,1/2) node[color=purple]{\(\varheartsuit\)};
\draw (0,-1) node[below]{\textbf{4-2-139}};
\end{tikzpicture}
\end{center}

We now deal with the case when both \(\frac{\rho(\clubsuit)}{2}\) and \(\frac{\rho(\varheartsuit)}{2}\) are vertices. 
There are infinitely many possibilities for \((a_2,b_2)\) that we need to rule out. 
For this, note that a locally factorial \(G/H\)-reflexive polytope for this homogeneous space cannot contain integral points different from its vertices or the origin. 
Since \((0,0)\) and \((1/2,0)\) are in the polytope, and \((1,1)\) cannot be in the relative interior of one of its facets, the polytope cannot contain any point \((x,y)\) satisfying \(x\leq y \leq 2x-1\) except \((1,1)\) as a vertex. 
In particular, if \(a_2>2\) we must have \(b_2 > 2a_2-1\).

Furthermore, \((1/2,0)\) and \((a_2/2,b_2/2)\) can be in the same facet only if \((1,0)\) and \((a_2,b_2)\) form a basis of \(\bbZ^2\), that is only if \(a_2=0\) (in which case \(b_2=1\)). 
When \(a_2\neq 0\), the vertex adjacent to \((1/2,0)\) in counterclockwise order is of the form \((k,1)\) for some \(k\in \bbZ_{>0}\).
Assume that \(k>1\), then write \((1,1)=x(k,1)+y(a_2/2,b_2/2)\). 
We have \(y=\frac{2(1-k)}{a_2-kb_2}\), and \(x+\frac{b_2}{2}y=1\). 
By the conditions on \(k\), \(a_2\) and \(b_2\), we have \(y>0\). 
Furthermore, \(b_2>a_2\) also implies that \(y< \frac{2}{b_2}\), hence \(x>0\). 
Finally, \(x+\frac{b_2}{2}y=1\) implies that \(x+y\leq 1\) since \(1\leq a_2<b_2\). 
As a consequence, \((1,1)\) is an integral point of the polytope which is not a vertex, a contradiction.  
We have shown by contradiction that \((1,1)\) must be the vertex adjacent to \((1/2,0)\) if \(a_2\neq 0\). 

Keep assuming \(a_2\neq 0\), and consider now the point \((1,2)\). 
It is an integral point, so it cannot be in the relative interior of any face of the polytope. 
Since \((1,1)\) and the origin are in the polytope, we deduce that we cannot have \(b_2\geq 2a_2\) for \(a_2\geq 2\).  
The two inequalities obtained give \(2a_2-1<b_2<2a_2\) for \(a_2>2\), which does not have any integral solution. 
We have thus shown that either \((a_2,b_2)=(0,1)\), or \(a_2=1\), or \(a_2=2\) and \(b_2=3\).

Finally, noting that the vertex adjacent to \((1/2,0)\) in clockwise order is of the form \((k,-1)\) with \(k\leq -1\), we have that \((0,1)\) is an integral point of the polytope which is not vertex if \(a_2=1\) and \(b_2\geq 4\). 
Only four possibilities remain to be explored for \((a_2,b_2)\): \((0,1)\), \((1,2)\), \((1,3)\) or \((2,3)\). 

For each of these cases, by examining the different possibilities for vertices, we quickly reach a classification, with 19 polytopes for \((a_2,b_2)=(0,1)\), one polytope for \((a_2,b_2)=(1,2)\) and \((a_2,b_2)=(2,3)\), and no polytope for \((1,3)\). 

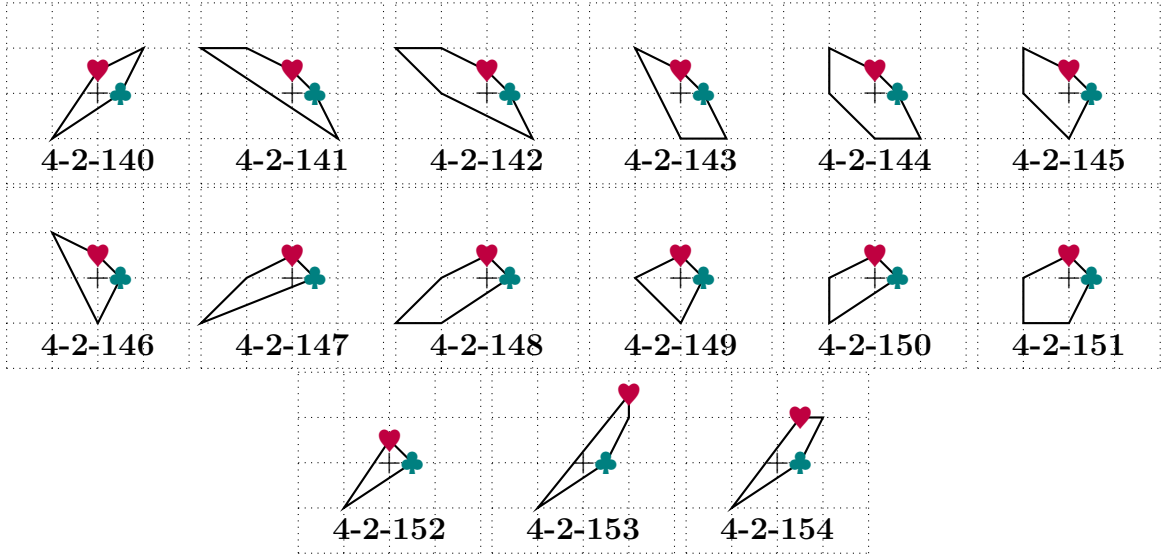
\begin{center}
\begin{tikzpicture}[scale=0.6]
\draw[dotted] (-2,2) grid (2,-2);
\draw (0,0) node{+};
\draw[thick] (1/2,0) -- (1,1) -- (0,1/2) -- (-1,-1) -- cycle;
\draw (1/2,0) node[color=teal]{\(\clubsuit\)};
\draw (0,1/2) node[color=purple]{\(\varheartsuit\)};
\draw (0,-1) node[below]{\textbf{4-2-140}};
\end{tikzpicture}
\begin{tikzpicture}[scale=0.6]
\draw[dotted] (-2,2) grid (2,-2);
\draw (0,0) node{+};
\draw[thick] (1/2,0) -- (0,1/2) -- (-1,1) -- (-2,1) -- (1,-1) -- cycle;
\draw (1/2,0) node[color=teal]{\(\clubsuit\)};
\draw (0,1/2) node[color=purple]{\(\varheartsuit\)};
\draw (0,-1) node[below]{\textbf{4-2-141}};
\end{tikzpicture}
\begin{tikzpicture}[scale=0.6]
\draw[dotted] (-2,2) grid (2,-2);
\draw (0,0) node{+};
\draw[thick] (1/2,0) -- (0,1/2) -- (-1,1) -- (-2,1) -- (-1,0) -- (1,-1) -- cycle;
\draw (1/2,0) node[color=teal]{\(\clubsuit\)};
\draw (0,1/2) node[color=purple]{\(\varheartsuit\)};
\draw (0,-1) node[below]{\textbf{4-2-142}};
\end{tikzpicture}
\begin{tikzpicture}[scale=0.6]
\draw[dotted] (-2,2) grid (2,-2);
\draw (0,0) node{+};
\draw[thick] (1/2,0) -- (0,1/2) -- (-1,1) -- (0,-1) -- (1,-1) -- cycle;
\draw (1/2,0) node[color=teal]{\(\clubsuit\)};
\draw (0,1/2) node[color=purple]{\(\varheartsuit\)};
\draw (0,-1) node[below]{\textbf{4-2-143}};
\end{tikzpicture}
\begin{tikzpicture}[scale=0.6]
\draw[dotted] (-2,2) grid (2,-2);
\draw (0,0) node{+};
\draw[thick] (1/2,0) -- (0,1/2) -- (-1,1) -- (-1,0) -- (0,-1) -- (1,-1) -- cycle;
\draw (1/2,0) node[color=teal]{\(\clubsuit\)};
\draw (0,1/2) node[color=purple]{\(\varheartsuit\)};
\draw (0,-1) node[below]{\textbf{4-2-144}};
\end{tikzpicture}
\begin{tikzpicture}[scale=0.6]
\draw[dotted] (-2,2) grid (2,-2);
\draw (0,0) node{+};
\draw[thick] (1/2,0) -- (0,1/2) -- (-1,1) -- (-1,0) -- (0,-1) -- cycle;
\draw (1/2,0) node[color=teal]{\(\clubsuit\)};
\draw (0,1/2) node[color=purple]{\(\varheartsuit\)};
\draw (0,-1) node[below]{\textbf{4-2-145}};
\end{tikzpicture}
\begin{tikzpicture}[scale=0.6]
\draw[dotted] (-2,2) grid (2,-2);
\draw (0,0) node{+};
\draw[thick] (1/2,0) -- (0,1/2) -- (-1,1) -- (0,-1) -- cycle;
\draw (1/2,0) node[color=teal]{\(\clubsuit\)};
\draw (0,1/2) node[color=purple]{\(\varheartsuit\)};
\draw (0,-1) node[below]{\textbf{4-2-146}};
\end{tikzpicture}
\begin{tikzpicture}[scale=0.6]
\draw[dotted] (-2,2) grid (2,-2);
\draw (0,0) node{+};
\draw[thick] (1/2,0) -- (0,1/2) -- (-1,0) -- (-2,-1) -- cycle;
\draw (1/2,0) node[color=teal]{\(\clubsuit\)};
\draw (0,1/2) node[color=purple]{\(\varheartsuit\)};
\draw (0,-1) node[below]{\textbf{4-2-147}};
\end{tikzpicture}
\begin{tikzpicture}[scale=0.6]
\draw[dotted] (-2,2) grid (2,-2);
\draw (0,0) node{+};
\draw[thick] (1/2,0) -- (0,1/2) -- (-1,0) -- (-2,-1) -- (-1,-1) -- cycle;
\draw (1/2,0) node[color=teal]{\(\clubsuit\)};
\draw (0,1/2) node[color=purple]{\(\varheartsuit\)};
\draw (0,-1) node[below]{\textbf{4-2-148}};
\end{tikzpicture}
\begin{tikzpicture}[scale=0.6]
\draw[dotted] (-2,2) grid (2,-2);
\draw (0,0) node{+};
\draw[thick] (1/2,0) -- (0,1/2) -- (-1,0) -- (0,-1) -- cycle;
\draw (1/2,0) node[color=teal]{\(\clubsuit\)};
\draw (0,1/2) node[color=purple]{\(\varheartsuit\)};
\draw (0,-1) node[below]{\textbf{4-2-149}};
\end{tikzpicture}
\begin{tikzpicture}[scale=0.6]
\draw[dotted] (-2,2) grid (2,-2);
\draw (0,0) node{+};
\draw[thick] (1/2,0) -- (0,1/2) -- (-1,0) -- (-1,-1) -- cycle;
\draw (1/2,0) node[color=teal]{\(\clubsuit\)};
\draw (0,1/2) node[color=purple]{\(\varheartsuit\)};
\draw (0,-1) node[below]{\textbf{4-2-150}};
\end{tikzpicture}
\begin{tikzpicture}[scale=0.6]
\draw[dotted] (-2,2) grid (2,-2);
\draw (0,0) node{+};
\draw[thick] (1/2,0) -- (0,1/2) -- (-1,0) -- (-1,-1) -- (0,-1) -- cycle;
\draw (1/2,0) node[color=teal]{\(\clubsuit\)};
\draw (0,1/2) node[color=purple]{\(\varheartsuit\)};
\draw (0,-1) node[below]{\textbf{4-2-151}};
\end{tikzpicture}
\begin{tikzpicture}[scale=0.6]
\draw[dotted] (-2,2) grid (2,-2);
\draw (0,0) node{+};
\draw[thick] (1/2,0) -- (0,1/2) -- (-1,-1) -- cycle;
\draw (1/2,0) node[color=teal]{\(\clubsuit\)};
\draw (0,1/2) node[color=purple]{\(\varheartsuit\)};
\draw (0,-1) node[below]{\textbf{4-2-152}};
\end{tikzpicture}
\begin{tikzpicture}[scale=0.6]
\draw[dotted] (-2,2) grid (2,-2);
\draw (0,0) node{+};
\draw[thick] (1/2,0) -- (1,1) -- (1,3/2) -- (-1,-1) -- cycle;
\draw (1/2,0) node[color=teal]{\(\clubsuit\)};
\draw (1,3/2) node[color=purple]{\(\varheartsuit\)};
\draw (0,-1) node[below]{\textbf{4-2-153}};
\end{tikzpicture}
\begin{tikzpicture}[scale=0.6]
\draw[dotted] (-2,2) grid (2,-2);
\draw (0,0) node{+};
\draw[thick] (1/2,0) -- (1,1) -- (1/2,1) -- (-1,-1) -- cycle;
\draw (1/2,0) node[color=teal]{\(\clubsuit\)};
\draw (1/2,1) node[color=purple]{\(\varheartsuit\)};
\draw (0,-1) node[below]{\textbf{4-2-154}};
\end{tikzpicture}
\end{center}

\begin{exa}
The embedding \textbf{4-2-152} is \(\bbP^4\), where the action is induced by the linear action on the affine chart \(\bbC^4=\bbC^2\times \bbC^2\) given by the product of the standard actions for each factor \(\SL_2\). 
The orbits are: the fixed point \(\{0\}\subset \bbC^4\), the two \(\bbC^2\setminus\{0\}\times \{0\}\) factors in \(\bbC^4\setminus \{0\}\), the complement of these three orbits in \(\bbC^4\) which is our homogeneous space, then the corresponding non-zero orbits on the hyperplane \(\bbP^3\) at infinity. 

Blowing up a closed orbit in the hyperplane at infinity yields the embedding \textbf{4-2-150}. 
Blowing up one \(\bbC^2\setminus\{0\}\times \{0\}\) yields the embedding \textbf{4-2-139}, and blowing up both yields \textbf{4-2-112}. 
Blowing up the fixed point yields \textbf{4-2-140}. 
Many other examples can be obtained by wisely chosen blowups. 
\end{exa}

\section{Locally factorial Fano spherical \(\SL_3\times \bbG_m^n\)-fourfolds}
\label{section_SL3}

\subsection{Dimension four, rank one symmetric}

\begin{prop}
The combinatorial data for \(\SL_3/S(\GL_1\times \GL_2)\) is as follows. 
\[\arraycolsep=12pt
\begin{array}{lll}
M=\langle \alpha_1+\alpha_2 \rangle  &  \rho(\clubsuit)=1 & \zeta(\clubsuit)=\{\alpha_1\} 
\\  
\Sigma=\alpha_1+\alpha_2 & \rho(\varheartsuit)=1 &  \zeta(\varheartsuit)=\{\alpha_2\}
\\ 
\mathcal{D} = \{\clubsuit,\varheartsuit\} &  
\\
\midrule
\kappa = 2\alpha_1+2\alpha_2 & f = (2+x_1)^3 & m_{\clubsuit} = m_{\varheartsuit} = 2
\end{array}
\]
\end{prop}

\begin{proof}
Since \(\SL_3/S(\GL_1\times \GL_2)\) is a rank one, Hermitian symmetric space, its combinatorial data is well known since \cite{Vust_1990}. 
It is very quick to recover this data in this case: consider the product \(\bbP^2\times \bbP^2\) equipped with the action of \(\SL_3\) induced by 
\[ A\cdot \left(\begin{bmatrix}x_1 \\ y_1 \\ z_1 \end{bmatrix}, \begin{bmatrix}x_2 & y_2 & z_2 \end{bmatrix}\right) = \left(A\begin{bmatrix}x_1 \\ y_1 \\ z_1 \end{bmatrix}, \begin{bmatrix}x_2 & y_2 & z_2 \end{bmatrix}A^{-1}\right) \]
More geometrically, we consider an element of \(\bbP^2\times \bbP^2\) as the data of a line and a plane in \(\bbC^3\), equipped with the natural linear action of \(\SL_3\). 
The stabilizer of \(\left(\begin{bmatrix}1 \\ 0 \\ 0 \end{bmatrix}, \begin{bmatrix}1 & 0 & 0 \end{bmatrix}\right)\) is \(S(\GL_1\times \GL_2)\), its orbit is open, and its complement is the \(\SL_3\)-orbit defined by \(\{x_1x_2+y_1y_2+z_1z_2=0\}\), which corresponds to flags in \(\bbC^3\). 

We consider now the action of the Borel subgroup \(B\) of upper triangular matrices. 
The colors are easily identified as \(D_1=\{x_2=0\}\) and \(D_2=\{z_1=0\}\), the numbering being chosen so that \(\zeta(D_i)=\{\alpha_i\}\). 
We deduce that the adapted parabolic is \(B\), that the rank is one, that \(\kappa=2\alpha_1+2\alpha_2\) and that \(m_{D_i}=2\). 

Observe that the rational function 
\[ \left(\begin{bmatrix}x_1 \\ y_1 \\ z_1 \end{bmatrix}, \begin{bmatrix}x_2 & y_2 & z_2 \end{bmatrix}\right) \mapsto \frac{x_2z_1}{x_1x_2+y_1y_2+z_1z_2} \]
is \(B\)-semi-invariant with weight \(\alpha_1+\alpha_2\). 
Since it vanishes to order \(1\) on the colors, we deduce that its weight generates \(M\) as well as the images of the colors. 
\end{proof}

In view of the combinatorial data, there is only one possible complete embedding. 
It has already been described in the previous proof as \(\bbP^2\times (\bbP^2)^*\) equipped with the diagonal action. 
Its polytope is the following.

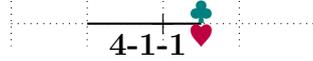
\begin{center}
\begin{tikzpicture}
\draw[dotted] (-2,0) -- (2,0);
\draw (0,0) node{+};
\foreach \i in {0,...,4}
{
\draw[dotted] (-2+\i,.3) -- (-2+\i,-.3);
}
\draw[thick] (-1,0) -- (1/2,0);
\draw (1/2,0) node[above=-4, color=teal]{\(\clubsuit\)};
\draw (1/2,0) node[below=-4, color=purple]{\(\varheartsuit\)};
\draw (-0.2,0) node[below]{\textbf{4-1-1}};
\end{tikzpicture}
\end{center}

\subsection{Dimension four, rank one horosymmetric}

Consider the case of a non-trivial parabolic induction from \(\SL_2/T\) with respect to a maximal parabolic subgroup in \(\SL_3\). 
The combinatorial data is deduced from the case of \(\SL_2/T\) and the parabolic induction procedure. 

\begin{prop}
For \(H=\langle Q^u , T \rangle \), the combinatorial data is as follows. 
\[\arraycolsep=12pt
\begin{array}{lll}
M=\langle \alpha_2 \rangle  &  \rho(\clubsuit)=-1 & \zeta(\clubsuit)=\{\alpha_1\} 
\\  
\Sigma=\alpha_2 & \rho(\varheartsuit)=\rho(\vardiamondsuit)=1 &  \zeta(\varheartsuit)=\zeta(\vardiamondsuit)=\{\alpha_2\}
\\ 
\mathcal{D} = \{\clubsuit,\varheartsuit,\vardiamondsuit\} &   &  
\\
\midrule
\kappa = 2\alpha_1+2\alpha_2 & f = (2-x_1)(1+x_1)(4+x_1) & m_{\clubsuit} = 2, \quad m_{\vardiamondsuit}=m_{\varheartsuit} = 1
\end{array}
\]
\end{prop}

There are two possible locally factorial \(G/H\)-reflexive polytopes:

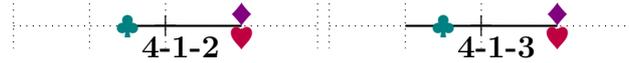
\begin{center}
\begin{tikzpicture}
\draw[dotted] (-2,0) -- (2,0);
\draw (0,0) node{+};
\foreach \i in {0,...,4}
{
\draw[dotted] (-2+\i,.3) -- (-2+\i,-.3);
}
\draw[thick] (-1/2,0) -- (1,0);
\draw (-1/2,0) node[color=teal]{\(\clubsuit\)};
\draw (1,0) node[below=-4, color=purple]{\(\varheartsuit\)};
\draw (1,0) node[above=-4, color=violet]{\(\vardiamondsuit\)};
\draw (0.2,0) node[below]{\textbf{4-1-2}};
\end{tikzpicture}
\begin{tikzpicture}
\draw[dotted] (-2,0) -- (2,0);
\draw (0,0) node{+};
\foreach \i in {0,...,4}
{
\draw[dotted] (-2+\i,.3) -- (-2+\i,-.3);
}
\draw[thick] (-1,0) -- (1,0);
\draw (-1/2,0) node[color=teal]{\(\clubsuit\)};
\draw (1,0) node[below=-4, color=purple]{\(\varheartsuit\)};
\draw (1,0) node[above=-4, color=violet]{\(\vardiamondsuit\)};
\draw (0.2,0) node[below]{\textbf{4-1-3}};
\end{tikzpicture}
\end{center}

The corresponding embeddings are easily identified: consider \(\bbP^2\times \bbP^2\) equipped with the diagonal action of \(\SL_3\). 
Then there are two orbits: the diagonal embedding of \(\bbP^2\) and its complement. 
It is easy to check that the complement is indeed our spherical homogeneous space. Note that this embedding is colored, since the closed orbit has codimension two, hence it corresponds to the first embedding. 
The second embedding is the blowup \(X\) of \(\bbP^2\times \bbP^2\) along the diagonal. 

Consider now the case of non-trivial parabolic induction from \(\SL_2/N(T)\) with respect to a maximal parabolic subgroup \(Q\) in \(\SL_3\). 
The combinatorial data is deduced from the case of \(\SL_2/N(T)\) and the parabolic induction procedure. 

\begin{prop}
For \(H=N(\langle Q^u , T \rangle) \), the combinatorial data is as follows. 
\[\arraycolsep=12pt
\begin{array}{lll}
M=\langle 2\alpha_2 \rangle  &  \rho(\clubsuit)=-2 & \zeta(\clubsuit)=\{\alpha_1\} 
\\  
\Sigma=\{2\alpha_2\} & \rho(\varheartsuit)=2 &  \zeta(\varheartsuit)=\{\alpha_2\}
\\ 
\mathcal{D} = \{\clubsuit,\varheartsuit\} &   &  
\\
\midrule
\kappa = 2\alpha_1+2\alpha_2 & f = 4(1-x_1)(1+2x_1)(2+x_1) & m_{\clubsuit} = 2, \quad m_{\varheartsuit} = 1
\end{array}
\]
\end{prop}

There are no locally factorial \(G/H\)-reflexive polytopes. 

\subsection{Dimension three, rank one horospherical}

The next proposition follows directly from Proposition~\ref{prop_parabolic_induction}. 

\begin{prop}
The combinatorial data of \(\SL_3\times \bbG_m/\ker(a_1\varpi_1+\chi_1)|_{Q_{\alpha_1}}\) is as follows. 
\[\arraycolsep=12pt
\begin{array}{lll}
M=\langle a_1\varpi_1+\chi_1 \rangle  &  \rho(\clubsuit)=a_1 & \zeta(\clubsuit)=\{\alpha_1\} 
\\  
\Sigma=\emptyset & 
\\ 
\mathcal{D} = \{\clubsuit\} &   &  
\\
\midrule
\kappa = 2\alpha_1+\alpha_2 & f = \frac{(3+a_1x_1)^2}{2} & m_{\clubsuit} = 3
\end{array}
\]
\end{prop}

We deduce easily the possible locally factorial \(G/H\)-reflexive polytopes and corresponding embeddings. 
It suffices to note that the conditions imply that \(a_1=1\) if \(\frac{\rho(\alpha_1)}{m_{\alpha_1}}=a_1/3\) is a vertex, and otherwise the fact that \(a_1/3\) is in the interior of the polytope shows that \(a_1\in \{0,1,2\}\). 
If \(a_1=0\) then the homogeneous space is the product \(\bbP^2\times \bbG_m\), yielding only one Fano embedding \(\bbP^2\times \bbP^1\), numbered as \textbf{3-1-10}. 

We obtain three non-product embeddings, corresponding to the following polytopes.

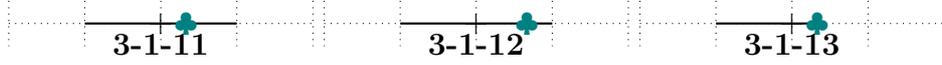
\begin{center}
\begin{tikzpicture}
\draw[dotted] (-2,0) -- (2,0);
\draw (0,0) node{+};
\foreach \i in {0,...,4}
{
\draw[dotted] (-2+\i,.3) -- (-2+\i,-.3);
}
\draw[thick] (-1,0) -- (1,0);
\draw (1/3,0) node[color=teal]{\(\clubsuit\)};
\draw (0,0) node[below]{\textbf{3-1-11}};
\end{tikzpicture}
\begin{tikzpicture}
\draw[dotted] (-2,0) -- (2,0);
\draw (0,0) node{+};
\foreach \i in {0,...,4}
{
\draw[dotted] (-2+\i,.3) -- (-2+\i,-.3);
}
\draw[thick] (-1,0) -- (1,0);
\draw (2/3,0) node[color=teal]{\(\clubsuit\)};
\draw (0,0) node[below]{\textbf{3-1-12}};
\end{tikzpicture}
\begin{tikzpicture}
\draw[dotted] (-2,0) -- (2,0);
\draw (0,0) node{+};
\foreach \i in {0,...,4}
{
\draw[dotted] (-2+\i,.3) -- (-2+\i,-.3);
}
\draw[thick] (-1,0) -- (1/3,0);
\draw (1/3,0) node[color=teal]{\(\clubsuit\)};
\draw (0,0) node[below]{\textbf{3-1-13}};
\end{tikzpicture}
\end{center}

\begin{exa}
The embeddings are easily identified as \(\bbP_{\bbP^2}(\mathcal{O}\oplus \mathcal{O}(k))\) for \(k\in \{1,2\}\), and \(\bbP^3\) as the Fano cone over the Kodaira embedding of \(\bbP^2\) in \(\bbP^3\) via \(\mathcal{O}(1)\). 
\end{exa}

\subsection{Dimension four, rank one horospherical}

\begin{prop}
The combinatorial data of \(\SL_3\times \bbG_m/\ker(a_1\varpi_1+a_2\varpi_2+\chi_1)|_{B^-}\) with \(a_i\in \bbZ\) and \(a_1\geq \lvert a_2\rvert\) is as follows. 
\[\arraycolsep=12pt
\begin{array}{lll}
M=\langle a_1\varpi_1+a_2\varpi_2+\chi_1 \rangle  &  \rho(\clubsuit)=a_1 & \zeta(\clubsuit)=\{\alpha_1\} 
\\  
\Sigma=\emptyset &   \rho(\varheartsuit)=a_2 & \zeta(\varheartsuit)=\{\alpha_2\} 
\\ 
\mathcal{D} = \{\clubsuit,\varheartsuit\} &   &  
\\
\midrule
\kappa = 2\alpha_1+2\alpha_2 & f = \frac{(2+a_1x_1)(2+a_2x_1)(4+(a_1+a_2)x_1)}{2} & m_{\clubsuit} = m_{\varheartsuit} = 2
\end{array}
\]
\end{prop}

We now determine the possible locally factorial \(G/H\)-reflexive polytopes. 
The conditions on \(\frac{\rho(\clubsuit)}{2}\) and \(\frac{\rho(\varheartsuit)}{2}\) show that, if one is a vertex, \(a_1=1\) and \(a_2\in \{0,1\}\), and if it is not, then \((a_1,a_2)\in \{(0,0),(1,0),(1,1),(1,-1)\}\). 
Furthermore, if \((a_1,a_2)=(0,0)\), the embedding \textbf{4-1-4} is the product \(W\times \bbP^1\), where \(W=\SL_3/B^-\) is the variety of full flags in \(\bbC^3\). 
We deduce that there are six polytopes corresponding to non-product embeddings. 

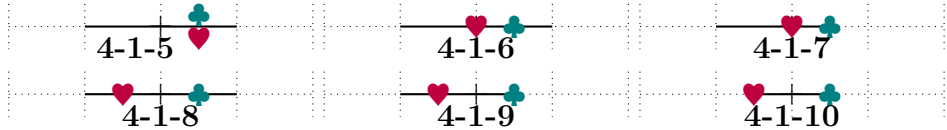
\begin{center}
\begin{tikzpicture}
\draw[dotted] (-2,0) -- (2,0);
\draw (0,0) node{+};
\foreach \i in {0,...,4}
{
\draw[dotted] (-2+\i,.3) -- (-2+\i,-.3);
}
\draw[thick] (-1,0) -- (1,0);
\draw (1/2,0) node[above=-4, color=teal]{\(\clubsuit\)};
\draw (1/2,0) node[below=-4, color=purple]{\(\varheartsuit\)};
\draw (-1/3,0) node[below]{\textbf{4-1-5}};
\end{tikzpicture}
\begin{tikzpicture}
\draw[dotted] (-2,0) -- (2,0);
\draw (0,0) node{+};
\foreach \i in {0,...,4}
{
\draw[dotted] (-2+\i,.3) -- (-2+\i,-.3);
}
\draw[thick] (-1,0) -- (1,0);
\draw (1/2,0) node[color=teal]{\(\clubsuit\)};
\draw (0,0) node[color=purple]{\(\varheartsuit\)};
\draw (0,0) node[below]{\textbf{4-1-6}};
\end{tikzpicture}
\begin{tikzpicture}
\draw[dotted] (-2,0) -- (2,0);
\draw (0,0) node{+};
\foreach \i in {0,...,4}
{
\draw[dotted] (-2+\i,.3) -- (-2+\i,-.3);
}
\draw[thick] (-1,0) -- (1/2,0);
\draw (1/2,0) node[color=teal]{\(\clubsuit\)};
\draw (0,0) node[color=purple]{\(\varheartsuit\)};
\draw (0,0) node[below]{\textbf{4-1-7}};
\end{tikzpicture}
\begin{tikzpicture}
\draw[dotted] (-2,0) -- (2,0);
\draw (0,0) node{+};
\foreach \i in {0,...,4}
{
\draw[dotted] (-2+\i,.3) -- (-2+\i,-.3);
}
\draw[thick] (-1,0) -- (1,0);
\draw (1/2,0) node[color=teal]{\(\clubsuit\)};
\draw (-1/2,0) node[color=purple]{\(\varheartsuit\)};
\draw (0,0) node[below]{\textbf{4-1-8}};
\end{tikzpicture}
\begin{tikzpicture}
\draw[dotted] (-2,0) -- (2,0);
\draw (0,0) node{+};
\foreach \i in {0,...,4}
{
\draw[dotted] (-2+\i,.3) -- (-2+\i,-.3);
}
\draw[thick] (-1,0) -- (1/2,0);
\draw (1/2,0) node[color=teal]{\(\clubsuit\)};
\draw (-1/2,0) node[color=purple]{\(\varheartsuit\)};
\draw (0,0) node[below]{\textbf{4-1-9}};
\end{tikzpicture}
\begin{tikzpicture}
\draw[dotted] (-2,0) -- (2,0);
\draw (0,0) node{+};
\foreach \i in {0,...,4}
{
\draw[dotted] (-2+\i,.3) -- (-2+\i,-.3);
}
\draw[thick] (-1/2,0) -- (1/2,0);
\draw (1/2,0) node[color=teal]{\(\clubsuit\)};
\draw (-1/2,0) node[color=purple]{\(\varheartsuit\)};
\draw (0,0) node[below]{\textbf{4-1-10}};
\end{tikzpicture}
\end{center}

\begin{exa}
The embeddings corresponding to \textbf{4-1-5}, \textbf{4-1-6} and \textbf{4-1-8}  are \(\bbP^1\)-bundles over \(W=\SL_3/B^-\) of the form \(\bbP(\mathcal{O}\oplus\mathcal{O}(a_1\varpi_1+a_2\varpi_2))\) where \(\mathcal{O}(a_1\varpi_1+a_2\varpi_2)\) is the homogeneous line bundle associated to the weight \(a_1\varpi_1+a_2\varpi_2\). 
\end{exa}

\begin{exa}
The geometrical description for the embedding associated with \textbf{4-1-10} follows from \cite[Proposition~1.8]{Pasquier_2009}: consider \(\bbP^5\) as the projectivization of the direct sum of the representations \(\bbC^3\) and \((\bbC^3)^*\) of \(\SL_3\). 
Consider furthermore the quadratic form on \(\bbC^3\oplus (\bbC^3)^*\) given by the duality bracket. 
Then the embedding is the corresponding quadric \(Q^4\subset \bbP^5\). 
Furthermore, as follows from the encoding of morphisms between spherical varieties, the variety corresponding to the penultimate polytope is the blowup of \(Q^4\) along the subvariety \(\bbP^2=\bbP(\{0\}\oplus (\bbC^3)^*)\). 
Note that we may also recover \textbf{4-1-8} as the blowup of \(Q^4\) along the two copies of \(\bbP^2\) given by each subrepresentation of \(\bbC^6=\bbC^3\oplus (\bbC^3)^*\). 
\end{exa}


\subsection{Dimension four, rank two horospherical}

In this case, we obtain upgradings of toric structures. 

\subsubsection{Combinatorial data}

\begin{prop}
Let \(H=\ker(a_1\varpi_1+\chi_1)\cap \ker(\chi_2) \subset Q_{\alpha_1}\subset G=\SL_3\times \bbG_m^2\). 
The combinatorial data of \(G/H\) is as follows.
\[\arraycolsep=12pt
\begin{array}{lll}
M=\langle a_1\varpi_1+\chi_1,\chi_2 \rangle  &  \rho(\clubsuit)=a_1 & \zeta(\clubsuit)=\{\alpha_1\} 
\\  
\Sigma=\emptyset &  
\\ 
\mathcal{D} = \{\clubsuit\} &   &  
\\
\midrule
\kappa = 2\alpha_1+\alpha_2 & f = \frac{(3+a_1x_1)^2}{2} & m_{\clubsuit} = 3
\end{array}
\]
\end{prop}

\begin{proof}
This follows directly from Proposition~\ref{prop_parabolic_induction}.
\end{proof}

\subsubsection{Polytopes}

When \(a_1=0\), the homogeneous space is the product of homogeneous spaces \(\bbP^2\times \bbG_m^2\). 
All locally factorial embeddings are thus products of \(\bbP^2\) with toric Fano surfaces, there are 5 such embeddings, that we number from \textbf{4-2-155} to \textbf{4-2-159}.

We now assume that \(a_1>0\). 
Recall that by Remark~\ref{rem_SL3_non-uniqueness_of_chi}, we can work up to the action of \(\begin{bmatrix} 1 & k \\ 0& \pm 1\end{bmatrix}\) for \(k\in \bbZ\). 
If \(\frac{\rho(\clubsuit)}{3}=(\frac{a_1}{3},0)\) is a vertex, then \(\rho(\alpha_1)\) must be primitive, hence \(a_1=1\). 
If it is not a vertex, it must be an interior point, hence we must have \(a_i\in \{1,2\}\). 

If \((\frac{a_1}{3},0)\) is not a vertex, then it suffices to go through the list of smooth Fano polytopes up to \(\GL_2(\bbZ)\), and find the representatives that work up to the action of  \(\begin{bmatrix} 1 & k \\ 0& \pm 1\end{bmatrix}\). 
There are nine possibilities for \(a_1=2\) and 17 for \(a_1=1\), given by the following possible locations for \(\frac{\rho(\clubsuit)}{3}\). 
To draw the polytopes, we change conventions a bit: we print only the lattice \(\frac{1}{3}\bbZ^2\), modify freely the representation by a unimodular representation for conciseness, and display the possible locations for \(\frac{\rho(\clubsuit)}{3}\) by the number \textbf{I} such that the corresponding embedding is labelled \textbf{4-2-I}. 

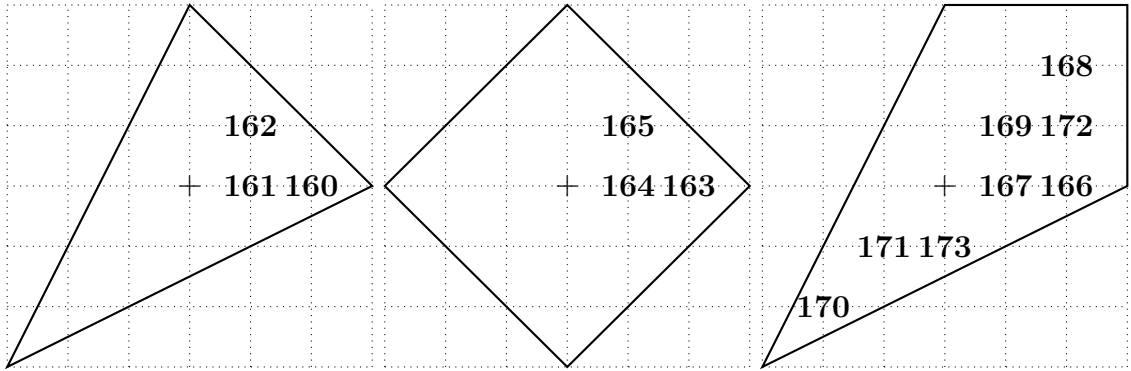
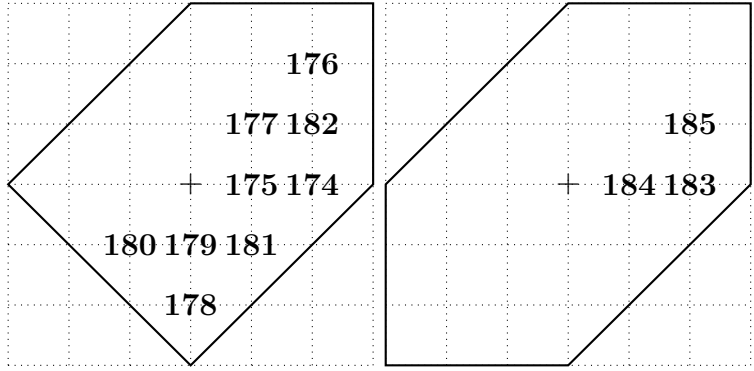
\begin{center}
\begin{tikzpicture}[scale=0.8]
\draw[dotted] (-3,3) grid (3,-3);
\draw (0,0) node{+};
\draw[thick] (3,0) -- (0,3) -- (-3,-3) -- cycle;
\draw (2,0) node{\textbf{160}};
\draw (1,0) node{\textbf{161}};
\draw (1,1) node{\textbf{162}};
\end{tikzpicture}
\begin{tikzpicture}[scale=0.8]
\draw[dotted] (-3,3) grid (3,-3);
\draw (0,0) node{+};
\draw[thick] (3,0) -- (0,3) -- (-3,0) -- (0,-3) -- cycle;
\draw (2,0) node{\textbf{163}};
\draw (1,0) node{\textbf{164}};
\draw (1,1) node{\textbf{165}};
\end{tikzpicture}
\begin{tikzpicture}[scale=0.8]
\draw[dotted] (-3,3) grid (3,-3);
\draw (0,0) node{+};
\draw[thick] (3,0) -- (3,3) -- (0,3) -- (-3,-3) -- cycle;
\draw (2,0) node{\textbf{166}};
\draw (1,0) node{\textbf{167}};
\draw (2,2) node{\textbf{168}};
\draw (1,1) node{\textbf{169}};
\draw (-2,-2) node{\textbf{170}};
\draw (-1,-1) node{\textbf{171}};
\draw (2,1) node{\textbf{172}};
\draw (0,-1) node{\textbf{173}};
\end{tikzpicture}
\end{center}
\begin{center}
\begin{tikzpicture}[scale=0.8]
\draw[dotted] (-3,3) grid (3,-3);
\draw (0,0) node{+};
\draw[thick] (3,0) -- (3,3) -- (0,3) -- (-3,0) -- (0,-3) -- cycle;
\draw (2,0) node{\textbf{174}};
\draw (1,0) node{\textbf{175}};
\draw (2,2) node{\textbf{176}};
\draw (1,1) node{\textbf{177}};
\draw (0,-2) node{\textbf{178}};
\draw (0,-1) node{\textbf{179}};
\draw (-1,-1) node{\textbf{180}};
\draw (1,-1) node{\textbf{181}};
\draw (2,1) node{\textbf{182}};
\end{tikzpicture}
\begin{tikzpicture}[scale=0.8]
\draw[dotted] (-3,3) grid (3,-3);
\draw (0,0) node{+};
\draw[thick] (3,0) -- (3,3) -- (0,3) -- (-3,0) -- (-3,-3) -- (0,-3) -- cycle;
\draw (2,0) node{\textbf{183}};
\draw (1,0) node{\textbf{184}};
\draw (2,1) node{\textbf{185}};
\end{tikzpicture}
\end{center}

Let us now deal with the case when \((\frac{1}{3},0)\) is a vertex. 
By considering the possible successive next vertices, and working up to \((x,y)\mapsto (x,-y)\) which induces a \(G\)-equivariant isomorphism, we obtain the following nine polytopes. 

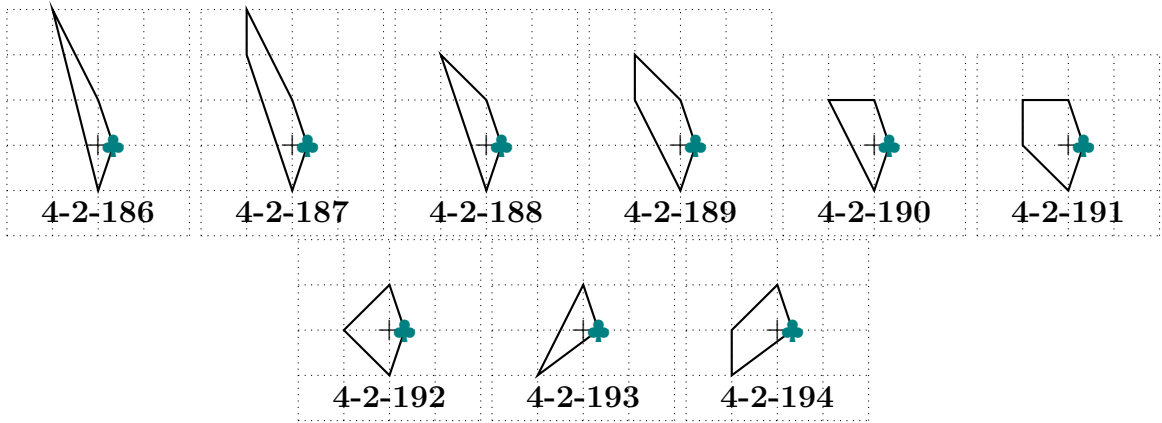
\begin{center}
\begin{tikzpicture}[scale=0.6]
\draw[dotted] (-2,3) grid (2,-2);
\draw (0,0) node{+};
\draw[thick] (1/3,0) -- (0,1) -- (-1,3) -- (0,-1) -- cycle;
\draw (1/3,0) node[color=teal]{\(\clubsuit\)};
\draw (0,-1) node[below]{\textbf{4-2-186}};
\end{tikzpicture}
\begin{tikzpicture}[scale=0.6]
\draw[dotted] (-2,3) grid (2,-2);
\draw (0,0) node{+};
\draw[thick] (1/3,0) -- (0,1) -- (-1,3) -- (-1,2) -- (0,-1) -- cycle;
\draw (1/3,0) node[color=teal]{\(\clubsuit\)};
\draw (0,-1) node[below]{\textbf{4-2-187}};
\end{tikzpicture}
\begin{tikzpicture}[scale=0.6]
\draw[dotted] (-2,3) grid (2,-2);
\draw (0,0) node{+};
\draw[thick] (1/3,0) -- (0,1) -- (-1,2) -- (0,-1) -- cycle;
\draw (1/3,0) node[color=teal]{\(\clubsuit\)};
\draw (0,-1) node[below]{\textbf{4-2-188}};
\end{tikzpicture}
\begin{tikzpicture}[scale=0.6]
\draw[dotted] (-2,3) grid (2,-2);
\draw (0,0) node{+};
\draw[thick] (1/3,0) -- (0,1) -- (-1,2)  -- (-1,1) -- (0,-1) -- cycle;
\draw (1/3,0) node[color=teal]{\(\clubsuit\)};
\draw (0,-1) node[below]{\textbf{4-2-189}};
\end{tikzpicture}
\begin{tikzpicture}[scale=0.6]
\draw[dotted] (-2,2) grid (2,-2);
\draw (0,0) node{+};
\draw[thick] (1/3,0) -- (0,1) -- (-1,1) -- (0,-1) -- cycle;
\draw (1/3,0) node[color=teal]{\(\clubsuit\)};
\draw (0,-1) node[below]{\textbf{4-2-190}};
\end{tikzpicture}
\begin{tikzpicture}[scale=0.6]
\draw[dotted] (-2,2) grid (2,-2);
\draw (0,0) node{+};
\draw[thick] (1/3,0) -- (0,1) -- (-1,1) -- (-1,0) -- (0,-1) -- cycle;
\draw (1/3,0) node[color=teal]{\(\clubsuit\)};
\draw (0,-1) node[below]{\textbf{4-2-191}};
\end{tikzpicture}
\begin{tikzpicture}[scale=0.6]
\draw[dotted] (-2,2) grid (2,-2);
\draw (0,0) node{+};
\draw[thick] (1/3,0) -- (0,1) -- (-1,0) -- (0,-1) -- cycle;
\draw (1/3,0) node[color=teal]{\(\clubsuit\)};
\draw (0,-1) node[below]{\textbf{4-2-192}};
\end{tikzpicture}
\begin{tikzpicture}[scale=0.6]
\draw[dotted] (-2,2) grid (2,-2);
\draw (0,0) node{+};
\draw[thick] (1/3,0) -- (0,1) -- (-1,-1) -- cycle;
\draw (1/3,0) node[color=teal]{\(\clubsuit\)};
\draw (0,-1) node[below]{\textbf{4-2-193}};
\end{tikzpicture}
\begin{tikzpicture}[scale=0.6]
\draw[dotted] (-2,2) grid (2,-2);
\draw (0,0) node{+};
\draw[thick] (1/3,0) -- (0,1) -- (-1,0) -- (-1,-1) -- cycle;
\draw (1/3,0) node[color=teal]{\(\clubsuit\)};
\draw (0,-1) node[below]{\textbf{4-2-194}};
\end{tikzpicture}
\end{center}

\begin{exa}
The embedding \textbf{4-2-193} is \(\bbP^4\), and the action is induced by the standard action of \(\GL_4\) on an affine chart \(\bbC^4\), and the embedding of \(\SL_3\times \bbG_m\) in \(\GL_4\) by block-diagonal matrices. 
Blowing up one of the fixed points yields the embedding \textbf{4-2-190}, while blowing up the \(\bbP^2\) at infinity yields \textbf{4-2-194}. 
Blowing up \(\{0\}\times \bbC\) on the other hand yields the embedding \textbf{4-2-160}. 
Several examples can be obtained by further explicit blowups of either one of these, or of \textbf{4-2-192}, which is the product of embeddings \(\bbP^3\times \bbP^1\). 
\end{exa}

\section{Remaining rank one examples}
\label{section_remaining}

\subsection{Under \(\Sp_4\)}

The combinatorial data of rank one spherical homogeneous spaces are well-known, and by \cite{Akhiezer_1983}, we know that for \(\Sp_4/N(\SL_2\times \Sp_2)\) there exists a unique complete embedding, which is \(\bbP^4\) and that we label \textbf{4-1-11}, while \(Q^4\) is the unique complete embedding of \(\Sp_4/\SL_2\times \Sp_2\), labelled as \textbf{4-1-12}. 
We recall the combinatorial data, and the details on the embeddings, which allows to recover the combinatorial data easily. 

\begin{prop}
For \(H=N(\SL_2\times \Sp_2)\), the combinatorial data is as follows. 
\[\arraycolsep=12pt
\begin{array}{lll}
M=\langle 2(\alpha_1+\alpha_2) \rangle  &  \rho(\clubsuit)=2 & \zeta(\clubsuit)=\{\alpha_2\} 
\\  
\Sigma=\{2(\alpha_1+\alpha_2)\} &  
\\ 
\mathcal{D} = \{\clubsuit\} &   &  
\\
\midrule
\kappa = 3\alpha_1+3\alpha_2 & f = \frac{(3+2x_1)^3}{3} & m_{\clubsuit} = 3
\end{array}
\]
\end{prop}

\begin{proof}
Consider the natural linear action of \(\SO_5\) on \(\bbC^5\), where \(\SO_5\) is realized in \(\SL_5\) as the group of matrices \(A\) such that \(A^{T}KAK=I_5\), where 
\[ K = \begin{bmatrix} 0 &0 &0&0&1\\0&0&0&1&0\\0&0&1&0&0\\0&1&0&0&0\\1&0&0&0&0\end{bmatrix} \]
There are two orbits of \(\SO_5\) on the induced projective space \(\bbP^4\): the three dimensional quadric \(Q^3=\{x_1x_5+x_2x_4+x_3^2=0\}=\SO_5/Q_{\alpha_2}\) and its complement. 
Furthermore, the stabilizer of a point in the complement is isomorphic to \(S(O(1)\times O(4))\). 
Note that the action of \(\Sp_4\) on \(\Sp_4/N(\SL_2\times \Sp_2)\) factors through \(\SO_5\), and that the image of the isotropy group is conjugate to \(S(O(1)\times O(4))\), so we indeed have a model of the desired homogeneous space. 

The subgroup of upper triangular matrices contained in \(\SO_5\) is a Borel subgroup of \(\SO_5\), and under this action there is only one codimension one orbit in \(\bbP^4\setminus Q^3\), hence only one color \(D_1\). 
The equation of the color is given by \(\{x_5=0\}\), this is a hyperplane in \(\bbP^4\). 
We have thus determined \(\mathcal{D}\), and the description of \(\zeta\) follows easily. 
From \(\zeta\), we have the adapted parabolic, which is \(Q_{\alpha_2}\). 
From that, we recover that the rank of the homogeneous space is one, that \(\kappa=3\alpha_1+3\alpha_2\) and that \(m_{D_1}=3\). 

Consider the rational function defined by \(\frac{x_5^2}{x_1x_5+x_2x_4+x_3^2}\), it is \(B\)-semi-invariant with weight \(2(\alpha_1+\alpha_2)\). It vanishes to order \(2\) on \(D_1\), and to order \(-1\) on the \(\SO_5\) stable prime divisor \(Q^3\). 
Hence it is a generator of \(M\) and we deduce \(\rho\). 
\end{proof}

For \(\Sp_4/\SL_2\times \Sp_2\simeq \SO_5/\SO_4\), the projective model is given by compactifying the affine quadric \(\{x_1x_5+x_2x_4+x_3^2=1\}\subset \bbC^5\) inside \(\bbP^5\). 
The result is the smooth projective quadric \(Q^4\), on which \(\SO_5\) acts with two orbits: the three dimensional quadric in the hyperplane at infinity, and the affine quadric. 
The same arguments as in the above proof give the combinatorial data. 

\begin{prop}
For \(H=\SL_2\times \Sp_2\), the combinatorial data is as follows. 
\[\arraycolsep=12pt
\begin{array}{lll}
M=\langle \alpha_1+\alpha_2 \rangle  &  \rho(\clubsuit)=1 & \zeta(\clubsuit)=\{\alpha_2\} 
\\  
\Sigma=\{\alpha_1+\alpha_2\} &  
\\ 
\mathcal{D} = \{\clubsuit\} &   &  
\\
\midrule
\kappa = 3\alpha_1+3\alpha_2 & f = \frac{(3+x_1)^3}{3} & m_{\clubsuit} = 3
\end{array}
\]
\end{prop}

\subsection{Under \(\SL_3\times \SL_2\)}

Since there is a parabolic subgroup of \(\SL_3\) as a factor, we only obtain products of lower dimensional spherical embeddings, namely \(\bbP^2\times \bbP^1\times \bbP^1\), labelled as \textbf{4-1-13}, and \(\bbP^2\times \bbP^2\) labelled as \textbf{4-1-14}.  

\subsection{Under \(\SL_2^3\)}

Again, since there is at least one parabolic subgroup of \(\SL_2\) as a factor, we only obtain products of lower dimensional spherical embeddings, namely \textbf{4-1-15} is \(\bbP^1\times Q^3\), \textbf{4-1-16} is \(\bbP^1\times \bbP^3\), \textbf{4-1-17} is \(\bbP^1\times \bbP^1\times (\bbP^1\times \bbP^1)\), and \textbf{4-1-18} is \(\bbP^1\times \bbP^1\times \bbP^2\). 

\subsection{Under \(\SL_2^3\times \bbG_m\)}

\begin{prop}
Let \(H\) be a rank one horospherical subgroup of \(G=(\SL_2)^3 \times \bbG_m\), with normalizer the Borel subgroup \(B^-\), satisfying Assumptions~\ref{assumptions}. 
Then \(H=\ker(a_1\varpi_1+a_2\varpi_2+a_3\varpi_3+\chi_1)\) for some integers \(a_1\geq \lvert a_2\rvert \geq \lvert a_3\rvert\), where \(\chi_1\) is the projection \(G\to \bbG_m\) to the last factor. 
The combinatorial data of \(G/H\) is:
\[\arraycolsep=12pt
\begin{array}{lll}
M=\langle a_1\varpi_1+a_2\varpi_2+a_3\varpi_3+\chi_1\rangle  & \zeta(\clubsuit) =  \{\alpha_1\}  & \rho(\clubsuit) = a_1 \\
\Sigma = \emptyset &  \zeta(\varheartsuit) = \{\alpha_2\} & \rho(\varheartsuit) = a_2 \\
\mathcal{D} =\{\clubsuit,\varheartsuit,\spadesuit\} & \zeta(\spadesuit) =  \{\alpha_3\} & \rho(\spadesuit) = a_3 \\
\midrule
f=(2+a_1x_1)(2+a_2x_1)(2+a_3x_1) & \kappa = \alpha_1+\alpha_2+\alpha_3 & m_{\clubsuit} = m_{\varheartsuit} = m_{\spadesuit} = 2 \\
\end{array}
\]
\end{prop}

Again, if one color is sent to the origin, then the embeddings are products of embeddings, there are seven such products, labelled \textbf{4-1-19} to \textbf{4-1-25}. 
The non-product examples are given by the following three polytopes 

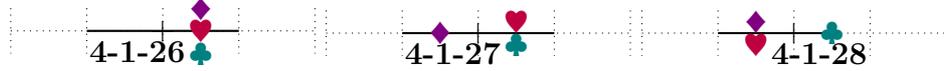
\begin{center}
\begin{tikzpicture}
\draw[dotted] (-2,0) -- (2,0);
\draw (0,0) node{+};
\foreach \i in {0,...,4}
{
\draw[dotted] (-2+\i,.3) -- (-2+\i,-.3);
}
\draw[thick] (-1,0) -- (1,0);
\draw (1/2,0) node[below, color=teal]{\(\clubsuit\)};
\draw (1/2,0) node[color=purple]{\(\varheartsuit\)};
\draw (1/2,0) node[above, color=violet]{\(\vardiamondsuit\)};
\draw (-1/3,0) node[below]{\textbf{4-1-26}};
\end{tikzpicture}
\begin{tikzpicture}
\draw[dotted] (-2,0) -- (2,0);
\draw (0,0) node{+};
\foreach \i in {0,...,4}
{
\draw[dotted] (-2+\i,.3) -- (-2+\i,-.3);
}
\draw[thick] (-1,0) -- (1,0);
\draw (1/2,0) node[below=-4, color=teal]{\(\clubsuit\)};
\draw (1/2,0) node[above=-4, color=purple]{\(\varheartsuit\)};
\draw (-1/2,0) node[color=violet]{\(\vardiamondsuit\)};
\draw (-1/3,0) node[below]{\textbf{4-1-27}};
\end{tikzpicture}
\begin{tikzpicture}
\draw[dotted] (-2,0) -- (2,0);
\draw (0,0) node{+};
\foreach \i in {0,...,4}
{
\draw[dotted] (-2+\i,.3) -- (-2+\i,-.3);
}
\draw[thick] (-1,0) -- (1/2,0);
\draw (1/2,0) node[color=teal]{\(\clubsuit\)};
\draw (-1/2,0) node[below=-4, color=purple]{\(\varheartsuit\)};
\draw (-1/2,0) node[above=-4, color=violet]{\(\vardiamondsuit\)};
\draw (1/3,0) node[below]{\textbf{4-1-28}};
\end{tikzpicture}
\end{center}

\begin{exa}
The first and second correspond to \(\bbP_{\bbP^1\times\bbP^1\times \bbP^1}(\mathcal{O}\oplus \mathcal{O}(1,1,1))\) and \(\bbP_{\bbP^1\times\bbP^1\times \bbP^1}(\mathcal{O}\oplus \mathcal{O}(1,1,-1))\). \end{exa}

\subsection{Under \(\SL_3\times \SL_2\times \bbG_m\)}

\begin{prop}
Let \(H\) be a rank one horospherical subgroup of \(G=\SL_3\times \SL_2 \times \bbG_m\), with normalizer the parabolic subgroup \(Q_{{\alpha_1,\alpha_3}}\), satisfying Assumptions~\ref{assumptions}. 
Then \(H=\ker(a_1\varpi_1+a_3\varpi_3+\chi_1)\) for some non-negative integer \(a_1\), some integer \(a_3\), where \(\chi_1\) is the projection \(G\to \bbG_m\) to the third factor. 
The combinatorial data of \(G/H\) is:
\[\arraycolsep=12pt
\begin{array}{lll}
M=\langle a_1\varpi_1+a_3\varpi_3+\chi_1\rangle  & \zeta(\clubsuit) =  \{\alpha_1\}  & \rho(\clubsuit) = a_1 \\
\Sigma = \emptyset &  \zeta(\varheartsuit) = \{\alpha_3\} & \rho(\varheartsuit) = a_3 \\
\mathcal{D} =\{\clubsuit,\varheartsuit\} &  \\
\midrule
\kappa = 2\alpha_1+\alpha_2+\alpha_3 & f=\frac{(3+a_1x_1)^2(2+a_3x_1)}{2} & m_{\clubsuit} = 3, \quad  m_{\varheartsuit} =  2 \\
\end{array}
\]
\end{prop}

We easily deduce the possible locally factorial \(G/H\)-reflexive polytope. 
If one of the colors is sent to the origin, then the embedding is a product of embeddings, there are six such products, labelled \textbf{4-1-29} to \textbf{4-1-34}. 
The non-product examples are given by the following nine polytopes. 

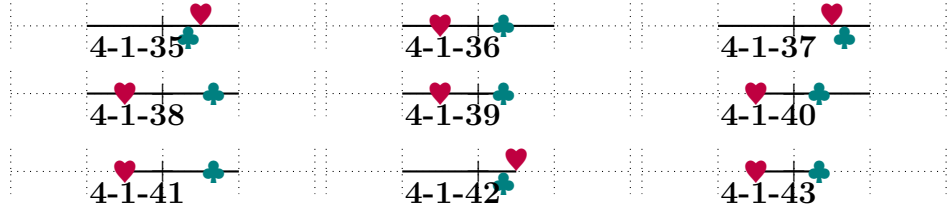
\begin{center}
\begin{tikzpicture}
\draw[dotted] (-2,0) -- (2,0);
\draw (0,0) node{+};
\foreach \i in {0,...,4}
{
\draw[dotted] (-2+\i,.3) -- (-2+\i,-.3);
}
\draw[thick] (-1,0) -- (1,0);
\draw (1/3,0) node[below=-4, color=teal]{\(\clubsuit\)};
\draw (1/2,0) node[above=-4, color=purple]{\(\varheartsuit\)};
\draw (-1/3,0) node[below]{\textbf{4-1-35}};
\end{tikzpicture}
\begin{tikzpicture}
\draw[dotted] (-2,0) -- (2,0);
\draw (0,0) node{+};
\foreach \i in {0,...,4}
{
\draw[dotted] (-2+\i,.3) -- (-2+\i,-.3);
}
\draw[thick] (-1,0) -- (1,0);
\draw (1/3,0) node[color=teal]{\(\clubsuit\)};
\draw (-1/2,0) node[color=purple]{\(\varheartsuit\)};
\draw (-1/3,0) node[below]{\textbf{4-1-36}};
\end{tikzpicture}
\begin{tikzpicture}
\draw[dotted] (-2,0) -- (2,0);
\draw (0,0) node{+};
\foreach \i in {0,...,4}
{
\draw[dotted] (-2+\i,.3) -- (-2+\i,-.3);
}
\draw[thick] (-1,0) -- (1,0);
\draw (2/3,0) node[below=-4, color=teal]{\(\clubsuit\)};
\draw (1/2,0) node[above=-4, color=purple]{\(\varheartsuit\)};
\draw (-1/3,0) node[below]{\textbf{4-1-37}};
\end{tikzpicture}
\begin{tikzpicture}
\draw[dotted] (-2,0) -- (2,0);
\draw (0,0) node{+};
\foreach \i in {0,...,4}
{
\draw[dotted] (-2+\i,.3) -- (-2+\i,-.3);
}
\draw[thick] (-1,0) -- (1,0);
\draw (2/3,0) node[color=teal]{\(\clubsuit\)};
\draw (-1/2,0) node[color=purple]{\(\varheartsuit\)};
\draw (-1/3,0) node[below]{\textbf{4-1-38}};
\end{tikzpicture}
\begin{tikzpicture}
\draw[dotted] (-2,0) -- (2,0);
\draw (0,0) node{+};
\foreach \i in {0,...,4}
{
\draw[dotted] (-2+\i,.3) -- (-2+\i,-.3);
}
\draw[thick] (-1,0) -- (1/3,0);
\draw (1/3,0) node[color=teal]{\(\clubsuit\)};
\draw (-1/2,0) node[color=purple]{\(\varheartsuit\)};
\draw (-1/3,0) node[below]{\textbf{4-1-39}};
\end{tikzpicture}
\begin{tikzpicture}
\draw[dotted] (-2,0) -- (2,0);
\draw (0,0) node{+};
\foreach \i in {0,...,4}
{
\draw[dotted] (-2+\i,.3) -- (-2+\i,-.3);
}
\draw[thick] (-1/2,0) -- (1,0);
\draw (1/3,0) node[color=teal]{\(\clubsuit\)};
\draw (-1/2,0) node[color=purple]{\(\varheartsuit\)};
\draw (-1/3,0) node[below]{\textbf{4-1-40}};
\end{tikzpicture}
\begin{tikzpicture}
\draw[dotted] (-2,0) -- (2,0);
\draw (0,0) node{+};
\foreach \i in {0,...,4}
{
\draw[dotted] (-2+\i,.3) -- (-2+\i,-.3);
}
\draw[thick] (-1/2,0) -- (1,0);
\draw (2/3,0) node[color=teal]{\(\clubsuit\)};
\draw (-1/2,0) node[color=purple]{\(\varheartsuit\)};
\draw (-1/3,0) node[below]{\textbf{4-1-41}};
\end{tikzpicture}
\begin{tikzpicture}
\draw[dotted] (-2,0) -- (2,0);
\draw (0,0) node{+};
\foreach \i in {0,...,4}
{
\draw[dotted] (-2+\i,.3) -- (-2+\i,-.3);
}
\draw[thick] (-1,0) -- (1/2,0);
\draw (1/3,0) node[below=-4, color=teal]{\(\clubsuit\)};
\draw (1/2,0) node[above=-4, color=purple]{\(\varheartsuit\)};
\draw (-1/3,0) node[below]{\textbf{4-1-42}};
\end{tikzpicture}
\begin{tikzpicture}
\draw[dotted] (-2,0) -- (2,0);
\draw (0,0) node{+};
\foreach \i in {0,...,4}
{
\draw[dotted] (-2+\i,.3) -- (-2+\i,-.3);
}
\draw[thick] (-1/2,0) -- (1/3,0);
\draw (1/3,0) node[color=teal]{\(\clubsuit\)};
\draw (-1/2,0) node[color=purple]{\(\varheartsuit\)};
\draw (-1/3,0) node[below]{\textbf{4-1-43}};
\end{tikzpicture}
\end{center}

\begin{exa}
The corresponding toroidal embeddings are the \(\bbP_{\bbP^1\times \bbP^2}(\mathcal{O}\oplus \mathcal{O}(k,l))\) for \(1\leq k\leq 3\) and \(l=\pm 1\).  
The last embedding is \(\bbP^4\) equipped with the action of \(S(\GL_2\times\GL_3)\subset \SL_5\), which may be blown up at either of the two closed orbits to recover two other examples. 
\end{exa}

\subsection{Under \(\Sp_4\times \bbG_m\)}

The result below, as next statements in the remaining of the section, follows directly from Proposition~\ref{prop_parabolic_induction}.
\begin{prop}
Let \(H=\ker(a_i\varpi_i+\chi_1)|_{Q_{\alpha_i}}\subset G=\Sp_4\times \bbG_m\) for some \(i\in \{1,2\}\) and \(a_i\geq 0\). 
The combinatorial data of \(G/H\) when \(i=1\) is as follows. 
\[\arraycolsep=12pt
\begin{array}{lll}
M=\langle a_1\varpi_1+\chi_1\rangle  & \zeta(\clubsuit) =  \{\alpha_1\}  & \rho(\clubsuit) = a_1 \\
\Sigma = \emptyset &   \\
\mathcal{D} =\{\clubsuit\} &  \\
\midrule
\kappa = 4\alpha_1+2\alpha_2 & f=\frac{(4+a_1x_1)^3}{6} & m_{\clubsuit} = 4 \\
\end{array}
\]
The combinatorial data of \(G/H\) when \(i=2\) is as follows. 
\[\arraycolsep=12pt
\begin{array}{lll}
M=\langle a_2\varpi_2+\chi_1\rangle  & \zeta(\clubsuit) =  \{\alpha_2\}  & \rho(\clubsuit) = a_2 \\
\Sigma = \emptyset &   \\
\mathcal{D} =\{\clubsuit\} &  \\
\midrule
\kappa = 3\alpha_1+3\alpha_2 & f=\frac{(3+a_2x_1)^3}{3} & m_{\clubsuit} = 3 \\
\end{array}
\]
\end{prop}

The locally factorial \(G/H\)-reflexive polytopes are easily derived. 
Note that they are products if \(a_i=0\): \textbf{4-1-44} is \(\bbP^3\times \bbP^1\) and \textbf{4-1-45} is \(Q^3\times \bbP^1\). 
The non-product possibilities are, when \(i=1\), the following four polytopes.

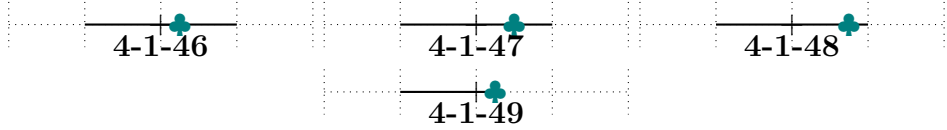
\begin{center}
\begin{tikzpicture}
\draw[dotted] (-2,0) -- (2,0);
\draw (0,0) node{+};
\foreach \i in {0,...,4}
{
\draw[dotted] (-2+\i,.3) -- (-2+\i,-.3);
}
\draw[thick] (-1,0) -- (1,0);
\draw (1/4,0) node[color=teal]{\(\clubsuit\)};
\draw (0,0) node[below]{\textbf{4-1-46}};
\end{tikzpicture}
\begin{tikzpicture}
\draw[dotted] (-2,0) -- (2,0);
\draw (0,0) node{+};
\foreach \i in {0,...,4}
{
\draw[dotted] (-2+\i,.3) -- (-2+\i,-.3);
}
\draw[thick] (-1,0) -- (1,0);
\draw (2/4,0) node[color=teal]{\(\clubsuit\)};
\draw (0,0) node[below]{\textbf{4-1-47}};
\end{tikzpicture}
\begin{tikzpicture}
\draw[dotted] (-2,0) -- (2,0);
\draw (0,0) node{+};
\foreach \i in {0,...,4}
{
\draw[dotted] (-2+\i,.3) -- (-2+\i,-.3);
}
\draw[thick] (-1,0) -- (1,0);
\draw (3/4,0) node[color=teal]{\(\clubsuit\)};
\draw (0,0) node[below]{\textbf{4-1-48}};
\end{tikzpicture}
\begin{tikzpicture}
\draw[dotted] (-2,0) -- (2,0);
\draw (0,0) node{+};
\foreach \i in {0,...,4}
{
\draw[dotted] (-2+\i,.3) -- (-2+\i,-.3);
}
\draw[thick] (-1,0) -- (1/4,0);
\draw (1/4,0) node[color=teal]{\(\clubsuit\)};
\draw (0,0) node[below]{\textbf{4-1-49}};
\end{tikzpicture}
\end{center}

\begin{exa}
The corresponding embeddings are the \(\bbP_{\bbP^3}(\mathcal{O}\oplus \mathcal{O}(k))\) for \(k\in \{1,2,3\}\), and \(\bbP^4\) as the cone over \(\bbP^3\) given by the Kodaira embedding with respect to \(\mathcal{O}(1)\). 
\end{exa}

When \(i=2\), the non-product possibilities are the following three polytopes.

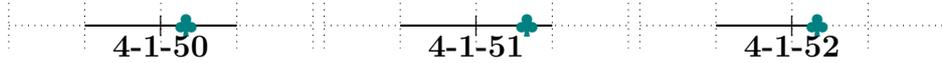
\begin{center}
\begin{tikzpicture}
\draw[dotted] (-2,0) -- (2,0);
\draw (0,0) node{+};
\foreach \i in {0,...,4}
{
\draw[dotted] (-2+\i,.3) -- (-2+\i,-.3);
}
\draw[thick] (-1,0) -- (1,0);
\draw (1/3,0) node[color=teal]{\(\clubsuit\)};
\draw (0,0) node[below]{\textbf{4-1-50}};
\end{tikzpicture}
\begin{tikzpicture}
\draw[dotted] (-2,0) -- (2,0);
\draw (0,0) node{+};
\foreach \i in {0,...,4}
{
\draw[dotted] (-2+\i,.3) -- (-2+\i,-.3);
}
\draw[thick] (-1,0) -- (1,0);
\draw (2/3,0) node[color=teal]{\(\clubsuit\)};
\draw (0,0) node[below]{\textbf{4-1-51}};
\end{tikzpicture}
\begin{tikzpicture}
\draw[dotted] (-2,0) -- (2,0);
\draw (0,0) node{+};
\foreach \i in {0,...,4}
{
\draw[dotted] (-2+\i,.3) -- (-2+\i,-.3);
}
\draw[thick] (-1,0) -- (1/3,0);
\draw (1/3,0) node[color=teal]{\(\clubsuit\)};
\draw (0,0) node[below]{\textbf{4-1-52}};
\end{tikzpicture}
\end{center}

\begin{exa}
The corresponding embeddings are the \(\bbP_{Q^3}(\mathcal{O}\oplus \mathcal{O}(k))\) for \(k\in \{1,2\}\), and the cone over \(Q^3\) given by the Kodaira embedding with respect to \(\mathcal{O}(1)\), which is a singular quadric in \(\bbP^5\). 
\end{exa}

\subsection{Under \(\SL_4\times \bbG_m\)}

\begin{prop}
Let \(H\) be a rank one horospherical subgroup of \(G=\SL_4\times \bbG_m\), with normalizer the maximal parabolic subgroup \(Q_{\alpha_1}\), satisfying Assumptions~\ref{assumptions}. 
Then \(H=\ker(a_1\varpi_1+\chi_1)\) for some non-negative integer \(a_1\), where \(\chi_1\) is the projection \(G\to \bbG_m\) to the second factor. 
The combinatorial data of \(G/H\) is:
\[\arraycolsep=12pt
\begin{array}{lll}
M=\langle a_1\varpi_1+\chi_1\rangle  & \zeta(\clubsuit) =  \{\alpha_1\}  & \rho(\clubsuit) = a_1 \\
\Sigma = \emptyset &   \\
\mathcal{D} =\{\clubsuit\} &  \\
\midrule
\kappa = 3\alpha_1 + 2\alpha_2 + \alpha_3 & f=\frac{(4+a_1x_1)^3}{6} & m_{\clubsuit} = 4 \\
\end{array}
\]
\end{prop}

We easily deduce the possible locally factorial \(G/H\)-reflexive polytope, and the corresponding four non-product embeddings (we omit the polytope for the product \(\bbP^3\times \bbP^1\), numbered \textbf{4-1-53}, that is when \(a_1=0\)).

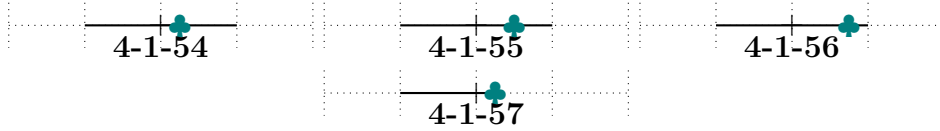
\begin{center}
\begin{tikzpicture}
\draw[dotted] (-2,0) -- (2,0);
\draw (0,0) node{+};
\foreach \i in {0,...,4}
{
\draw[dotted] (-2+\i,.3) -- (-2+\i,-.3);
}
\draw[thick] (-1,0) -- (1,0);
\draw (1/4,0) node[color=teal]{\(\clubsuit\)};
\draw (0,0) node[below]{\textbf{4-1-54}};
\end{tikzpicture}
\begin{tikzpicture}
\draw[dotted] (-2,0) -- (2,0);
\draw (0,0) node{+};
\foreach \i in {0,...,4}
{
\draw[dotted] (-2+\i,.3) -- (-2+\i,-.3);
}
\draw[thick] (-1,0) -- (1,0);
\draw (2/4,0) node[color=teal]{\(\clubsuit\)};
\draw (0,0) node[below]{\textbf{4-1-55}};
\end{tikzpicture}
\begin{tikzpicture}
\draw[dotted] (-2,0) -- (2,0);
\draw (0,0) node{+};
\foreach \i in {0,...,4}
{
\draw[dotted] (-2+\i,.3) -- (-2+\i,-.3);
}
\draw[thick] (-1,0) -- (1,0);
\draw (3/4,0) node[color=teal]{\(\clubsuit\)};
\draw (0,0) node[below]{\textbf{4-1-56}};
\end{tikzpicture}
\begin{tikzpicture}
\draw[dotted] (-2,0) -- (2,0);
\draw (0,0) node{+};
\foreach \i in {0,...,4}
{
\draw[dotted] (-2+\i,.3) -- (-2+\i,-.3);
}
\draw[thick] (-1,0) -- (1/4,0);
\draw (1/4,0) node[color=teal]{\(\clubsuit\)};
\draw (0,0) node[below]{\textbf{4-1-57}};
\end{tikzpicture}
\end{center}

\begin{exa}
The corresponding embeddings are \(\bbP_{\bbP^3}(\mathcal{O}\oplus \mathcal{O}(-k))\) for \(0\leq k\leq 3\), and \(\bbP^4\) as the locally factorial Fano cone over \(\bbP^3\) obtained by contracting the zero section of \(\mathcal{O}(-1)\) in \(\bbP_{\bbP^3}(\mathcal{O}\oplus \mathcal{O}(-1))\).
\end{exa}

\appendix

\section{Rank 1 and 2 spherical actions on Fano threefolds}

In the following table, we provide for each faithful spherical action on a Fano threefold the identifier allowing to find the detailed combinatorial data in the body of the article, followed by several additional useful characteristics. 
We computed the Picard rank and anticanonical degree, determined whether or not the variety admitted Kähler-Einstein metrics, then using this data identified the underlying Fano threefold using the standard Iskovskikh–Prokhorov numbering \cite{Iskovskikh_Prokhorov_1999}. 
We also recall the group acting faithfully (up to finite cover), and the type of spherical homogeneous space. 

\begin{longtable}{CCCCCCC}
\text{Identifier} & \text{Pic} & \text{Degree} & \text{KE?} & \text{Fano threefold n°} & \text{Group} & \text{Type}\\
\toprule
\text{3-1-1} & 1 & 54 & \text{True} & 1.16 & (\SL_2)^2 & \text{symmetric}\\
\text{3-1-2} & 1 & 64 & \text{True} & 1.17 & (\SL_2)^2 & \text{symmetric}\\
\midrule
\text{3-1-3} & 3 & 48 & \text{True} & 3.27 & (\SL_2)^2\times \bbG_m & \text{horospherical}\\
\text{3-1-4} & 3 & 48 & \text{False} & 3.28 & (\SL_2)^2\times \bbG_m & \text{horospherical}\\
\text{3-1-5} & 2 & 54 & \text{True} & 2.34 & (\SL_2)^2\times \bbG_m & \text{horospherical}\\
\text{3-1-6} & 3 & 52 & \text{False} & 3.31 & (\SL_2)^2\times \bbG_m & \text{horospherical}\\
\text{3-1-7} & 3 & 44 & \text{True} & 3.25 & (\SL_2)^2\times \bbG_m & \text{horospherical}\\
\text{3-1-8} & 2 & 54 & \text{False} & 2.33 & (\SL_2)^2\times \bbG_m & \text{horospherical}\\
\text{3-1-9} & 1 & 64 & \text{True} & 1.17 & (\SL_2)^2\times \bbG_m & \text{horospherical}\\
\midrule
\text{3-1-10} & 2 & 54 & \text{True} & 2.34 & \SL_3\times\bbG_m & \text{horospherical}\\
\text{3-1-11} & 2 & 56 & \text{False} & 2.35 & \SL_3\times\bbG_m & \text{horospherical}\\
\text{3-1-12} & 2 & 62 & \text{False} & 2.36 & \SL_3\times\bbG_m & \text{horospherical}\\
\text{3-1-13} & 1 & 64 & \text{True} & 1.17 & \SL_3\times\bbG_m & \text{horospherical}\\
\midrule
\text{3-2-1} & 3 & 48 & \text{True} & 3.27 & \SL_2\times\bbG_m & \text{symmetric}\\
\text{3-2-2} & 3 & 52 & \text{False} & 3.31 & \SL_2\times\bbG_m & \text{symmetric}\\
\text{3-2-3} & 4 & 38 & \text{False} & 4.8 & \SL_2\times\bbG_m & \text{symmetric}\\
\text{3-2-4} & 1 & 54 & \text{True} & 1.16 & \SL_2\times\bbG_m & \text{type T}\\
\text{3-2-5} & 2 & 48 & \text{True} & 2.32 & \SL_2\times\bbG_m & \text{type T}\\
\text{3-2-6} & 2 & 46 & \text{False} & 2.31 & \SL_2\times\bbG_m & \text{type T}\\
\text{3-2-7} & 2 & 54 & \text{True} & 2.34 & \SL_2\times\bbG_m & \text{type T}\\
\text{3-2-8} & 3 & 42 & \text{False} & 3.24 & \SL_2\times\bbG_m & \text{type T}\\
\text{3-2-9} & 3 & 38 & \text{True} & 3.20 & \SL_2\times\bbG_m & \text{type T}\\
\text{3-2-10} & 3 & 48 & \text{False} & 3.28 & \SL_2\times\bbG_m & \text{type T}\\
\text{3-2-11} & 4 & 36 & \text{True} & 4.7 & \SL_2\times\bbG_m & \text{type T}\\
\text{3-2-12} & 1 & 64 & \text{True} & 1.17 & \SL_2\times\bbG_m & \text{type T}\\
\text{3-2-13} & 2 & 54 & \text{False} & 2.33 & \SL_2\times\bbG_m & \text{type T}\\
\text{3-2-14} & 3 & 44 & \text{True} & 3.25 & \SL_2\times\bbG_m & \text{type T}\\
\text{3-2-15} & 2 & 54 & \text{True} & 2.34 & \SL_2\times\bbG_m & \text{symmetric}\\
\text{3-2-16} & 2 & 62 & \text{False} & 2.36 & \SL_2\times\bbG_m & \text{symmetric}\\
\text{3-2-17} & 3 & 40 & \text{False} & 3.22 & \SL_2\times\bbG_m & \text{symmetric}\\
\text{3-2-18} & 1 & 54 & \text{True} & 1.16 & \SL_2\times\bbG_m & \text{symmetric}\\
\text{3-2-19} & 2 & 40 & \text{True} & 2.29 & \SL_2\times\bbG_m & \text{symmetric}\\
\text{3-2-20} & 1 & 64 & \text{True} & 1.17 & \SL_2\times\bbG_m & \text{symmetric}\\
\text{3-2-21} & 2 & 46 & \text{False} & 2.30 & \SL_2\times\bbG_m & \text{symmetric}\\
\text{3-2-22} & 2 & 56 & \text{False} & 2.35 & \SL_2\times\bbG_m & \text{symmetric}\\
\text{3-2-23} & 3 & 38 & \text{True} & 3.19 & \SL_2\times\bbG_m & \text{symmetric}\\
\midrule
\text{3-2-24} & 2 & 54 & \text{True} & 2.34 & \SL_2\times\bbG_m^2 & \text{horospherical}\\
\text{3-2-25} & 3 & 48 & \text{False} & 3.28 & \SL_2\times\bbG_m^2 & \text{horospherical}\\
\text{3-2-26} & 3 & 48 & \text{True} & 3.27 & \SL_2\times\bbG_m^2 & \text{horospherical}\\
\text{3-2-27} & 4 & 42 & \text{False} & 4.10 & \SL_2\times\bbG_m^2 & \text{horospherical}\\
\text{3-2-28} & 5 & 36 & \text{True} & 5.3 & \SL_2\times\bbG_m^2 & \text{horospherical}\\
\text{3-2-29} & 5 & 36 & \text{False} & 5.2 & \SL_2\times\bbG_m^2 & \text{horospherical}\\
\text{3-2-30} & 4 & 40 & \text{False} & 4.9 & \SL_2\times\bbG_m^2 & \text{horospherical}\\
\text{3-2-31} & 4 & 44 & \text{False} & 4.11 & \SL_2\times\bbG_m^2 & \text{horospherical}\\
\text{3-2-32} & 4 & 46 & \text{False} & 4.12 & \SL_2\times\bbG_m^2 & \text{horospherical}\\
\text{3-2-33} & 3 & 50 & \text{False} & 3.30 & \SL_2\times\bbG_m^2 & \text{horospherical}\\
\text{3-2-34} & 3 & 44 & \text{True} & 3.25 & \SL_2\times\bbG_m^2 & \text{horospherical}\\
\text{3-2-35} & 3 & 52 & \text{False} & 3.31 & \SL_2\times\bbG_m^2 & \text{horospherical}\\
\text{3-2-36} & 3 & 48 & \text{False} & 3.28 & \SL_2\times\bbG_m^2 & \text{horospherical}\\
\text{3-2-37} & 2 & 54 & \text{False} & 2.33 & \SL_2\times\bbG_m^2 & \text{horospherical}\\
\text{3-2-38} & 2 & 62 & \text{False} & 2.36 & \SL_2\times\bbG_m^2 & \text{horospherical}\\
\text{3-2-39} & 3 & 50 & \text{False} & 3.29 & \SL_2\times\bbG_m^2 & \text{horospherical}\\
\text{3-2-40} & 2 & 56 & \text{False} & 2.35 & \SL_2\times\bbG_m^2 & \text{horospherical}\\
\text{3-2-41} & 3 & 46 & \text{False} & 3.26 & \SL_2\times\bbG_m^2 & \text{horospherical}\\
\text{3-2-42} & 2 & 54 & \text{True} & 2.34 & \SL_2\times\bbG_m^2 & \text{horospherical}\\
\text{3-2-43} & 2 & 54 & \text{False} & 2.33 & \SL_2\times\bbG_m^2 & \text{horospherical}\\
\text{3-2-44} & 1 & 64 & \text{True} & 1.17 & \SL_2\times\bbG_m^2 & \text{horospherical}\\
\bottomrule
\end{longtable}

\section{Rank 1 and 2 spherical actions on locally factorial Fano fourfolds}

As in the previous table, we gather additional useful data on the fourfolds for which we obtained a spherical action. 

\begin{longtable}{CCCCCCC}
\text{Identifier} & \text{Pic} & \text{Degree} & \text{KE?}  & \text{Group} & \text{Type}\\
\toprule
\text{4-1-1} & 2 & 486 & \text{True}  & \SL_3 & \text{symmetric}\\
\text{4-1-2} & 2 & 486 & \text{True}  & \SL_3 & \text{horosymmetric}\\
\text{4-1-3} & 3 & 336 & \text{True}  & \SL_3 & \text{horosymmetric}\\
\midrule
\text{4-1-4} & 3 & 384 & \text{True}  & \SL_3\times\bbG_m & \text{horospherical}\\
\text{4-1-5} & 3 & 480 & \text{False}  & \SL_3\times\bbG_m & \text{horospherical}\\
\text{4-1-6} & 3 & 400 & \text{False}  & \SL_3\times\bbG_m & \text{horospherical}\\
\text{4-1-7} & 2 & 432 & \text{False}  & \SL_3\times\bbG_m & \text{horospherical}\\
\text{4-1-8} & 3 & 352 & \text{True}  & \SL_3\times\bbG_m & \text{horospherical}\\
\text{4-1-9} & 2 & 432 & \text{False}  & \SL_3\times\bbG_m & \text{horospherical}\\
\text{4-1-10} & 1 & 512 & \text{True}  & \SL_3\times\bbG_m & \text{horospherical}\\
\midrule
\text{4-1-11} & 1 & 625 & \text{True}  & \Sp_4 & \text{symmetric}\\
\text{4-1-12} & 1 & 512 & \text{True}  & \Sp_4 & \text{symmetric}\\
\midrule
\text{4-1-13} & 3 & 432 & \text{True}  & \SL_3\times\SL_2 & \text{horosymmetric}\\
\text{4-1-14} & 2 & 486 & \text{True}  & \SL_3\times\SL_2 & \text{horosymmetric}\\
\midrule
\text{4-1-15} & 2 & 432 & \text{True}  & (\SL_2)^3 & \text{horosymmetric}\\
\text{4-1-16} & 2 & 512 & \text{True}  & (\SL_2)^3 & \text{horosymmetric}\\
\text{4-1-17} & 4 & 384 & \text{True}  & (\SL_2)^3 & \text{horosymmetric}\\
\text{4-1-18} & 3 & 432 & \text{True}  & (\SL_2)^3 & \text{horosymmetric}\\
\midrule
\text{4-1-19} & 4 & 384 & \text{True}  & (\SL_2)^3\times\bbG_m & \text{horospherical}\\
\text{4-1-20} & 4 & 384 & \text{False} & (\SL_2)^3\times\bbG_m & \text{horospherical}\\
\text{4-1-21} & 3 & 432 & \text{True}  & (\SL_2)^3\times\bbG_m & \text{horospherical}\\
\text{4-1-22} & 4 & 416 & \text{False}  & (\SL_2)^3\times\bbG_m & \text{horospherical}\\
\text{4-1-23} & 4 & 352 & \text{True}  & (\SL_2)^3\times\bbG_m & \text{horospherical}\\
\text{4-1-24} & 3 & 432 & \text{False}  & (\SL_2)^3\times\bbG_m & \text{horospherical}\\
\text{4-1-25} & 2 & 512 & \text{True}  & (\SL_2)^3\times\bbG_m & \text{horospherical}\\
\text{4-1-26} & 4 & 480 & \text{False}  & (\SL_2)^3\times\bbG_m & \text{horospherical}\\
\text{4-1-27} & 4 & 352 & \text{False}  & (\SL_2)^3\times\bbG_m & \text{horospherical}\\
\text{4-1-28} & 3 & 486 & \text{False}  & (\SL_2)^3\times\bbG_m & \text{horospherical}\\
\midrule
\text{4-1-29} & 3 & 432 & \text{True}  & \SL_3\times\SL_2\times\bbG_m & \text{horospherical}\\
\text{4-1-30} & 3 & 432 & \text{False}  & \SL_3\times\SL_2\times\bbG_m & \text{horospherical}\\
\text{4-1-31} & 2 & 486 & \text{True}  & \SL_3\times\SL_2\times\bbG_m & \text{horospherical}\\
\text{4-1-32} & 3 & 448 & \text{False}  & \SL_3\times\SL_2\times\bbG_m & \text{horospherical}\\
\text{4-1-33} & 3 & 496 & \text{False}  & \SL_3\times\SL_2\times\bbG_m & \text{horospherical}\\
\text{4-1-34} & 2 & 512 & \text{True}  & \SL_3\times\SL_2\times\bbG_m & \text{horospherical}\\
\text{4-1-35} & 3 & 496 & \text{False}  & \SL_3\times\SL_2\times\bbG_m & \text{horospherical}\\
\text{4-1-36} & 3 & 400 & \text{False}  & \SL_3\times\SL_2\times\bbG_m & \text{horospherical}\\
\text{4-1-37} & 3 & 592 & \text{False}  & \SL_3\times\SL_2\times\bbG_m & \text{horospherical}\\
\text{4-1-38} & 3 & 400 & \text{False}  & \SL_3\times\SL_2\times\bbG_m & \text{horospherical}\\
\text{4-1-39} & 2 & 512 & \text{False}  & \SL_3\times\SL_2\times\bbG_m & \text{horospherical}\\
\text{4-1-40} & 2 & 513 & \text{False}  & \SL_3\times\SL_2\times\bbG_m & \text{horospherical}\\
\text{4-1-41} & 2 & 594 & \text{False}  & \SL_3\times\SL_2\times\bbG_m & \text{horospherical}\\
\text{4-1-42} & 2 & 513 & \text{False} & \SL_3\times\SL_2\times\bbG_m & \text{horospherical}\\
\text{4-1-43} & 1 & 625 & \text{True} &  \SL_3\times\SL_2\times\bbG_m & \text{horospherical}\\
\midrule
\text{4-1-44} & 2 & 512 & \text{True}  & \Sp_4\times\bbG_m & \text{horospherical}\\
\text{4-1-45} & 2 & 432 & \text{True} & \Sp_4\times\bbG_m & \text{horospherical}\\
\text{4-1-46} & 2 & 544 & \text{False}  & \Sp_4\times\bbG_m & \text{horospherical}\\
\text{4-1-47} & 2 & 640 & \text{False}  & \Sp_4\times\bbG_m & \text{horospherical}\\
\text{4-1-48} & 2 & 800 & \text{False}  & \Sp_4\times\bbG_m & \text{horospherical}\\
\text{4-1-49} & 1 & 625 & \text{True}  & \Sp_4\times\bbG_m & \text{horospherical}\\
\text{4-1-50} & 2 & 480 & \text{False}  & \Sp_4\times\bbG_m & \text{horospherical}\\
\text{4-1-51} & 2 & 624 & \text{False}  & \Sp_4\times\bbG_m & \text{horospherical}\\
\text{4-1-52} & 1 & 512 & \text{False}  & \Sp_4\times\bbG_m & \text{horospherical}\\
\midrule
\text{4-1-53} & 2 & 512 & \text{True}   & \SL_4\times\bbG_m & \text{horospherical}\\
\text{4-1-54} & 2 & 544 & \text{False}  & \SL_4\times\bbG_m & \text{horospherical}\\
\text{4-1-55} & 2 & 640 & \text{False}  & \SL_4\times\bbG_m & \text{horospherical}\\
\text{4-1-56} & 2 & 800 & \text{False}  & \SL_4\times\bbG_m & \text{horospherical}\\
\text{4-1-57} & 1 & 625 & \text{True}  & \SL_4\times\bbG_m & \text{horospherical}\\
\midrule
\text{4-2-1} & 3 & 430 & \text{False}  & (\SL_2^2)\times\bbG_m & \text{symmetric}\\
\text{4-2-2} & 2 & 624 & \text{False}  & (\SL_2^2)\times\bbG_m & \text{symmetric}\\
\text{4-2-3} & 3 & 346 & \text{False}  & (\SL_2^2)\times\bbG_m & \text{symmetric}\\
\text{4-2-4} & 2 & 480 & \text{False}  & (\SL_2^2)\times\bbG_m & \text{symmetric}\\
\text{4-2-5} & 2 & 432 & \text{True} & (\SL_2^2)\times\bbG_m & \text{symmetric}\\
\text{4-2-6} & 1 & 512 & \text{False}   & (\SL_2^2)\times\bbG_m & \text{symmetric}\\
\text{4-2-7} & 2 & 378 & \text{False}  & (\SL_2^2)\times\bbG_m & \text{symmetric}\\
\text{4-2-8} & 2 & 512 & \text{True} & (\SL_2^2)\times\bbG_m & \text{symmetric}\\
\text{4-2-9} & 2 & 640 & \text{False}  & (\SL_2^2)\times\bbG_m & \text{symmetric}\\
\text{4-2-10} & 3 & 376 & \text{False}  & (\SL_2^2)\times\bbG_m & \text{symmetric}\\
\text{4-2-11} & 3 & 454 & \text{False}  & (\SL_2^2)\times\bbG_m & \text{symmetric}\\
\text{4-2-12} & 2 & 544 & \text{False}  & (\SL_2^2)\times\bbG_m & \text{symmetric}\\
\text{4-2-13} & 3 & 350 & \text{True}  & (\SL_2^2)\times\bbG_m & \text{symmetric}\\
\text{4-2-14} & 2 & 800 & \text{False}  & (\SL_2^2)\times\bbG_m & \text{symmetric}\\
\text{4-2-15} & 2 & 378 & \text{True}  & (\SL_2^2)\times\bbG_m & \text{symmetric}\\
\text{4-2-16} & 1 & 512 & \text{True}  & (\SL_2^2)\times\bbG_m & \text{symmetric}\\
\text{4-2-17} & 2 & 431 & \text{False}  & (\SL_2^2)\times\bbG_m & \text{symmetric}\\
\text{4-2-18} & 1 & 625 & \text{True}  & (\SL_2^2)\times\bbG_m & \text{symmetric}\\
\midrule
\text{4-2-19} & 2 & 432 & \text{True} & (\SL_2^2) & \text{solvable}\\
\text{4-2-20} & 2 & 384 & \text{True}  & (\SL_2^2) & \text{solvable}\\
\text{4-2-21} & 3 & 304 & \text{True}  & (\SL_2^2) & \text{solvable}\\
\text{4-2-22} & 3 & 324 & \text{True}  & (\SL_2^2) & \text{solvable}\\
\text{4-2-23} & 2 & 512 & \text{True}  & (\SL_2^2) & \text{solvable}\\
\text{4-2-24} & 4 & 384 & \text{True}  & (\SL_2^2) & \text{symmetric}\\
\text{4-2-25} & 5 & 252 & \text{True} & (\SL_2^2) & \text{symmetric}\\
\text{4-2-26} & 3 & 432 & \text{True}  & (\SL_2^2) & \text{symmetric}\\
\text{4-2-27} & 4 & 248 & \text{True}  & (\SL_2^2) & \text{symmetric}\\
\text{4-2-28} & 2 & 486 & \text{True}  & (\SL_2^2) & \text{symmetric}\\
\text{4-2-29} & 3 & 230 & \text{True}  & (\SL_2^2) & \text{symmetric}\\
\text{4-2-30} & 2 & 352 & \text{True}  & (\SL_2^2) & \text{symmetric}\\
\text{4-2-31} & 1 & 512 & \text{True}  & (\SL_2^2) & \text{symmetric}\\
\midrule
\text{4-2-32} & 4 & 384 & \text{True}  & (\SL_2^2)\times\bbG_m & \text{horosymmetric}\\
\text{4-2-33} & 4 & 416 & \text{False}  & (\SL_2^2)\times\bbG_m & \text{horosymmetric}\\
\text{4-2-34} & 5 & 304 & \text{False}  & (\SL_2^2)\times\bbG_m & \text{horosymmetric}\\
\text{4-2-35} & 2 & 432 & \text{True}  & (\SL_2^2)\times\bbG_m & \text{ind. type T}\\
\text{4-2-36} & 3 & 384 & \text{True}  & (\SL_2^2)\times\bbG_m & \text{ind. type T} \\
\text{4-2-37} & 3 & 368 & \text{False}  & (\SL_2^2)\times\bbG_m & \text{ind. type T} \\
\text{4-2-38} & 3 & 432 & \text{True}  & (\SL_2^2)\times\bbG_m & \text{ind. type T} \\
\text{4-2-39} & 4 & 336 & \text{False}  & (\SL_2^2)\times\bbG_m & \text{ind. type T} \\
\text{4-2-40} & 4 & 304 & \text{True}  & (\SL_2^2)\times\bbG_m & \text{ind. type T} \\
\text{4-2-41} & 4 & 384 & \text{False}  & (\SL_2^2)\times\bbG_m & \text{ind. type T} \\
\text{4-2-42} & 5 & 288 & \text{True}  & (\SL_2^2)\times\bbG_m & \text{ind. type T} \\
\text{4-2-43} & 2 & 384 & \text{True}  & (\SL_2^2)\times\bbG_m & \text{ind. type T} \\
\text{4-2-44} & 3 & 432 & \text{False}  & (\SL_2^2)\times\bbG_m & \text{ind. type T} \\
\text{4-2-45} & 4 & 352 & \text{True}  & (\SL_2^2)\times\bbG_m & \text{ind. type T} \\
\text{4-2-46} & 5 & 330 & \text{False}  & (\SL_2^2)\times\bbG_m & \text{horosymmetric}\\
\text{4-2-47} & 4 & 480 & \text{False}  & (\SL_2^2)\times\bbG_m & \text{horosymmetric}\\
\text{4-2-48} & 4 & 384 & \text{False}  & (\SL_2^2)\times\bbG_m & \text{horosymmetric}\\
\text{4-2-49} & 5 & 278 & \text{False}  & (\SL_2^2)\times\bbG_m & \text{horosymmetric}\\
\text{4-2-50} & 4 & 352 & \text{False}  & (\SL_2^2)\times\bbG_m & \text{horosymmetric}\\
\text{4-2-51} & 3 & 486 & \text{False}  & (\SL_2^2)\times\bbG_m & \text{horosymmetric}\\
\text{4-2-52} & 3 & 432 & \text{True} & (\SL_2^2)\times\bbG_m & \text{horosymmetric}\\
\text{4-2-53} & 4 & 326 & \text{False}  & (\SL_2^2)\times\bbG_m & \text{horosymmetric}\\
\text{4-2-54} & 4 & 288 & \text{False}  & (\SL_2^2)\times\bbG_m &  \text{ind. type T}\\
\text{4-2-55} & 4 & 368 & \text{False}  & (\SL_2^2)\times\bbG_m & \text{ind. type T} \\
\text{4-2-56} & 5 & 288 & \text{False}  & (\SL_2^2)\times\bbG_m & \text{ind. type T} \\
\text{4-2-57} & 4 & 416 & \text{False}  & (\SL_2^2)\times\bbG_m & \text{ind. type T} \\
\text{4-2-58} & 3 & 432 & \text{False}  & (\SL_2^2)\times\bbG_m & \text{ind. type T} \\
\text{4-2-59} & 2 & 512 & \text{True}  & (\SL_2^2)\times\bbG_m & \text{ind. type T} \\
\text{4-2-60} & 3 & 432 & \text{False}  & (\SL_2^2)\times\bbG_m & \text{ind. type T} \\
\text{4-2-61} & 3 & 486 & \text{False}  & (\SL_2^2)\times\bbG_m & \text{ind. type T} \\
\text{4-2-62} & 4 & 352 & \text{False}  & (\SL_2^2)\times\bbG_m & \text{ind. type T} \\
\text{4-2-63} & 3 & 432 & \text{True} & (\SL_2^2)\times\bbG_m & \text{horosymmetric}\\
\text{4-2-64} & 3 & 496 & \text{False}  & (\SL_2^2)\times\bbG_m & \text{horosymmetric}\\
\text{4-2-65} & 4 & 320 & \text{False}  & (\SL_2^2)\times\bbG_m & \text{horosymmetric}\\
\text{4-2-66} & 2 & 432 & \text{True}  & (\SL_2^2)\times\bbG_m & \text{horosymmetric}\\
\text{4-2-67} & 3 & 320 & \text{True}  & (\SL_2^2)\times\bbG_m & \text{horosymmetric}\\
\text{4-2-68} & 2 & 512 & \text{True}  & (\SL_2^2)\times\bbG_m & \text{horosymmetric}\\
\text{4-2-69} & 3 & 368 & \text{False}  & (\SL_2^2)\times\bbG_m & \text{horosymmetric}\\
\text{4-2-70} & 3 & 448 & \text{False}  & (\SL_2^2)\times\bbG_m & \text{horosymmetric}\\
\text{4-2-71} & 4 & 304 & \text{True}  & (\SL_2^2)\times\bbG_m & \text{horosymmetric}\\
\text{4-2-72} & 3 & 496 & \text{False}  & (\SL_2^2)\times\bbG_m & \text{horosymmetric}\\
\text{4-2-73} & 4 & 304 & \text{False} & (\SL_2^2)\times\bbG_m & \text{horosymmetric}\\
\text{4-2-74} & 3 & 416 & \text{False}  & (\SL_2^2)\times\bbG_m & \text{horosymmetric}\\
\text{4-2-75} & 2 & 512 & \text{False}  & (\SL_2^2)\times\bbG_m & \text{horosymmetric}\\
\text{4-2-76} & 3 & 400 & \text{False}  & (\SL_2^2)\times\bbG_m & \text{horosymmetric}\\
\text{4-2-77} & 2 & 513 & \text{False}  & (\SL_2^2)\times\bbG_m & \text{horosymmetric}\\
\text{4-2-78} & 2 & 433 & \text{False}  & (\SL_2^2)\times\bbG_m & \text{horosymmetric}\\
\text{4-2-79} & 3 & 321 & \text{False}  & (\SL_2^2)\times\bbG_m & \text{horosymmetric}\\
\text{4-2-80} & 3 & 321 & \text{False}  & (\SL_2^2)\times\bbG_m & \text{horosymmetric}\\
\text{4-2-81} & 1 & 625 & \text{True} & (\SL_2^2)\times\bbG_m & \text{horosymmetric}\\
\text{4-2-82} & 2 & 433 & \text{False}  & (\SL_2^2)\times\bbG_m & \text{horosymmetric}\\
\text{4-2-83} & 4 & 284 & \text{False}  & (\SL_2^2)\times\bbG_m & \text{horosymmetric}\\
\text{4-2-84} & 3 & 432 & \text{False}  & (\SL_2^2)\times\bbG_m & \text{horosymmetric}\\
\text{4-2-85} & 4 & 356 & \text{False}  & (\SL_2^2)\times\bbG_m & \text{horosymmetric}\\
\text{4-2-86} & 3 & 592 & \text{False}  & (\SL_2^2)\times\bbG_m & \text{horosymmetric}\\
\text{4-2-87} & 3 & 400 & \text{False}  & (\SL_2^2)\times\bbG_m & \text{horosymmetric}\\
\text{4-2-88} & 3 & 338 & \text{False} & (\SL_2^2)\times\bbG_m & \text{horosymmetric}\\
\text{4-2-89} & 2 & 550 & \text{False}  & (\SL_2^2)\times\bbG_m & \text{horosymmetric}\\
\text{4-2-90} & 2 & 486 & \text{True} & (\SL_2^2)\times\bbG_m & \text{horosymmetric}\\
\midrule
\text{4-2-91} & 3 & 432 & \text{True}  & (\SL_2^2)\times\bbG_m^2 & \text{horospherical}\\
\text{4-2-92} & 4 & 384 & \text{False}  & (\SL_2^2)\times\bbG_m^2 & \text{horospherical}\\
\text{4-2-93} & 4 & 384 & \text{True}  & (\SL_2^2)\times\bbG_m^2 & \text{horospherical}\\
\text{4-2-94} & 5 & 336 & \text{False}  & (\SL_2^2)\times\bbG_m^2 & \text{horospherical}\\
\text{4-2-95} & 6 & 288 & \text{True}  & (\SL_2^2)\times\bbG_m^2 & \text{horospherical}\\
\text{4-2-96} & 6 & 288 & \text{False}  & (\SL_2^2)\times\bbG_m^2 & \text{horospherical}\\
\text{4-2-97} & 5 & 320 & \text{False}  & (\SL_2^2)\times\bbG_m^2 & \text{horospherical}\\
\text{4-2-98} & 5 & 352 & \text{False}  & (\SL_2^2)\times\bbG_m^2 & \text{horospherical}\\
\text{4-2-99} & 5 & 368 & \text{False}  & (\SL_2^2)\times\bbG_m^2 & \text{horospherical}\\
\text{4-2-100} & 4 & 400 & \text{False}  & (\SL_2^2)\times\bbG_m^2 & \text{horospherical}\\
\text{4-2-101} & 4 & 352 & \text{True}  & (\SL_2^2)\times\bbG_m^2 & \text{horospherical}\\
\text{4-2-102} & 4 & 416 & \text{False}  & (\SL_2^2)\times\bbG_m^2 & \text{horospherical}\\
\text{4-2-103} & 4 & 384 & \text{False}  & (\SL_2^2)\times\bbG_m^2 & \text{horospherical}\\
\text{4-2-104} & 3 & 432 & \text{False}  & (\SL_2^2)\times\bbG_m^2 & \text{horospherical}\\
\text{4-2-105} & 3 & 496 & \text{False}  & (\SL_2^2)\times\bbG_m^2 & \text{horospherical}\\
\text{4-2-106} & 4 & 400 & \text{False}  & (\SL_2^2)\times\bbG_m^2 & \text{horospherical}\\
\text{4-2-107} & 3 & 448 & \text{False}  & (\SL_2^2)\times\bbG_m^2 & \text{horospherical}\\
\text{4-2-108} & 4 & 368 & \text{False}  & (\SL_2^2)\times\bbG_m^2 & \text{horospherical}\\
\text{4-2-109} & 3 & 432 & \text{True}  & (\SL_2^2)\times\bbG_m^2 & \text{horospherical}\\
\text{4-2-110} & 3 & 432 & \text{False}  & (\SL_2^2)\times\bbG_m^2 & \text{horospherical}\\
\text{4-2-111} & 2 & 512 & \text{True}  & (\SL_2^2)\times\bbG_m^2 & \text{horospherical}\\
\text{4-2-112} & 3 & 405 & \text{False}  & (\SL_2^2)\times\bbG_m^2 & \text{horospherical}\\
\text{4-2-113} & 4 & 384 & \text{False}  & (\SL_2^2)\times\bbG_m^2 & \text{horospherical}\\
\text{4-2-114} & 4 & 352 & \text{False}  & (\SL_2^2)\times\bbG_m^2 & \text{horospherical}\\
\text{4-2-115} & 4 & 448 & \text{False}  & (\SL_2^2)\times\bbG_m^2 & \text{horospherical}\\
\text{4-2-116} & 4 & 384 & \text{False}  & (\SL_2^2)\times\bbG_m^2 & \text{horospherical}\\
\text{4-2-117} & 5 & 299 & \text{False}  & (\SL_2^2)\times\bbG_m^2 & \text{horospherical}\\
\text{4-2-118} & 5 & 405 & \text{False}  & (\SL_2^2)\times\bbG_m^2 & \text{horospherical}\\
\text{4-2-119} & 5 & 341 & \text{False}  & (\SL_2^2)\times\bbG_m^2 & \text{horospherical}\\
\text{4-2-120} & 5 & 331 & \text{False}  & (\SL_2^2)\times\bbG_m^2 & \text{horospherical}\\
\text{4-2-121} & 5 & 363 & \text{False}  & (\SL_2^2)\times\bbG_m^2 & \text{horospherical}\\
\text{4-2-122} & 6 & 298 & \text{False}  & (\SL_2^2)\times\bbG_m^2 & \text{horospherical}\\
\text{4-2-123} & 6 & 298 & \text{False} & (\SL_2^2)\times\bbG_m^2 & \text{horospherical}\\
\text{4-2-124} & 3 & 592 & \text{False}  & (\SL_2^2)\times\bbG_m^2 & \text{horospherical}\\
\text{4-2-125} & 3 & 560 & \text{False}  & (\SL_2^2)\times\bbG_m^2 & \text{horospherical}\\
\text{4-2-126} & 3 & 400 & \text{False}  & (\SL_2^2)\times\bbG_m^2 & \text{horospherical}\\
\text{4-2-127} & 4 & 463 & \text{False}  & (\SL_2^2)\times\bbG_m^2 & \text{horospherical}\\
\text{4-2-128} & 4 & 433 & \text{False}  & (\SL_2^2)\times\bbG_m^2 & \text{horospherical}\\
\text{4-2-129} & 4 & 496 & \text{False}  & (\SL_2^2)\times\bbG_m^2 & \text{horospherical}\\
\text{4-2-130} & 4 & 337 & \text{False}  & (\SL_2^2)\times\bbG_m^2 & \text{horospherical}\\
\text{4-2-131} & 3 & 496 & \text{False}  & (\SL_2^2)\times\bbG_m^2 & \text{horospherical}\\
\text{4-2-132} & 3 & 400 & \text{False}  & (\SL_2^2)\times\bbG_m^2 & \text{horospherical}\\
\text{4-2-133} & 3 & 464 & \text{False}  & (\SL_2^2)\times\bbG_m^2 & \text{horospherical}\\
\text{4-2-134} & 4 & 389 & \text{False}  & (\SL_2^2)\times\bbG_m^2 & \text{horospherical}\\
\text{4-2-135} & 4 & 411 & \text{False} & (\SL_2^2)\times\bbG_m^2 & \text{horospherical}\\
\text{4-2-136} & 4 & 347 & \text{False}  & (\SL_2^2)\times\bbG_m^2 & \text{horospherical}\\
\text{4-2-137} & 3 & 432 & \text{False}  & (\SL_2^2)\times\bbG_m^2 & \text{horospherical}\\
\text{4-2-138} & 3 & 459 & \text{False}  & (\SL_2^2)\times\bbG_m^2 & \text{horospherical}\\
\text{4-2-139} & 2 & 512 & \text{False}  & (\SL_2^2)\times\bbG_m^2 & \text{horospherical}\\
\text{4-2-140} & 2 & 544 & \text{False}  & (\SL_2^2)\times\bbG_m^2 & \text{horospherical}\\
\text{4-2-141} & 3 & 529 & \text{False}  & (\SL_2^2)\times\bbG_m^2 & \text{horospherical}\\
\text{4-2-142} & 4 & 368 & \text{False}  & (\SL_2^2)\times\bbG_m^2 & \text{horospherical}\\
\text{4-2-143} & 3 & 433 & \text{False}  & (\SL_2^2)\times\bbG_m^2 & \text{horospherical}\\
\text{4-2-144} & 4 & 326 & \text{False}  & (\SL_2^2)\times\bbG_m^2 & \text{horospherical}\\
\text{4-2-145} & 3 & 406 & \text{False}  & (\SL_2^2)\times\bbG_m^2 & \text{horospherical}\\
\text{4-2-146} & 2 & 513 & \text{False}  & (\SL_2^2)\times\bbG_m^2 & \text{horospherical}\\
\text{4-2-147} & 2 & 594 & \text{False}  & (\SL_2^2)\times\bbG_m^2 & \text{horospherical}\\
\text{4-2-148} & 3 & 450 & \text{False} & (\SL_2^2)\times\bbG_m^2 & \text{horospherical}\\
\text{4-2-149} & 2 & 486 & \text{True} & (\SL_2^2)\times\bbG_m^2 & \text{horospherical}\\
\text{4-2-150} & 2 & 513 & \text{False}  & (\SL_2^2)\times\bbG_m^2 & \text{horospherical}\\
\text{4-2-151} & 3 & 401 & \text{False}  & (\SL_2^2)\times\bbG_m^2 & \text{horospherical}\\
\text{4-2-152} & 1 & 625 & \text{True}  & (\SL_2^2)\times\bbG_m^2 & \text{horospherical}\\
\text{4-2-153} & 2 & 800 & \text{False}  & (\SL_2^2)\times\bbG_m^2 & \text{horospherical}\\
\text{4-2-154} & 2 & 544 & \text{False}  & (\SL_2^2)\times\bbG_m^2 & \text{horospherical}\\
\midrule
\text{4-2-155} & 2 & 486 & \text{True}  & \SL_3\times\bbG_m^2 & \text{horospherical}\\
\text{4-2-156} & 3 & 432 & \text{False} & \SL_3\times\bbG_m^2 & \text{horospherical}\\
\text{4-2-157} & 3 & 432 & \text{True}  & \SL_3\times\bbG_m^2 & \text{horospherical}\\
\text{4-2-158} & 4 & 378 & \text{False}  & \SL_3\times\bbG_m^2 & \text{horospherical}\\
\text{4-2-159} & 5 & 324 & \text{True}  & \SL_3\times\bbG_m^2 & \text{horospherical}\\
\text{4-2-160} & 2 & 594 & \text{False}  & \SL_3\times\bbG_m^2 & \text{horospherical}\\
\text{4-2-161} & 2 & 513 & \text{False}  & \SL_3\times\bbG_m^2 & \text{horospherical}\\
\text{4-2-162} & 2 & 513 & \text{False}  & \SL_3\times\bbG_m^2 & \text{horospherical}\\
\text{4-2-163} & 3 & 496 & \text{False}  & \SL_3\times\bbG_m^2 & \text{horospherical}\\
\text{4-2-164} & 3 & 448 & \text{False}  & \SL_3\times\bbG_m^2 & \text{horospherical}\\
\text{4-2-165} & 3 & 464 & \text{False}  & \SL_3\times\bbG_m^2 & \text{horospherical}\\
\text{4-2-166} & 3 & 576 & \text{False}  & \SL_3\times\bbG_m^2 & \text{horospherical}\\
\text{4-2-167} & 3 & 480 & \text{False}  & \SL_3\times\bbG_m^2 & \text{horospherical}\\
\text{4-2-168} & 3 & 592 & \text{False}  & \SL_3\times\bbG_m^2 & \text{horospherical}\\
\text{4-2-169} & 3 & 496 & \text{False}  & \SL_3\times\bbG_m^2 & \text{horospherical}\\
\text{4-2-170} & 3 & 400 & \text{False}  & \SL_3\times\bbG_m^2 & \text{horospherical}\\
\text{4-2-171} & 3 & 400 & \text{False}  & \SL_3\times\bbG_m^2 & \text{horospherical}\\
\text{4-2-172} & 3 & 560 & \text{False}  & \SL_3\times\bbG_m^2 & \text{horospherical}\\
\text{4-2-173} & 3 & 432 & \text{False}  & \SL_3\times\bbG_m^2 & \text{horospherical}\\
\text{4-2-174} & 4 & 478 & \text{False} & \SL_3\times\bbG_m^2 & \text{horospherical}\\
\text{4-2-175} & 4 & 415 & \text{False}  & \SL_3\times\bbG_m^2 & \text{horospherical}\\
\text{4-2-176} & 4 & 558 & \text{False}  & \SL_3\times\bbG_m^2 & \text{horospherical}\\
\text{4-2-177} & 4 & 447 & \text{False}  & \SL_3\times\bbG_m^2 & \text{horospherical}\\
\text{4-2-178} & 4 & 382 & \text{False}  & \SL_3\times\bbG_m^2 & \text{horospherical}\\
\text{4-2-179} & 4 & 367 & \text{False}  & \SL_3\times\bbG_m^2 & \text{horospherical}\\
\text{4-2-180} & 4 & 351 & \text{False} & \SL_3\times\bbG_m^2 & \text{horospherical}\\
\text{4-2-181} & 4 & 409 & \text{False}  & \SL_3\times\bbG_m^2 & \text{horospherical}\\
\text{4-2-182} & 4 & 505 & \text{False}  & \SL_3\times\bbG_m^2 & \text{horospherical}\\
\text{4-2-183} & 5 & 364 & \text{False}  & \SL_3\times\bbG_m^2 & \text{horospherical}\\
\text{4-2-184} & 5 & 334 & \text{False}  & \SL_3\times\bbG_m^2 & \text{horospherical}\\
\text{4-2-185} & 5 & 354 & \text{False}  & \SL_3\times\bbG_m^2 & \text{horospherical}\\
\text{4-2-186} & 2 & 800 & \text{False}  & \SL_3\times\bbG_m^2 & \text{horospherical}\\
\text{4-2-187} & 3 & 605 & \text{False}  & \SL_3\times\bbG_m^2 & \text{horospherical}\\
\text{4-2-188} & 2 & 640 & \text{False} & \SL_3\times\bbG_m^2 & \text{horospherical}\\
\text{4-2-189} & 3 & 489 & \text{False}  & \SL_3\times\bbG_m^2 & \text{horospherical}\\
\text{4-2-190} & 2 & 544 & \text{False}  & \SL_3\times\bbG_m^2 & \text{horospherical}\\
\text{4-2-191} & 3 & 431 & \text{False}  & \SL_3\times\bbG_m^2 & \text{horospherical}\\
\text{4-2-192} & 2 & 512 & \text{True}  & \SL_3\times\bbG_m^2 & \text{horospherical}\\
\text{4-2-193} & 1 & 625 & \text{True}  & \SL_3\times\bbG_m^2 & \text{horospherical}\\
\text{4-2-194} & 2 & 512 & \text{False}  & \SL_3\times\bbG_m^2 & \text{horospherical}\\
\end{longtable}

\bibliographystyle{alpha}
\bibliography{spherical4folds}
\end{document}